\documentclass[11pt,a4paper, envcountsame,mathserif]{amsart}
\usepackage[usenames,dvipsnames]{color}

\usepackage[utf8]{inputenc}

 \usepackage[colorlinks,citecolor=blue,urlcolor=blue, linkcolor=blue, backref]{hyperref}

\usepackage[capitalise]{cleveref}		% Use \cref without environment name

 \usepackage{amsmath, amssymb, xspace}
 \usepackage{graphicx}

 \usepackage[all]{xy}

\usepackage{tikz}
\usetikzlibrary{arrows,snakes,positioning,backgrounds,shadows}

\newtheorem{theorem}{Theorem}[section]

\newtheorem{thm}[theorem]{Theorem}

\newtheorem{fact}[theorem]{Fact}

\newtheorem{example}[theorem]{Example}

\newtheorem{proposition}[theorem]{Proposition}

\newtheorem{prop}[theorem]{Proposition}

\newtheorem{claim}[theorem]{Claim}

\newtheorem{lemma}[theorem]{Lemma}		

\newtheorem{corollary}[theorem]{Corollary}

\newtheorem{cor}[theorem]{Corollary}

\newtheorem{question}[theorem]{Question}

\theoremstyle{definition}
\newtheorem{definition}[theorem]{Definition}

\newtheorem{remark}[theorem]{Remark}

\newcommand{\KO}{\mathcal{O}}
\newcommand{\DII}{\Delta^0_2}
\newcommand{\NN}{{\mathbb{N}}}
\newcommand{\RR}{{\mathbb{R}}}

\newcommand{\QQ}{{\mathbb{Q}}}
\newcommand{\ZZ}{{\mathbb{Z}}}
\newcommand{\sub}{\subseteq}
\newcommand{\sN}[1]{_{#1\in \NN}}
\newcommand{\uhr}[1]{\! \upharpoonright_{#1}}

\newcommand{\SI}[1]{\Sigma^0_{#1}}
\newcommand{\PI}[1]{\Pi^0_{#1}}
\newcommand{\PPI}{\PI{1}}

\newcommand{\bi}{\begin{itemize}}
\newcommand{\ei}{\end{itemize}}
\newcommand{\bc}{\begin{center}}
\newcommand{\ec}{\end{center}}

\newcommand{\Halt}{{\ES'}}
\newcommand{\ES}{\emptyset}
\newcommand{\estring}{\la \ra}

\newcommand{\tp}[1]{2^{#1}}
\newcommand{\ex}{\exists}
\newcommand{\fa}{\forall}

\newcommand{\la}{\langle}
\newcommand{\ra}{\rangle}

\newcommand{\seqcantor}{2^{ \NN}}

\newcommand{\cantor}{\seqcantor}

\newcommand{\leT}{\le_{\mathrm{T}}}

\newcommand{\MLR}{\mbox{\rm \textsf{MLR}}}
\newcommand{\WtwoR}{\mbox{\rm \textsf{W2R}}}
\newcommand{\twoR}{\mbox{\rm \textsf{2R}}}

\newcommand{\n}{\noindent}

\newcommand{\vsps}{\vspace{3pt}}

\newcommand{\vsp}{\vspace{6pt}}
\newcommand{\leb}{\mathbf{\lambda}}

\newcommand{\sss}{\sigma}

\newcommand{\w}{\omega}
\newcommand{\s}{\sigma}

\newcommand{\lland}{\, \land \, }

\newcommand \seq[1]{{\left\langle{#1}\right\rangle}}

\newcommand\+[1]{\mathcal{#1}}

\newcommand{\wt}{\widetilde}
\newcommand{\ol}{\overline}
\newcommand{\ul}{\underline}

\newcommand{\lra}{\leftrightarrow}
\newcommand{\LR}{\Leftrightarrow}
\newcommand{\RA}{\Rightarrow}
\newcommand{\LA}{\Leftarrow}

\newcommand{\sssl}{\ensuremath{|\sigma|}}

\def\uh{\upharpoonright}

\DeclareMathOperator \init{init}

\newcommand{\RCA}{\ensuremath{\mathbf{RCA_0}}}

\newcommand{\frb}{\mathfrak{b}}
\newcommand{\frd}{\mathfrak{d}}

\numberwithin{equation}{section}
\renewcommand{\hat}{\widehat}
\begin{document}

\title{Logic Blog 2016}

 \author{Editor: Andr\'e Nies}

\maketitle

%\begin{abstract}  %The 2015  logic blog has focussed on the following:
% \end{abstract}

 {
The Logic Blog is for
\bi \item rapidly announcing    results related to logic
\item putting up results and their proofs for further research
\item archiving results for later use
\item getting feedback before submission to   a journal.   \ei

Each year's  blog is    posted on arXiv shortly after the year has ended.
\vsp
\begin{tabbing} 

  \href{http://arxiv.org/abs/1602.04432}{Logic Blog 2015} \ \ \ \   \= (Link: \texttt{http://arxiv.org/abs/1602.04432})  \\
  
  \href{http://arxiv.org/abs/1504.08163}{Logic Blog 2014} \ \ \ \   \= (Link: \texttt{http://arxiv.org/abs/1504.08163})  \\

   \href{http://arxiv.org/abs/1403.5719}{Logic Blog 2013} \ \ \ \   \= (Link: \texttt{http://arxiv.org/abs/1403.5719})  \\

    \href{http://arxiv.org/abs/1302.3686}{Logic Blog 2012}  \> (Link: \texttt{http://arxiv.org/abs/1302.3686})   \\

 \href{http://arxiv.org/abs/1403.5721}{Logic Blog 2011}   \> (Link: \texttt{http://arxiv.org/abs/1403.5721})   \\

 \href{http://dx.doi.org/2292/9821}{Logic Blog 2010}   \> (Link: \texttt{http://dx.doi.org/2292/9821})  
     \end{tabbing}

\vsp

\n {\bf How does the Logic Blog work?}

\vsp

\n {\bf Writing and editing.}  The source files are in a shared dropbox.
 Ask Andr\'e (\email{andre@cs.auckland.ac.nz})  in order    to gain access.

\vsp

\n {\bf Citing.}  Postings can be cited.  An example of a citation is:

\vsp

\n  H.\ Towsner, \emph{Computability of Ergodic Convergence}. In  Andr\'e Nies (editor),  Logic Blog, 2012, Part 1, Section 1, available at
\url{http://arxiv.org/abs/1302.3686}.}

\vsp

\n {\bf Announcements on the wordpress front end.}  The Logic Blog has a \href{http://logicblogfrontend.hoelzl.fr/
}{front-end}  managed by Rupert H\"olzl.   

\n (Link: \texttt{http://logicblogfrontend.hoelzl.fr/})

\vsps
When you post source code on the logic blog in the dropbox, you can post a comment on the front-end alerting the community, and possibly summarising the result in brief.  The front-end is also good for posting questions. It allows MathJax.
 
The logic blog,  once it is on  arXiv,  produces citations on Google Scholar.
%\n  A. Taveneaux, \emph{Randomness Zoo}, Logic Blog, Section 6, available at
%
%\texttt{http://dl.dropbox.com/u/370127/Blog/Blog2011.pdf}.
\newpage
\tableofcontents

\part{Randomness, analysis and ergodic theory}
 \section{Westrick: randomness and rotations of the unit circle}
The following result was obtained at the computability retreat at  Research Centre Coromandel in February. It  started through   discussions between Adam Day, Andre Nies, Dan Turetsky and Brown Westrick.

\vsp

Let $P\sub[0,1]$ be a $\Pi^0_1$ class.  
\begin{thm} \label{thm:multiple recurrence}
Let $X \in MLR$. Let  $k\in \NN$. 
\begin{enumerate}
\item There is a rational number $q \neq 0$ such that for all $i\leq k$, $X + qi \in P$.
\item Let $\alpha$ be a computable irrational number.
There is an integer $n \neq 0$ such that   for all $i\leq k$, $(X + \alpha n i \mod 1) \in P$.
\end{enumerate}
\end{thm}
\begin{proof}

Both (1) and (2) are proved by the same method.  Dynamically construct a Solovay test as follows.  Put the empty string $\estring$ in the test.  Then for any $\sigma$ that has been placed in the test, let $r_\sigma$ be a rational number (for (1)) or let $n_\sigma >0$  be an integer (for (2)) such that \bc $\frac{2^{-|\sigma|}}{2k+3} < r_\sigma, (\alpha n_\sigma \mod 1) < \frac{2^{-|\sigma|}}{2k+2}$.  \ec These bounds are chosen so that for every $z \in [\sigma]$,   $z + kr_\sigma \in [\sigma]$ or $z - kr_\sigma \in [\sigma]$.  Let $U_s$ denote the complement of $P$ as seen at stage $s$.  We will enumerate $\tau$ into the test at stage $s$ if we see the following occur:
\begin{enumerate}
\item $\sigma \prec \tau$, 
\item $\tau$ is incomparable with any $\tau'$ which has already entered the test on the basis of this same $\sigma$
\item $[\tau] \cap U_s = \emptyset$
\item For some integer $i \in [-k, k]$, $[\tau] + i r_\sigma \subseteq [\sigma]\cap U_s$.
\end{enumerate}
The point is to enumerate $[\tau]$ if its potential $\pm r_\sigma$ $k$-recurrence inside $[\sigma]$ is invalidated.  For (2), replace $r_\sigma$ everywhere with $\alpha n_\sigma \mod 1$.

The test catches any $z\in P$ which fails to $k$-recur for all $r \in \mathbb{Q}\setminus \{0\}$  (resp. for all $\alpha n$ with $n \in \mathbb{Z}\setminus \{0\}$).  We need to show it is a Solovay test.

We show that for every $\sigma$ in the test, the total measure of all $\tau$ added to the test as a result of $\sigma$ is at most $\frac{2k+2}{2k+3} \mu [\sigma]$.  So if the depth of $\tau$ is the number of initial segments that $\tau$ has in the test, then $\mu(\cup\{[\tau] : \tau \text{ of depth } d\}) \leq \left(\frac{2k+2}{2k+3}\right)^d$, so this will suffice to show that the total measure of the test is finite.

When $\sigma$ is first put in the test, $[\sigma]$ is disjoint from $U_s$.  When first $\rho$ enters $U_s$ with $\sigma \prec \rho$, we can consider $[\sigma]$ as being divided into two parts, both invariant under addition of $r_\sigma$: $C=\{z \in [\sigma] : z + ir_\sigma \in [\rho]$ for some $i \in \mathbb{Z}\}$, and its complement.  All $z$ whose initial segments could potentially be enumerated as a result of $\rho$ are contained in $C$.  Also $[\rho]\subseteq C$, but nothing comparable with $\rho$ will be enumerated.  By the choice of $r_\sigma$, $\frac{\mu[\rho]}{\mu C} \leq \frac{1}{2k+3}$.  Therefore, the measure added to the test as a result of the addition of $\rho$, and as a result of the addition of any future $\rho'$ with $[\rho'] \subseteq C$, is bounded by $\frac{2k+2}{2k+3}\mu C$.  

The remaining set $[\sigma] \setminus C$ is currently untouched but may in the future be seen to intersect the complement of $P$.  When that happens, we apply the same reasoning inside of $[\sigma]\setminus C$, partitioning it into two invariant pieces and arguing that one piece remains untouched, while the other can never contribute more than $\frac{2k+2}{2k+3}$ of its measure to the test.  Continuing in this way, we get the desired bound on the measure the test uses in response to~$\sigma$.  
\end{proof}

\section{Nies: alternative proof of Thm.\ \ref{thm:multiple recurrence} (1)}

The  case $k=1$ of the theorem can be derived from a result of Figueira et al.\ \cite{Figueira.Miller.ea:09}; also see \cite[3.3.7]{Nies:book}. That theorem of \cite{Figueira.Miller.ea:09} says that if $X$ is not autoreducible, then there is a bit position $n$ such that the bit $n$ is indifferent for $X$ with respect to $P$, which means that we can change $X$ at $n$ and remain in $P$. This change corresponds to adding or subtracting the rational $\tp{-n}$. 

Now let us obtain an arithmetical progression  namely $X + qi \in P$ for $i < k/2$ (for technical reasons). 

At first we work with $Z\in k^\omega$. Extending the case of $k=2$, we say that $Z$ is auto-reducible if there is a reduction procedure $\Phi$ such that $\Phi^Z(n) \neq Z(n)$ for each $n$, and the computation $\Phi^Z(n)$ only queries the oracle at values other than $n$. As before, one checks that a  ML-random sequence $Z$ is not auto-reducible. We say that position $n$ is indifferent for $Z$ with respect to  a $\PI 1$ class $P\sub k^\omega$ if we can change $Z$ at $n$  to any value $< k$ and remain in $P$.

\begin{prop}  Let $Q\sub  k^\omega$ be a $\PPI$ class. If  $Z \in Q$ is not autoreducible, then there is a   position $n$ such that the bit $n$ is indifferent for $Z$ with respect to $Q$. \end{prop}
\begin{proof} A straightforward extension of the argument for $k=2$. If there is no such position, given input $n$ the reduction $\Phi$ searches for a stage $s$ and $i < k$ such that $Z$, with position $n$ changed to value $i$, is not in $Q$, and once found, output this $i$. 
\end{proof}
To obtain the arithmetical progression assume that $k= \tp r$ for $r \in \NN$.  Given $\PPI$ class $P \sub \cantor$ and ML-random $X\in P$  , let $Q$, $Z$ be the class/sequence rewritten using the alphabet $0, \ldots, k-1$. That is, the  block   bits of $X$ in positions $nr \ldots (n+1)r -1$ corresponds the  symbol of $Z$ in position $n$.  If $n$ is indifferent for $Z$ in $Q$, then we have an arithmetical progression   $X + qi\in P$,  where $i < k/2$, $q = \pm \tp{-rn}$.

%has to be part 2 because it's cited
   \part{Randomness via algorithmic tests}
\section{Smart sets for arbitrary cost functions}
\newcommand{\conc}{\hat{\,\,}}
\newcommand{\cost}{\mathbf{c}}
\newcommand{\dost}{\mathbf{d}}
\newcommand{\cc}{\mathbf{c}}
\newcommand{\limcost}{\underline{\cost}}
\newcommand{\converge}{\!\!\downarrow}

\newcommand \sinit {s_{\mathrm{init}}}

\label{s:smart for cost functions}

The following is work of Greenberg, Miller, Nies and Turetsky at RCC in Feb.\ 2015, and in Wellington slightly later. For background on cost functions see \cite{Nies:CalculusOfCostFunctions}. In the following all cost functions satisfy the limit condition $\lim_x \cost(x) = 0$. 

\begin{definition}[\cite{Bienvenu.Greenberg.ea:16}] \label{def:cost-bounded_test} 	Let~$\cost$ be a cost function. A descending  sequence $\seq{V_n}$ of uniformly c.e.\ open sets is a \emph{$\cost$-bounded test} if $\leb(V_n) =O( \limcost(n))$ for all~$n$. \end{definition}
We think of each $V_n$ as an approximation for $Y\in \bigcap_k V_k$. Being in  $\bigcap_n V_n$ can be viewed as a new sense of obeying $\cost$ that makes sense for  ML-random sets.    
\begin{lemma} \label{lem:shift} Suppose $aY$ fails a $\cost$-bounded test $\bigcap_n V_n$ where $a \in \{0,1\}$. Then $Y$ fails a $\cost$-bounded test. \end{lemma}

\begin{proof} We may suppose $a=0$ and $X \in V_n$ implies $X(0) =0$. Since $\leb T (V_n )\le 2 \leb V_n $ where $T$ is the usual shift operator on Cantor space, $ \seq{ T (V_n)}$ is also a $\cost$-bounded  test. Clearly $Y$ fails it. \end{proof}

The  basic motivating result    is     a  generalisation  in terms of cost functions of a fact of  Hirschfeldt and Miller.  %
\begin{prop} \label{prop:basic fact} If   $A \models \cost$ and $Y$ is a ML-random captured by a $\cost$-bounded test, then $A \leT Y$.  \end{prop}

\subsection{The most powerful kind of set   obeying a  given cost function}
The following is the  central definition for this entry: given $\cost$, we consider sets $A$ such that the converse   implication  holds as well.
\begin{definition} Let $\cost$ be a cost function and $A$ be a $\DII$ set. We say that $A$ is \emph{smart  for $\cost$} if $A \models \cost$ and for each ML-random set $Y$, 
\begin{center} $Y$  is captured by a $\cost$-bounded test $ \LR  A \leT Y$. \end{center} \end{definition}

Informally, if $A$ is smart for $\cost$ then $A$ is  as complex as possible among the sets obeying $\cost$, in the sense that  the only random sets $Y$ above $A$ are the ones that have to be there because $A$ obeys the cost function that puts it below $Y$ anyway.

%%%%%%%%%%%
\begin{theorem} \label{thm:smart} Let $\cost$ be  a cost function with $\cost \to \cost_\Omega$. Some  c.e.\ set $A$  is smart  for~$\cost$.  \end{theorem} 

\begin{proof} Recall  $\Upsilon$ is a ``universal"   Turing functional   in the sense    that  $\Upsilon({0^e1}\conc{X}) = \Phi_e(X)$ for each $X,e$. We build $A$ and a $\cost$-test $\seq{\+ U_k} $ capturing any ML-random $Y$ such that $A = \Upsilon^Y$. This suffices for the theorem by Lemma~\ref{lem:shift}.

 Since $2^{-x}\le^\times \Omega - \Omega_x$,   we may assume that $\cost(x,s) \ge \tp{-x}$ for $x \le s$. 

As in \cite{Bienvenu.Greenberg.ea:16}, during the construction of $A$ we build a global ``error set":
\[ \+ E_s = \{ Y \colon \ex n\, [ \Upsilon^Y_s(n)\converge=0 \lland A_s(n) = 1 ]\}\]
All all stages $s$ we will have 
\[  \tag{$\diamond$}   \leb \+ U_{k,s} \le \cost(k,s) + \leb ( \+ E_{s+1} - \+ E_k). \]
So since $\leb (\+ E - \+ E_k) \le^\times \cost_ \Omega(k)  $ and   $\cost_\Omega(k)  \le^\times\limcost(k) $, the test $\seq{\+ U_k} $ is indeed a $\cost$-test. 

We reserve the interval $I_k = [\tp k, \tp{k+1} )$ for ensuring  ($\diamond$). The construction of  $\+ U_k$ is as follows.
At stage $s > k$, let  $x = \min (I_k - A_{s-1})$. Let 
\[ \+ U_{k,s} = \bigcup_{t<s} \{ Y \colon A_t \uhr {x+1} \preceq \Upsilon_t^Y \} - \+ E_{k} \]  
If   ($\diamond$) threatens to fail at $s$, namely  $\leb \+ U_{k,s} >  \cost(k,s) + \leb ( \+ E_s - \+ E_k)$,   put $x$ into $A_{s+1}$. This causes $\+ U_{k,s}$ to go into $\+ E_{s+1}$. 

First we verify that $x$ always exists, that is, we enumerate at most $\tp k$ times for $\+ U_k$.   If we do this at stage $s$, $\leb\+ U_{k,s} > 2^{-k} + \leb(\+ E_s - \+E_k)$.  Since $\+U_{k,s} \cap \+E_k = \emptyset$ by definition, and $\+U_{k,s} \subseteq \+E_{s+1}$, it follows that $\leb(\+E_{s+1} - \+E_s) > 2^{-k}$.  Since $\leb \+E \le 1$, this can happen at most $2^k$ times.

 In particular, if $A = \Upsilon^Z$ then $Z \in \bigcap_k \+ U_k$. 
 
 It remains to verify that $A \models \cost$. If we enumerate $x$ for $\+ U_k$ at stage $s$ then 
 \[ \leb (\+ U_{k,s} - \+ E_s) = \leb (\+ U_{k,s} - (\+ E_s - \+ E_k) ) \ge \leb (\+ U_{k,s} ) - \leb(\+ E_s - \+ E_k) >  \cost(k,s ) \ge  \cost (x,s). \]  Since $ \+ U_{k,s} - \+ E_s \sub \+ E_{s+1} - \+ E_s$, we see that $\cost(x,s) < \leb(\+E_{s+1} - \+E_s)$.  This implies that the total cost of the enumeration of $A$  is at most $1$.
\end{proof} 

%%%%%%%
%\hl{(OK maybe we should first call it smart, but it's the same as ML complete anyway after the theorem is true. ML complete is the better term because one can guess what it means and it relates to the previous sections)}

\subsection{ML-reducibility}
\begin{definition}[\cite{Bienvenu.Greenberg.ea:16}]\label{def:MLred} For $K$-trivial sets $A$ and $B$, we write  $B \le_{\textup{ML}} A$ if $A \leT Y$ implies $B \leT Y$ for any ML-random set $Y$.
\end{definition}
 
Clearly, $\leT $ implies $\le_{\textup{ML}}$. The ML-degrees form an upper semilattice where the least upper bound of $K$-trivial sets  $C $ and $D$ is given by the $K$-trivial set $C \oplus D$.

\begin{definition} Let $\cost$ be a cost function and $A$ be a $\DII$ set. We say that $A$ is \emph{ML-complete   for $\cost$} if $A \models \cost$, and $\fa B \, [B \models \cost \RA  B \le_{ML}   A] $.  \end{definition}
\begin{cor} $A$ is smart for $\cost$ $\LR$ $A$ is ML-complete for $\cost$. \end{cor} 
\begin{proof} 
$\RA$:   Suppose $A\leT Y$ for ML-random $Y$. Then  some $\cost$-bounded test captures $Y$. If $B \models \cost$, then $B \leT Y$ by the basic fact~\ref{prop:basic fact}. Thus $B \le_{ML} A$ as required. 

\noindent $\LA$: By Theorem~\ref{thm:smart} let $\wt A$ be smart for $\cost$. Suppose $A \leT Y$ for ML-random $Y$;  we want to show that  $Y$ is captured by a $\cost$-bounded test. Since $A$ is ML-complete for $\cost$ we have  $\wt A \le_{ML} A$, so $\wt A \leT Y$, so $Y$ is captured by a $\cost$-bounded test as required.
\end{proof}
In particular, the ML-degree of a  smart set $A$ for $\cost$ is uniquely determined by $\cost$.  On the other hand, for each low c.e.\ set $A$ there is a c.e.\ set $B \not \le_T A$ such that $B \models \cost$ \cite[5.3.22]{Nies:book}.  If $A$ is smart for $\cost$, then $A \oplus B$ is also smart for $\cost$. As each $K$-trivial is low, the Turing degree of a set $A$ that is smart for $\cost$ is not uniquely determined by $\cost$. 

\begin{question} Given $\cost$  can we build a smart for $\cost$ set  $A$ that is cappable? Can we even have two smart for $\cost $ sets that form a minimal pair? \end{question}

%%%%%%%%%%%%%%%%%%%%%%%%%%%%%%%
%%%%%%%%%%%%%%%%%%%%%%%%%%%%%%%
\subsection{The strongest  cost function obeyed by a   given set}
 Given c.e.\ $K$-trivial $A$ we will define a c.f.\ $\cost_A$ with $A \models \cost_A$ such that every random computing $A$ is captured by a $\cost_A$ test. (In other words, $A$ is smart for $c_A$.) However, $\cost_A$ may not be a very natural cost function. We build a $K$-trivial $A$ such that the class of sets obeying $\cost_A$ is not closed downward under $\leT$, and in fact not even the shift $T(A) $ obeys $\cost_A$. 
 
As before $\Upsilon$ denotes a ``universal" Turing functional.  Given a c.e.\ $K$-trivial set $A$, fix a c.e.\ approximation $\seq{A_s} \models \cost_\Omega$.   We let 
 \[ \cost_A(x,s) = \leb \bigcup_{x\le t< s} \{ Y \colon A_t \uhr {x+1} \preceq \Upsilon_t^Y \}.\]
 \begin{prop} (i) $A \models \cost_A$. (ii) Suppose $\cost$ is a cost function such that $A \models \cost$. Then $\cost_A \to \cost$.   In particular, $\cost_A \to \cost_\Omega$. \end{prop} 
 \begin{proof}  (i) We show that the fixed approximation $\seq{A_s} \models \cost_A$.  Define the left-c.e.\ ``error real'' by:
 \[
 \epsilon_s = \leb\{Y : \exists n [\Upsilon^Y_s(n)\converge=0 \wedge A_s(n) =1]\}
 \]
Note that if $x \in A_s - A_{s-1}$, then $\cost_A(x,s) \le \epsilon_s - \epsilon_x = \cost_\epsilon(x,s)$.  So $\cost_A\seq{A_s} \le \cost_\epsilon\seq{A_s}$.  Since $\cost_\Omega \to \cost_\epsilon$, and $\seq{A_s} \models \cost_\Omega$, it follows that $\seq{A_s} \models \cost_A$.
  
  (ii) By multiplying by a constant, we may assume that $\cost(0) < 1/2$.  Fix a computable speed-up $f$ such that $\cost\seq{A_{f(s)}} < 1/2$.  Define a Turing functional $\Psi$ such that at every stage $f(s)$, $\leb\{Y : A_{f(s)}\uhr{x+1} \prec \Psi^Y_{f(s)} \} = \cost(x,s)$.  By a simple argument, the measure of the error-set $\+E$ for this functional will be $\cost\seq{A_{f(s)}} < 1/2$, so this construction may proceed.
  
Fix $e$ with $\Phi_e = \Psi$.  Then $\limcost_A(x) \ge 2^{-(e+1)} \limcost(x)$.
\end{proof}

 Recall that  $T(A)$ is the shift of $A$. 
 \begin{thm} 
\label{thm:shift} For each cost function $\dost$, there is a cost function $\cost \ge \dost$ and a c.e.\  set  $A $ such that $A \models \cost$ and $T(A) \not \models \cost$. \end{thm}
 Since $\cost_A \to \cost$, this shows that $T(A) \not \models \cost_A$. 
 
 \begin{proof} We   fix a listing  $\seq{\Phi_e}$  of all (possibly partial) computable enumerations   $\seq{B_t}$, where  $B_t \simeq D_{\Phi_e(t)}$ and $D_{\Phi_e(t)} \sub  D_{\Phi_e(t+1)}$ if defined.   
 %We write $\seq {\Phi_e} $ for this enumeration. 
 
 We may assume $\dost(s-1, s) \ge \tp{-s}$. We define $\cost(x,s)$ so that $\cost(x,s)  \ge \dost(x,s)$ for each $x,s$. At a stage $x$ of the construction we may also  declare that $\cost(x-1,x) \ge \alpha$, which by monotonicity entails that $\cost(y, s) \ge \alpha $ for each $y<x$ and $s \ge x$.
 
We meet the requirements
\[R_e \colon \, T(A) = \bigcup_t D_{\Phi_e(t)}    \RA \cost \seq {\Phi_e} \ge 1. \]
The strategy for $R_e$ tries to  ensure $\limcost (x-1)$ is large and $\limcost(x)$ is small for sufficiently many $x$. In that case $R_e$ can put $x$ into $A$ for the small cost, while the opponent's enumeration $\seq{\Phi_e}$ of $T(A)$ has to deal with the large cost. One    problems in implementing this idea is the timing as we will  of $\Phi_e (u) $ for a stage $u$ only at a stage $s$ much larger than $u$.  Also we always have  $\cost(x,s) \ge \dost(x,s)$, so once we discover that the enumeration of $x$ would help because  $\seq{\Phi_e}$ caught up sufficiently much, it may be that the enumeration of $x$ has become too expensive for $R_e$. Similar to the usual construction of  a set obeying $\dost$, in this case we simply pick a new $x$. Since $\dost$ has the limit condition, eventually we will always be able to keep $x$. 
 
 At a stage $s >0$ let $s_{\init}(e)$ be the greatest stage $t< s$ such that $t=0$ or $R_e$ has been initialised at $t$. If by stage $s$ the strategy for $R_e$ has been initialised for $b$ times it can spend a $\cost$-cost of $\tp{-b-e}$ in enumerating $A$. It is also allowed to raise $\cost(x,s)$ to $\tp{-\sinit(e)}$.
 
 The  $e$-expansionary stages are the ones at which $\seq {\Phi_e} $ catches up with $T(A)$. We declare $0$ as $e$-expansionary. A stage $s>0$ is $e$-\emph{expansionary} if for the largest $e$-expansionary stage $t<s$, we have that $T(A_s)\uhr{t+2} = \Phi_{e,s}(u) \uhr {t+2}$ where $u$ is largest such that $\Phi_{e,s}(u)$ is defined and $u>t$. 
 
 \noindent \emph{Strategy for $R_e$.}  Write $\alpha = \tp{-\sinit(e)}$. A  rational parameter $\gamma_e \in [0,1]$ measures progress of $R_e$. We set $\gamma_e$   to $0$ when $R_e$ is initialised. 
 
 At  $e$-expansionary stage $s$, if $\gamma_e \le 1$ do the following. Initialize lower priority requirements.   Declare that $\cost (s-1,s+1) \ge \alpha$. (So for $\Phi_e$, changing  at $s-1$ will be expensive after stage $s$.)
 
  Let $x < s $ be the last $e$-expansionary stage. If $\dost(x,s) < \tp{-b-e} \alpha$, put $x$ into $A_{s}$ (note that no-one has raised the $\cost$ cost for $x$ above  $\dost(x,s)$ yet), add $\alpha $ to $\gamma_e$.  Say that $R_e$ acts.

 Clearly   $R_e$   only acts $ \tp{\sinit(e)}= 1/\alpha $ times while it is not initialised.

 \begin{claim}  $A \models \cost$.  \end{claim}
 When $R_e$ acts at $s$  we have $\cost(x,s)\le \tp{-b-e} \alpha $ by initialisation. 
 
  \begin{claim}  $TA \not \models \cost$.  \end{claim}
Otherwise $TA \models \cost$ via some  enumeration $\seq {\Phi_e} $ such that $\cost \seq {\Phi_e} <1 $.  Since $\dost$ satisfies the limit condition, $\gamma_e$ reaches the value $1$. This is at least what $\seq {\Phi_e} $ pays for  the enumerations of $x-1$ into $TA$. For when $R_e$ acts at $s$ via $x$ then  $x-1 \not \in \Phi_e(u)$ for some $u>x$ since $s$ is $e$-expansionary. By the next $e$-expansionary stage we have $x-1 \in \Phi_e(u')$ for some $u' > u$. So $\Phi_e$ paid the cost $\cost(x-1,x+1) \ge \alpha$ that was set by $R_e$  at stage $x$.
 \end{proof}

\part{Computability theory}
 \section{Merkle, Nies and Stephan: A  dual of the Gamma question}

Merkle, Nies and Stephan worked at NUS in February. They considered a dual of the $\Gamma$ operator on Turing degrees.

For $Z\sub \NN$ the lower density is defined to be
$$\underline \rho (Z)  = \liminf_n \frac{|Z \cap [0, n)|}{n}.$$
Recall that $$\gamma(A)=\sup_{X \, \text{computable}} \underline \rho(A\leftrightarrow X)$$   
$$\Gamma(A) = \inf \{ \gamma(Y) \colon  \, Y \le_T A \}$$
which only depends on the Turing degree of $A$.  The $\Gamma$ operator was introduced by  Andrews, Cai, Diamondstone, Jockusch and Lempp \cite{Andrews.etal:2013}.
\begin{definition}  \label{def:delta}  $$\delta(A)=\inf_{X \, \text{computable}} \underline \rho(A\leftrightarrow X)$$
$$\Delta(A) = \sup \{ \delta(Y) \colon  \, Y \le_T A \}.$$
\end{definition}
Intuitively, \bi \item $\Gamma(A)$ measures  how well computable sets can approximate    the sets that $A$ computes, counting   the asymptotically worst case (the infimum over all $Y \le_T A$). In contrast, \item $\Delta(A)$ measures how well the sets   that $A$ computes can approximate   the computable sets, counting  the asymptotically best case (the supremum over all $Y \le_T A$). 
\ei
Clearly the maximum value of $\Delta(A)$ is $1/2$. The operator $\Delta(A)$ is related to the analog of a cardinal characteristic introduced by Brendle and Nies  in the 2015 Logic Blog~\cite{LogicBlog:15}.  They define  \mbox{$\+ B(\sim_p)$}  to be  the class of oracles $A$ that compute a set $Y$  such that   for each computable  set $X$, we have $\ul \rho (X \lra Y) > p$.  For each $p$ with $0 \le p < 1/2$,  
\bc  $\Delta(A) > p \RA   A \in \+ B(\sim_p) \RA  \Delta(A) \ge p$. \ec
	
	We state three minor  results.
	
\begin{prop} \label{prop:2gen} Let $A$ be  2-generic. Then $\Delta(A)=0$. \end{prop} 
A proof is given at the end of Section~\ref{s:Nies_Delta}.

\begin{prop} Let $A$ compute a  Schnorr random $Y$. Then $\Delta(A) = 1/2$. \end{prop}
This is clear because  $\underline \rho(Y\leftrightarrow R)=1/2$ for each computable $R$.

\begin{prop} Let $p\in (0,1/2)$ be computable. Let $A$ be Schnorr random for the Bernoulli measure w.r.t.\ $p$. Then $\delta(B) =p$ for each $B \equiv_1 A$. \end{prop}

The ``$\Delta$-question" is (or rather, was, and has been quite short-lived):

\begin{question}[solved] Can $\Delta(A)$ be properly between $0$ and $1/2$? \end{question}
Using a dual form of Monin's technique's below, this has been answered in the negative by Nies; see Section~\ref{s:Nies_Delta}.

\section{Monin - A resolution of the Gamma question} 	

We show that there is no sequence $X$ with a Gamma value strictly between $0$ and $1/2$.

\begin{definition}
For a given $n \in \omega$ and two strings $\s_1, \s_2 \in 2^n$, the notation $d(\s_1, \s_2)$ denotes the normalised  \textbf{hamming distance} between $\s_1$ and $\s_2$, that is, the     number of bits   on which $\s_1$ and $\s_2$ differ divided by  $n$.
\end{definition}

\begin{definition}
Let $F:\omega \rightarrow \omega$ be a function. Given  a function $f$ such that $f(n)<2^{F(n)}$, for each $n$, by   $[f(n)]$ we  denote the string of length $F(n)$ encoded by $f(n)$.
\end{definition}

The following weakens the  notion of infinitely often equal (i.o.e.) for a computable bound, which was  first studied in \cite{Monin.Nies:15}.
\begin{definition}
Let $F:\omega \rightarrow \omega$ be a function and $\alpha \in [0,1]$. A function $f$ is $2^{F(n)}$-infinitely often $\alpha$-equal if for every computable function $g:\omega \rightarrow \omega$ which is bounded by $2^{F(n)}$, we have:
$$\liminf_{n} d([f(n)], [g(n)])  \leq 1 - \alpha$$
%$$\limsup_{n} \frac{\#\{i : [f(n)](i) = [g(n)](i)\}}{n} \geq \alpha$$
Informally, we want $f$ to equal infinitely often on a fraction of at least $\alpha$ bits, to every computable function bounded by $2^{F(n)}$. Typically we will have $\alpha > 1/2$.
%A binary sequence is $2^{F(n)}$-infinitely often $\alpha$-equal if it computes a function which is $2^{F(n)}$-infinitely often $\alpha$-equal.
\end{definition}

\begin{proposition} \label{yoyo}
Let $1/2 < \alpha < 1$. Suppose that for no $k$, $X$ computes a function which is $2^{\lfloor 2^{n/k} \rfloor}$-i.o.$\alpha$-e. Then $\Gamma(X) \geq 1 - \alpha$.
\end{proposition}
\begin{proof}
Consider any sequence $Y$ computed by $X$. Fix some $c \in \omega$ and let $k$ be the smallest integer such that $2^{1/k} - 1 < 1/(2c)$. We then split $Y$ in blocks of bits of length $\lfloor 2^{n/k} \rfloor$. We now argue that for $n$ large enough, the number of bits in the $n+1$-th block is smaller than $1/c$ times the sum of the number of bits in the previous blocks. By the sum of geometric series we have:

$$\sum_{i=0}^{i \leq n} 2^{i/k} = \frac{2^{(n+1)/k} - 1}{2^{1/k} - 1}$$
which implies that 
$$
\begin{array}{rcl}
2^{(n+1)/k} - 1&=&(2^{1/k} - 1)\sum_{i=0}^{i \leq n} 2^{i/k}\\
						&\leq&\frac{1}{2c}\sum_{i=0}^{i \leq n} 2^{i/k}\\
\end{array}
$$
For $n$ large enough we have $\frac{1}{2} \sum_{i=0}^{i \leq n} 2^{i/k} > (n+c)$. Thus for $n$ large enough we have:

 %which implies that:
$$
\begin{array}{rcl}
2^{(n+1)/k} - 1&\leq&\frac{1}{2c}\sum_{i=0}^{i \leq n} 2^{i/k}\\
						 &\leq&\frac{1}{2c}\sum_{i=0}^{i \leq n} 2^{i/k} + \frac{1}{2c}\sum_{i=0}^{i \leq n} 2^{i/k} - \frac{1}{c}(n+c)\\
						 &\leq&\frac{1}{c}(\sum_{i=0}^{i \leq n} 2^{i/k} - n - c)\\
						 &\leq&\frac{1}{c}(\sum_{i=0}^{i \leq n} \lfloor 2^{i/k} \rfloor - c)\\
						 &\leq&\frac{1}{c}(\sum_{i=0}^{i \leq n} \lfloor 2^{i/k} \rfloor) - 1
\end{array}
$$
We then have for $n$ large enough that:
$$\lfloor 2^{(n+1)/k} \rfloor \leq \frac{1}{c}(\sum_{i=0}^{i \leq n} \lfloor 2^{i/k} \rfloor)$$

%We want to show that the length of the $n+1$-th block of bit is smaller than $1/c$ of the sum of the length of the previous blocks. Formally we want:
%
%$$2^{(n+1)/k} \leq 1/c \sum_{i=0}^{i \leq n} \lfloor 2^{i/k} \rfloor$$
%
%In order to achieve this we need to remove at most $1$ for each integer value, and $c$ more to remove $1$ on the left of the inequality. Thus we want:
%
%$$2^{(n+1)/k} - 1 < 1/c (\sum_{i=0}^{i \leq n} 2^{i/k} - n - c)$$
%
%In other word we want 
%
%$$\sum_{i=0}^{i \leq n} 2^{i/k} (2^{1/k} - 1) < 1/c (\sum_{i=0}^{i \leq n} 2^{i/k} - n - c)$$
%
%As $2^{1/k} - 1 < 1/(2c)$, this is true for $n$ large enough.

Thus the length of the $n+1$ block of bit is smaller than $1/c$ of the sum of the length of the previous blocks.

 Define the function $f < 2^{\lfloor 2^{n/k} \rfloor}$ by setting $f(n)$ to the value coded by the bits of the $n$-th block. Suppose that $f$ is not $2^{\lfloor 2^{n/k} \rfloor}$ i.o.$\alpha$-e. In particular there must be some computable function $h \leq 2^{\lfloor 2^{n/k} \rfloor}$ such that for almost every $n$, $[h(n)]$ agrees with $[f(n)]$ on a fraction of strictly less than $\alpha$ bits. Let  $h'$ be defined as the complement of $h$ bitwise. Then  for almost every $n$, $[h'(n)]$ agrees with $[f(n)]$ on a fraction of bits strictly bigger than $1-\alpha$. 

Now consider the bit number $m$ of the sequence defined by $f$, starting at the begining of a block. Let $m+n$ be the last position of that block. By hypothesis, among the $m$ first bits (for $m$ large enough), there are at least $m(1-\alpha) - \mathcal{O}(1)$ bits which are guessed correctly by $h'$. In particular, for any $1 \leq i \leq n$, there are also at least $m(1-\alpha) - \mathcal{O}(1)$ bits which are guessed correctly among the $m+i$ first bits. Also for any $1 \leq i \leq n$, the number of total bits is at most $m + n$, which is at most $m + m/c$. Thus for each $i$ the fraction of bits which are guessed correctly before $m+i$ is at least:

$$\frac{m(1-\alpha) - \mathcal{O}(1)}{m+m/c}$$
As $m$ goes to infinity, this value converges to:
$$\frac{1-\alpha}{1+1/c}$$

Thus $\gamma(Y) \geq \frac{1-\alpha}{1+1/c}$. We can carry out  this argument  for $c$ larger and larger, making lower bounds on  $\gamma(Y)$ closer and closer to $1-\alpha$. Thus if for any $k$, we can split  any $Y $ into blocks of bits as above such that the resulting function is not $2^{\lfloor 2^{n/k} \rfloor}$-i.o.$\alpha$-e., then $\gamma(Y) \geq 1 - \alpha$. By hypothesis we can do this  for  every $Y$ computable by $X$. Hence  $\Gamma(X) \geq 1 - \alpha$.
\end{proof}

By contrapositive, if $\Gamma(X) < (1-\alpha)$, for $1/2 < \alpha < 1$, then for some $k \in \omega$, $X$ computes a function which is $2^{\lfloor 2^{n/k} \rfloor}$-i.o.$\alpha$-e. \\

%\begin{corollary}
%If $\Gamma(X) < (1-\alpha)$ then for any $c \in \omega$, there exists a rational $x$ such that $X$ computes a function which is $2^{\lfloor 2^{c n + x} \rfloor}$-i.o.$\alpha$-e. 
%\end{corollary}

We now need to borrow some technics from the field of error correcting codes. The idea is the folloing: We want to transmit some message of length $m$. But some random bit flip can occur during the transmission. We want to make sure that if the percentage of error is small enough, we can still recover the original message. The idea is to use an injection $\Phi$ from $2^m$ into $2^n$ for some $n >> m$, in such a way that the elements in the range of $\Phi$, are pairwise far away from each other, in the sense of  the hamming distance. If $d$ is the smallest distance between two elements in the range of $\Phi$, then it is clear that we can recover up to $d/2$ error.

Also we would like $n$ to be not much bigger than $m$ (ideally a multiplicative constant). It is easy to find such a multiplicative constant allowing, to have a list of $2^m$ elements of $2^n$, which are all at a distance of at least $1/2 - \epsilon$ from each other, for $\epsilon$ as small as we want. Thus it is possible to correct this way, up to a fraction $1/4$ of error. It is not anymore possible if the fraction is bigger than $1/4$. However, if the fraction is smaller than $1/2$, it is possible to identify a small list of messages, among which must figure our original message. It is even possible to make the size of the list constant. Formally we use the following theorem, for which we provide a proof for completness:

\begin{theorem}[The list decoding capacity theorem] \label{list_decod}
Let $0 < \theta < 1/2$. There exists $L \in \omega$ and $0 < R < 1$ as follows.

For any $n$, there exists a set $C$ of $2^{\lfloor Rn \rfloor}$ many strings of length $n$ such that for any string $\s$ of length $n$, there are at most $L$ strings $\tau$ in $C$ such that $d(\s, \tau) < \theta$.
\end{theorem}
\begin{proof}
We prove that for parameters $L$ and $R$ well chosen, if we pick at random the strings in $C$, the theorem is true with positive probability. 

Let $\beta = 2(1/2 - \theta)^2\log(e)$ and pick $L$ such that $\beta - \frac{1}{L} > 0$. Then pick $R$ such that $R  < \beta - \frac{1}{L}$.\\

Using Chernoff bounds, for any string $\tau$ of length $n$, the measure of the set $\{[\s] : |\s| = n \text{ and } d(\tau, \s) < \theta\}$ is bounded by $e^{-2(1/2-\theta)^2n} = 2^{-\beta n}$. Thus, given a string $\s$ the probability that picking a string $\tau$ at random gives $d(\s, \tau) < \theta$, is bounded by $2^{- \beta n}$.

Now let $q = \lfloor Rn \rfloor$ and let $C$ be a collection of $2^{q}$ strings picked at random. For any subset of $L+1$ of these strings, the probability that a given string $\s$ has a hamming distance smaller than $\theta$ with each of them is bounded by $2^{-\beta n (L+1)}$. Thus the probability that a given $\s$ has a hamming distance smaller than $\theta$ with any possible subsets of size $L+1$ of $C$ is bounded by $\binom{2^q}{L+1} 2^{-\beta n (L+1)}$. And the probability that this happens for any string $\s$ is bounded by $2^n\binom{2^q}{L+1} 2^{-\beta n (L+1)}$. The following computation shows that this quantity is smaller than $1$:

$$
\begin{array}{rcl}
2^n\binom{2^q}{L+1} 2^{-\beta n (L+1)}&\leq&2^n 2^{R n (L+1)} 2^{-\beta n (L+1)}\\
&\leq&2^{-n(L+1)(-R - 1/(L+1) + \beta)}\\
&\leq&2^{-n(L+1)(-\beta + 1/L -  1/(L+1) + \beta)}\\
&\leq&2^{-n(L+1)((L+1 - L)/L(L+1))}\\
&\leq&2^{-n/L}
\end{array}
$$

It follows that that for any $n$, if $C$ is a collection of $2^{\lfloor Rn \rfloor}$ strings of length $n$ that we pick at random, the probability that no string $\s$ of length $n$ has a hamming distance smaller than $\theta$ with more than $L$ strings of $C$, is positive. In particular, for any $n$, there exists such a collection of strings.
\end{proof}

We can now prove that only $0$, $1/2$ and $1$ can be realized by $\Gamma$ values of sequences. First by \cite{Monin.Nies:15},     if $X$ compute a function bounded by $2^{(2^n)}$ which equals infinitely often every computable function bounded by $2^{(2^n)}$, then $\Gamma(X) = 0$: First,  it is also easy to show that if $f$ is $2^{(2^n)}$-i.o.e., then for any $c$, $f$ computes a function which is $2^{(2^{c \times n})}$-i.o.e.
Second, it  is easy to show that if $f$ is $2^{(2^{c \times n})}$-i.o.e. then $\gamma(f) < 1/(c+1)$.  
\begin{theorem}
Let $1/2 < \alpha < 1$. If $\Gamma(X) < 1-\alpha$ then $X$ is $2^{(2^n)}$-i.o.e. and hence $\Gamma(X) = 0$.
\end{theorem}
\begin{proof}
Suppose $\Gamma(X) < 1-\alpha$. In particular from \cref{yoyo}, $X$ computes a function $f$ which is $2^{\lfloor 2^{n/k} \rfloor}$-i.o.$\alpha$-e. for some $k \in \omega$.

Using \cref{list_decod}, we pick $L \in \omega$ and $0 < R < 1$ such that for any $n$, there exists a collection $C_n$ of $2^{\lfloor Rn \rfloor}$ strings of length $n$, such that no string $\s$ of length $n$ has a Hamming distance less than $\theta= 1 - \alpha$ with more than $L$ strings of $C$. Note that such a collection of strings $C_n$ is computable uniformly in $n$. Uniformly computable in  $n$ we fix a listing  $\s_0^n, \s_1^n, \dots \s^n_{2^{\lfloor Rn \rfloor}-1}$ of the elements of   $C_n$.

We define the following $X$-computable $L$-trace $\{T_n\}_{n \in \omega}$: For any $n$, $T_n$ is the collection of integer $i$ such that the hamming distance between  $[f(n)]$ and $\s_i^{\lfloor 2^{n/k} \rfloor}$ is less than $1 - \alpha$. Note that each $T_n$ is $X$-computable uniformly in $n$ and that $|T_n| \leq L$ (possibly $T_n$ is also empty). Note also that the values of $T_n$ are bounded by $2^{\lfloor R2^{n/k} \rfloor}$.

We claim that every computable function $g < 2^{\lfloor R2^{n/k} \rfloor}$ is traced by $T_n$. Indeed, given such a computable function $g$, consider the computable function $g'$ defined by $g'(n) = \s_i^{\lfloor 2^{n/k} \rfloor}$ if $g(n) = i$. We have that $g'$ is computable, and furthermore $g'(n) \leq 2^{\lfloor 2^{n/k} \rfloor}$. As $f$ is $2^{\lfloor 2^{n/k} \rfloor}$-i.o.$\alpha$-e., there   exist infinitely many   $m$ such that $d([f(m)], [g'(m)]) < 1 - \alpha$. Then, by definition of $\{T_n\}_{n \in \omega}$,  we   have $g(m) \in T_m$ for each of these $m$.

Thus every computable function bounded by $2^{\lfloor R2^{n/k} \rfloor}$ is captured infinitely often by $\{T_n\}_{n \in \omega}$.\\

Now, either the trace $T_{2n}$ must capture infinitely often every computable function bounded by $2^{\lfloor R2^{2n/k} \rfloor}$, or the trace $T_{2n+1}$ must capture infinitely often every computable function bounded by $2^{\lfloor R2^{(2n+1)/k} \rfloor}$, as otherwise, by combining the two computable witnesses that neither is the case, we would have a computable function bounded by $2^{\lfloor R2^{n/k} \rfloor}$ and not traced by $\{T_n\}_{n \in \omega}$. In either case, $\{T_n\}_{n \in \omega}$ can compute a trace which traces every function bounded by $2^{\lfloor R2^{2n/k} \rfloor}$ (Note that a trace capturing infinitely often every computable function bounded by $F$, also captures infinitely often every computable function bounded by $G < F$).

By iterating this argument, $X$ can compute an $L$-trace $\{T_n\}_{n \in \omega}$  which captures infinitely often every function bounded by $2^{L 2^{n}}$. We can also without loss of generality assume that each element of each $T_n$ is bounded by $2^{L 2^{n}}$, and thus coded on exactly $L 2^{n}$ bits.  We now use the fact that $|T_n| \leq L$ for every $n$, to compute using $T_n$ a function $h \leq 2^{2^{n}}$ which is equal infinitely often to every computable function bounded by $2^{2^{n}}$.

First for every $n$, we add if necessary some elements in $T_n$ such that $|T_n| = L$. Then we view each element $e_i$ of $T_n$ as an $L$-tuple $\langle e^1_i, \dots, e^L_i \rangle$. Formally $e^j_i$ is the $j$-th block of $2^n$ consecutive bits. Consider $L$ distinct $X$-computable functions $h_1, \dots, h_L$ given by $h_i(n) = e^i_i$ where $e_i = \langle e^1_i, \dots, e^i_i, \dots, e^L_i \rangle$ is the $i$-th element of $T_n$. We claim that at least one $h_i$ is $2^{(2^n)}$-i.o.e. Suppose otherwise, and consider the $L$ computable functions $p_1, \dots, p_L$ witnessing that. Then the computable function $p(n) = \langle p_1(n), \dots, p_L(n) \rangle$ is never captured by $T_n$, as the $i$-th component of $p(n)$ (seen as a $L$-tuple) is different from the $i$-th component of the $i$-th element of $T_n$ (seen as a $L$-tuple). This contradicts our hypothesis. So   at least one $h_i$ is  $2^{(2^n)}$-i.o.e.

Note that $h_i$  is computable from $X$. As in  \cite[Thm.\ III.4]{Monin.Nies:15}  from $h_i$ we can compute functions which are $2^{(a^n)}$-i.o.e. for any $a \in \omega$; the binary sequence encoding each of these function   have a $\gamma$ value  $\le 1/a$.
\end{proof}

Note that all the reductions are $tt$. Thus this also solves the $\Gamma$ question in the $tt$ degrees.

%\begin{lemma}[Andre] \label{lem:Andre} Given $c \ge 1$    let $k\ge$. Then for each $n \ge$
%$$\lfloor 2^{(n+1)/k} \rfloor \leq \frac{1}{c}(\sum_{i=0}^{ n} \lfloor 2^{i/k} \rfloor).$$ \end{lemma}
%
%\begin{proof} We have  $\lfloor 2^{i/k} \rfloor \le 2^{\lfloor i/k \rfloor}$. Next, $\sum_{i=0}^{ n} 2^{\lfloor i/k \rfloor} \le  k \sum_{r=0}^{ \lceil n/k \rceil} 2^r $ \end{proof}

\section{Brendle and Nies: Analog for cardinal characteristics of Monin's solution to the $\Gamma$ question }

For background and definitions of the characteristics see last year's blog \cite[Section 7]{LogicBlog:15}.  In analogy to Monin's post above, we will show that  $\frd(p)= \frd(\neq^*, \tp{(\tp n)})$  and $\frb(p)= \frb(\neq^*, \tp{(\tp n)})$  for each  $p\in (0,1/2)$.  

For convenience here are the main definitions:

Let  $R \subseteq  X \times Y$ be a relation between spaces $X,Y$ (such as Baire space) satisfying $\forall x \; \exists y \;
(x R y)$ and $\forall y \; \exists x \; \neg (x R y)$. Let $S = \{  \langle y,x \rangle \in Y \times X \colon \neg xR y\}$.  

\begin{definition} \label{df: bd} We write

\[ \frd(R) = \min\{|G|:G\subseteq Y \land \, \forall x \in X \,
\exists y \in  G   \, xR y\}.\]

\[ \frb(R) = \frd (S) =  \min\{|F|:F\subseteq X \land \, \forall y\in Y
\exists x \in F   \, \neg  xR y\}.\]
\end{definition}

%\subsection{The parameterised  families of relations} \label{ss:hbrelations}

We will study $\frd(R)$ and $\frb(R)$ for  two types of relations $R$.

\

\n {\it 1.}   Let $h \colon \omega \to \omega $  (usually unbounded).   Define  for $  x \in {}^\omega\omega$ and 

\n $  y \in \Pi_n \{0, \ldots, h(n)-1\}$,   

\[ x \neq^*_h  y \LR \fa^\infty n \,  [ x(n) \neq y(n)]. \]

\

\n {\it 2.} Let $ 0 \le p \le 1/2$. Define,   for $x,y \in {}^\omega 2$

\[ x \sim_p y \LR   \underline \rho (x \lra y)  >p, \]
where $x \lra y$ is  the set of $n$ such that $x(n) = y(n)$, and $\underline \rho$ denotes the lower density: $\ul \rho (z) = \liminf_n |z \cap n|/n$.

%\subsection{The cardinal characteristics}
\n It will be helpful to  express Definition~\ref{df: bd} for these relations   in words.

\

\n $\frd(\neq^*_h)$ is the least size of a set $G$ of $h$-bounded functions so that for each function $x$ there is a function $y$ in   $G$  such that    $\fa^\infty n [ x(n) \neq y(n)]$.  (Of course it suffices to require this for $h$-bounded $x$.)

\

\n  $\frb(\neq^*_h)$ is the least size of a set   $F$ of  functions such that for each $h$-bounded function $y$, there is a function $x$ in $F$  such that $\ex^\infty n \, x(n) = y(n)$.   (Of course we can require that each function in $F$ is $h$-bounded.)

 \

 \n $\frd(\sim_p)$ is the least size of a set $G$ of bit sequences  so that for each sequence $x$ there is a sequence  $y$ in  $G$   so that $\ul \rho (x \lra y) >p$. 
 
\

\n   $\frb(\sim_p)$ is the least size of a set  $F$ of  bit sequences such that for each bit sequence $y$, there is a sequence $x$ in $F$  such that $\ul \rho (x \lra y) \le p$.

\begin{definition} \label{df:KhLh} Let $h$ be a function of the form $2^{\hat h}$ with $\hat h \colon \, \omega \to \omega$, and let $X_h$ be the space of all $h$-bounded functions. 
 For  such a    function   we    view $x(n)$ either as a number, or as a binary string   of length $\hat h(n)$ via the binary expansion with leading zeros allowed. We  define $L_h\colon X_h \to   {}^\omega 2$ by   $L_h(x) = \prod_n x(n) $, i.e.  the concatenation of these strings.  We let $K_h \colon  {}^\omega 2 \to X_h$ be the inverse of $L_h$. \end{definition}

We begin with some preliminary facts of independent interest.
 On occasion we denote a function $\lambda n. f(n)$ simply by $f(n)$. The next lemma amplifies functions without changing the cardinal characteristics. 

\begin{lemma} \label{lem:double}  \bi \item[(i)] Let $h$ be nondecreasing and $g(n) = h(2n)$. We have $\frd(\neq^*,h) = \frd(\neq^*,g)$ and $\frb(\neq^*,h) = \frb(\neq^*,g)$.

\item [(ii)] For each $a,b >1$ we have $\frd(\neq^*,{\tp{(a^n)})}= \frd(\neq^*,{\tp{(b^n)}})$ and  

\n $\frb(\neq^*,{\tp{(a^n)})}= \frb(\neq^*,{\tp{(b^n)}})$. \ei \end{lemma}

\begin{proof} (i)  Trivially, $h \le g$ implies that   $\frd(\neq^*,h) \ge \frd(\neq^*,g)$ and $\frb(\neq^*,h) \le \frb(\neq^*,g)$. So it suffices to show two inequalities.

\vsps

\n \fbox{$\frd(\neq^*,h) \le\frd(\neq^*,g)$:}  Let $G$ be a witness set for $\frd(\neq^*,g)$. Note that $G$ is also a witness set for $\frd(\neq^*,h(2n+1))$. Let $\hat G= \{p_0 \oplus p_1\colon \, p_0, p_1 \in G \}$, where $(p_0 \oplus p_1)(2m+i) = p_i(m)$ for $i= 0,1$. Each function in $\hat G$ is bounded by $h$.  Since $G$ is infinite, $|\hat G| = |G|$. Clearly $\hat G$ is a witness set for $\frd(\neq^*,h)$.

\vsps

\n \fbox{$\frb(\neq^*,h) \ge \frb(\neq^*,g)$:}  Let $F$ be a witness set for $\frb(\neq^*,h)$.  Let $\hat F$ consist of the functions of the form $ n \to p(2n)$,  or  of the form  $n \to p(2n+1)$,  where $p \in F$. Then $|\hat F| = |F|$, and each function in $\hat F$ is $g$ bounded. 

Clearly,  $\hat F$ is a witness set for $\frb(\neq^*,g)$: if $q$ is $g$-bounded, then  $\hat q$ is $h$ bounded where $\hat q(2n+i) = q(n)$ for $i=0,1$. Let $p\in F$ be such that $\ex^\infty k \, p(k) = \hat q(k)$. Let $i \le 1 $ be such that infinitely many such $k$ have parity $i$. Then the function $n \to p(2n+i) $ which is in $\hat F$ is as required.

(ii) is immediate from (i) by iteration using that   $a^{2^i} >b$  and $b^{2^i} >a$ for sufficiently large $i$. 
\end{proof}

\begin{lemma} Let $a \in \omega - \{0\}$. We have $\frd( \neq^*,{\tp{(a^n)}}) \le \frd(1/a)$ and 

\n $\frb( \neq^*,{\tp{(a^n)}}) \ge \frb(1/a)$.\end{lemma}
\begin{proof} Let $I_m$ for $m \ge 2$  be the $m-1$-th consecutive interval of length $a^m$ in $\omega-\{0\}$, i.e.

\[ I_m = \big [\frac{a^m-1}{a-1} , \frac{a^{m+1} -1 } {a-1}\big). \]
First let $G$ be a witness set for $\frd(1/a)$. Let $h(n) = \tp{(a^n)}$.  We show that  $\hat G = \{ K_h(y) \colon \, y \in G \}$ is a witness set for $\frd( \neq^*,{\tp{(a^n)}})$. Otherwise there is a sequence $z\in {}^\omega 2$ such that for each $x \in  {}^\omega \omega$ there are infinitely many $m$ with  $x(m) = K_h(z)(m)$. Let $y' $ be the complement $\omega- z$ of  $z$, that is   $0$s and $1$s are interchanged. Then for infinitely many $m$, $L_h(x) (i) \neq y'(i)$ for each $i \in I_m$.  If we let $n = 1+ \max I_m$, the proportion  of $i < n$ such that $ L_h(x) (i) = y'(i)$ is therefore at most $(a^m-1) / (a^{m+1}-1)$, which converges to $1/a$ as $m \to \infty$. This contradicts the choice of $G$. 

Now let $F$ be a witness set for $\frb(\neq^*, h)$. Let $\hat F = \{ \omega - L_h(x) \colon \, x \in F \}$. For each  $y  \in \cantor$ there is $x\in F$ such that $\ex^\infty n  \,  K_h(y)(n) = x(n)$. This implies $\ul \rho (y \lra   x') \le 1/a  $ where $  x' =  \omega - L_h(x) \in \hat F$. Hence $\hat F$ is a witness set for $\frb(1/a)$.
\end{proof}

%\begin{thm} For any $p, q$ with  $0< p < q < 0.5$ we have $\frd(p)= \frd(q)$. \end{thm}
\begin{thm} Fix any  $p\in (0,1/2)$. We have $\frd(p)= \frd(\neq^*, \tp{(\tp n)})$ and 

\n $\frb(p)= \frb(\neq^*, \tp{(\tp n)})$. \end{thm}
\begin{proof} By the two foregoing lemmas we have $  \frd(p)  \ge \frd(\neq^*, \tp{(\tp n)}) $ and $\frb(p)\le  \frb(\neq^*, \tp{(\tp n)})$. It remains to show the converse inequalities:

\n   $\frd(p) \le  \frd(\neq^*, \tp{(\tp n)})$ and $\frb(p)\ge  \frb(\neq^*, \tp{(\tp n)})$.

\begin{definition} For strings $x,y$ of length $r$, the normalised Hamming distance is defined as the proportion of bits on which $x,y$ disagree, that is, 
\[ d(x,y) = \frac 1 r |\{ i \colon x(n) (i)    \neq y(n)(i) \}| \] \end{definition}
%%%%%%%%%%%%%
\begin{definition} Let $h$ be a function of the form $2^{\hat h}$ with $\hat h \colon \, \omega \to \omega$, and let $X=Y= X_h$ be the space of $h$-bounded functions.  Let $q \in (0, 1/2)$. 
 We   define a relation on $X \times Y$ by 
    \[x \neq^*_{\hat h,q} y \LR \fa^\infty n  \, [ d(x(n), y(n)) \ge q]   \]   
 namely for a.e. $n$ the   strings $x(n)$ and $y(n)$  disagree on a proportion of at least  $q$ of the bits. We will usually write $\la \neq^*, \hat h,q \ra $ for this  relation.  \end{definition}
 \begin{claim} \label{cl:1} For each $c\in \omega$ there is $k \in \omega$ such that \bc $\frd(q- 2/c) \le \frd \la \neq^*, \lfloor \tp{n/k} \rfloor, q\ra$, and    \ec 
 \bc $\frb(q- 2/c) \ge \frb \la \neq^*, \lfloor \tp{n/k} \rfloor, q\ra$. \ec \end{claim}
 \n To see this, let $k $ be large enough so that $\tp{1/k}-1 < \frac 1 {2c}$. Let  $\hat h(n) = \lfloor \tp{n/k} \rfloor$ and  $h = \tp{\hat h}$. Write $H(n) = \sum_{r\le  n} \hat h(r)$. Given an infinite  bit sequence, we refer to the bits with position in an interval $[H(n), H(n+1)) $ as \emph{Block}~$n$ (the first block is Block~0).  By Monin's 2016 logic blog entry, for sufficiently large $n$,  
 \[  \hat h(n+1) \le \frac 1 c H(n). \]  
 For the   inequality involving $\frd$, let $G$ be a witness set for $\frd \la \neq^*, \hat h , q\ra$.  Thus, for each function $x < h$ there is a function $y \in G$ such that for almost all $n$,  $L_h(x), L_h(y)$ disagree on a proportion of $q$ bits of Block $n$. Let $z$ be the complement of $L_h(y)$.   Given $m$, let $n$ be such that $H(n) \le m < H(n+1)$. Since $m- H(n) \le \frac 1 c H(n)$, for large enough $m$, $L_{  h}(x) $ and $z$ agree up to $m$ on a proportion of at least $q - 1.5/c$ bits. So  the set of complements of the $L_h(y)$, $y \in G$, forms a witness set for $\frd(q- 2/c)$ as required.

 %%%%%%%%%%%%%%
For the   inequality involving $\frb$, let $F$ be a witness set for  $\frb(q-2/c)$. Thus, for each $y \in \cantor$ there is $x \in F$ such that $\ul \rho(y \lra x) \le q-2/c$. Let $\hat F = \{K_h(1 -x) \colon \, x \in F \}$. We show that $\hat F$ is a witness set for $\frb \la \neq^*, \lfloor \tp{n/k} \rfloor, q\ra$. 

Give a function $y < h$,  let $y' = L_h(y)$.  There is $x \in F$ such that  $\ul {\rho} (y' \lra x) \le q - 2/c$, and hence $\ol  \rho (y' \lra  x' ) \ge 1- q+ 2/c$ where $x' = 1 -x$ is the complement and $\ol  \rho$ denotes the upper density.  Then there are infinitely many $m$ such that the strings $y'\uhr m$ and $x'\uhr m$ agree on a proportion of $> q+1/c$ bits.
Suppose that  $H(n) \le m < H(n+1)$ then the contribution  of disagreement of   Block  $n$ is at most $1/c$. So 
  there are infinitely many $k$ so that in Block $k$, $y'$ and $x'$ agree on  a proportion of more than $1-q$ bits, and hence disagree on a proportion of fewer than $q$ bits.  
 This shows the claim. 
 
 As in Monin's entry we use the list decoding capacity theorem from the theory of error-correcting codes.  Given $q $ as above  and $L \in \omega$, for each $r$ there is a ``fairly large" set $C$ of strings of length $r$ (the  allowed code words) such that for each string, at most $L$ strings in $C$ have normalised Hamming distance less than $q$ from $\sss$. (Hence there is only a small set of strings that could be the error-corrected  version of $\sss$.) Given string $\sss$ of length $r$, let $B_q(\sss)$ denote  an  open  ball around $\sss$ in the normalised Hamming distance, namely, $B_q(\sss) = \{ \tau \in  {}^r 2 \colon \, \sigma, \tau \text{  disagree on fewer than $qr$ bits}\}$
 \begin{lemma}[List decoding] \label{lem:ListDec} Let $q\in (0,1/2)$.  There are $\epsilon >0$ and $L \in \omega$ such that for each $r$, there is a set $C$ of $\tp{\lfloor \epsilon r \rfloor }$ strings of length $r$ as follows: \bc $\forall \sss \in {}^r 2 \, [ | B_q(\sss) \cap C | \le L]$.  \ec
 \end{lemma}
%%%%%%
For $L \in \omega$, an $L$-slalom is a function $s\colon \, \omega \to \omega^{[\le L]}$, i.e.\ a function that  maps  natural numbers to sets  of natural numbers with a  size of at most $L$. 
\begin{definition}  Fix  a function $u \colon \omega \to \omega$ and $L \in \omega$. Let $X$ be the space of $L$-slaloms $s$ such that $\max s(n) < u(n)$ for each $n$, and let $Y$  be the set of functions such that  $y(n)  <  u (n)$ for each $n$. Define a relation on $X \times Y$ by
 \[s \not \ni^*_{u,L} y \LR \fa^\infty n [s(n) \not \ni y(n)]. \] 
  We will   write $\la \not \ni^*, u ,L \ra $ for this  relation. 
\end{definition} 

\begin{claim} \label{cl:2} Given $q < 1/2$, let $L, \epsilon $ be as in Lemma~\ref{lem:ListDec}. Fix  a nondecreasing  function $\hat h$, and   let  $u(n)= \tp{\lfloor \epsilon \hat h(n) \rfloor }$. We have 
\[ \frd \la \neq^*,  \hat h , q\ra \le \frd \la \not \ni^*, u,L \ra \text{ and } \frb \la \neq^*,  \hat h , q\ra \ge \frb \la \not \ni^*, u,L \ra. \]
  \end{claim}
For  the   inequality involving $\frd$, let $G$ be a set of functions bounded by $u$ such that $|G| < \frd \la \neq^*,  \hat h , q\ra$. We show that $G$ is not a witness set  for  the right hand side $\frd \la \not \ni^*, u,L \ra$. 

For each $r$ of the form $\hat h(n)$ choose a set $C= C_r$ as in Lemma~\ref{lem:ListDec}. Since $|C_r| = \tp{\lfloor \epsilon r \rfloor}$  we may choose a sequence $\seq {\sss^r_i}_{i<  \tp{\lfloor \epsilon r \rfloor}}$ listing $C_r$ without repetitions. For a  function $y <  u$ let $\wt y$ be the function given by $\wt y(n)= \sss^{\hat h(n)}_{y(n)}$. Thus $\wt y(n)$ is a binary string of length $\hat h(n)$. Let $\wt G = \{ \wt y  \colon \, y \in G \}$. Then  $|\wt G| = |G| < \frd \la \neq^*,  \hat h , q\ra $. So there is a function $x$ with $x(n) \in {}^{\hat h(n)}2$ for each $n$ such that for each $y \in \hat G$ we have  $\ex^\infty n \, [d(x(n), y(n)) < q]$. Let $s$ be the slalom given by 
\[ s(n) = \{ i \colon \, d (x(n), \sss^{\hat h(n)}_i )< q  \}.\]
Note that by the choice of the $C_r$ according to Lemma~\ref{lem:ListDec},  $s$ is an $L$-slalom.   By the definitions, for each $y \in G$ we have $\ex^\infty n \, [s(n) \ni y(n)]$. So $G$ is not a witness set for $\frd \la \not \ni^*, u,L \ra$.

For  the   inequality involving $\frb$,  suppose $F$ is a witness set  for $\frb \la \neq^*,  \hat h , q\ra$. That is, for each $h= \tp{\hat h}$-bounded function $y$, there is $x\in F$ such that 
\bc $\ex^\infty n \, [ d(x(n), y(n) <q]$ \ec (as usual we view $x(n), y(n)$ as  binary strings of length $\hat h(n)$).
For $x\in F$ let $s_x$ be the  $L$-slalom such that
 \bc  $s_x(n) = \{i < u(n) \colon \, d(\sss^{\hat h(n)}_i, x(n) )< q\}$. \ec
 Let $\hat F = \{s_x \colon \, x \in F\}$. Given an $u$-bounded function $y$, let $y'(n) =  \sss^{\hat h(n)}_{y(n)}$. There is $x \in F$ such that      $d(x(n), y'(n) < q$ for     infinitely many $n$. This means that $y(n) \in s_x(n)$. Hence $\hat F$ is a witness set for 
$\frb \la \not \ni^*, u,L \ra$. This shows the claim.

% be the string of length $\hat h(n)$ which is the complement  of $

We next need an amplification tool   in the context of slaloms. The proof is almost verbatim the one in Lemma~\ref{lem:double}(i), so we omit it.

\begin{claim} \label{lem:double2}     Let $ L \in \omega$, let the function $u$ be nondecreasing and let $w(n) = u(2n)$. We have $\frd(\la \not \ni^*,u, L \ra = \frd\la \not \ni^*,w,L\ra$ and $\frb(\la \not \ni^*,u, L \ra = \frb\la \not \ni^*,w,L\ra$.
 \end{claim}
Iterating the claim,  starting with the function $\hat h(n) = \lfloor \tp{n/k}\rfloor$ with $k$ as in Claim~\ref{cl:1},  we   obtain that $\frd \la \not \ni^*,\tp{\hat h}, L \ra=  \frd \la \not \ni^*,\tp{(L 2^n)}, L \ra$, and similarly for~$\frb$. 
It remains to verify the following.
\begin{claim} \label{cl:final} $\frd \la \not \ni^*,\tp{(L 2^n)}, L \ra \le \frd(\neq^*, \tp{(\tp n)})$ and 

\n  $\frb \la \not \ni^*,\tp{(L 2^n)}, L \ra \ge \frb(\neq^*, \tp{(\tp n)})$. \end{claim}
 Given  $n$, we write  a number $k< \tp{(L \tp n)}$ in binary with leading zeros if necessary, and so can view $k$ as a binary string of length $L \tp n$. We view such a string as    consisting  of $L$ consecutive blocks of length  $\tp n$. 

For the inequality involving $\frd$,
let $G$ be a witness set for   $\frd(\neq^*, \tp{(\tp n)})$.  For   functions $y_1, \ldots , y_L$  such that  $y_i(n) < \tp{( \tp n)}$ for each $n$, let $(y_1, \ldots, y_L)$ denote the function $y$ with $y(n) <  \tp{(L \tp n)} $ for each $n$  such that  the $i$-th block of $y(n)$ equals $y_i(n)$ for  each   $i$ with  $1 \le i \le L$.  Let 
\bc $\hat G = \{ (y_1, \ldots, y_n) \colon \, y_1, \ldots, y_L \in G\}$. \ec
Since  $G $ is infinite we have $|\hat G| = |G|$. We check that  $\hat G$ is a witness set for  the left hand side $\frd \la \not \ni^*,\tp{(L 2^n)}\ra$. Given an $L$-slalom $s$ bounded by $\tp{(L \tp n)} $ we may assume that $s(n) $ has exactly $L$ members, and they  are binary  strings of length $L \tp n$. For $i \le L$ let $x_i(n) $ be the $i$-th block of the $i$-th string in $s(n) $,  so that $|x_i(n)|  =  \tp n$. Viewing  the $x_i$ as  functions  bounded by $\tp{(\tp n)}$,  we can choose $y_1, \ldots, y_L \in G$ such that $\fa^\infty n \, x_i(n) \neq y_i(n)$. Let $y = (y_1, \ldots, y_n) \in \hat G$. Then $\fa^\infty n \, [s(n) \ni y(n)]$,  as required. 

For the inequality involving $\frb$ let $F$ be a witness set for $\frb \la \not \ni^*,\tp{(L 2^n)}$. That is,  $F$ is a set of $L$-slaloms $s$ such that for each function $y$ with $y(n) < 2^{(L 2^n)}$, there is $s\in F$ such that $s(n) \ni y(n)$ for infinitely many $n$.  

Let $\hat F$ be the set of functions $s_i$, for $s\in F$ and  $i<L$, such that $s_i(n) $ is the $i$-th block of the $i$-th element of $s(n)$ (as before we may assume that each string in $s(n)$ has length $L2^n$). Now let $y$ be a given function bounded by $2^{(2^n)}$. Let $y'$ be the function bounded by $2^{(L2^n)}$ such that for each $n$, each block of $y'(n)$ equals $y(n)$. There is $s\in F$ such that $s(n) \ni y'(n)$ for infinitely many $n$. There is $i<L$ such that , $y'(n)$ is the $i$-th string in $s(n)$ for infinitely many of these $n$, and hence  $y(n) = s_i(n)$. Thus  $\hat F$ is a witness set for $\frb(\neq^*, \tp{(\tp n)})$.
This proves the claim. 

We can now summarise the argument  that  $\frd(p) \le  \frd(\neq^*, \tp{(\tp n)})$. Pick $c$ large enough such that $q= p+2/c < 1/2$.  

By Claim~\ref{cl:1} there is $k$ such that \bc $\frd(p) \le \frd \la \neq^*, \lfloor \tp{n/k} \rfloor, q\ra$.  \ec

 By Claim~\ref{cl:2} there are $L$, $\epsilon$ such that where $\hat h(n) =  \lfloor \tp{n/k} \rfloor$, we have \bc  $\frd \la \neq^*,  \hat h , q\ra \le \frd \la \not \ni^*, u,L \ra$,   \ec 
where $u(n)= \tp{\lfloor \epsilon \hat h(n) \rfloor } $.

Applying Claim~\ref{lem:double2}  sufficiently many times we  have \bc $\frd \la \not \ni^*, u,L \ra \le \frd \la \not \ni^*, \tp{(L 2^n)} ,L \ra$.  \ec 
\n Finally, $\frd \la \not \ni^*,\tp{(L 2^n)}, L \ra \le \frd(\neq^*, \tp{(\tp n)})$ by Claim~\ref{cl:final}. The argument for $\frb(p)\ge  \frb(\neq^*, \tp{(\tp n)})$ is the exact dual. 
\end{proof}

\section{Nies:  answering the $\Delta$ question }
\label{s:Nies_Delta}

We use the notation in  \cite{Brendle.Brooke.ea:14}. Let  $R \subseteq  X \times Y$ be a relation between spaces $X,Y$, and  let $S = \{  \langle y,x \rangle \in Y \times X \colon \neg xR y\}$.
Suppose we have specified what it means for     objects $x$ in $X$,    $y$ in  $Y$  to be computable in a Turing oracle $A$. We denote this by for example $x\leT A$. In particular, for $A= \emptyset$ we have a notion of computable objects.

Let the variable $x$ range over $X$, and let $y$ range over $Y$. We define the highness properties

\[ \+ B(R) =   \{ A:  \,\exists y \leT A \, \fa x \ \text{computable} \ [xRy]     \} \]

\[ \+ D(R) =  \+ B(S) =   \{ A:  \,\exists x \leT A \, \fa y \ \text{computable} \ [\neg xRy]   \}  \]

\n {\it 1.}   Let $h \colon \omega \to \omega - \{0,1\} $.   Define  for $  x \in {}^\omega\omega$ and 

\n $  y \in \prod_n \{0, \ldots, h(n)-1\} \sub  {}^\omega\omega$,   

\[ x \neq^*_h  y \LR \fa^\infty n \,  [ x(n) \neq y(n)]. \]

\n {\it 2.} Let $ 0 \le p \le 1/2$. Define,   for $x,y \in {}^\omega 2$

\[ x \sim_p y \LR   \underline \rho (x \lra y)  >p, \]
where $x \lra y$ is  the set of $n$ such that $x(n) = y(n)$, and $\underline \rho$ denotes the lower density: $\ul \rho (z) = \liminf_n |z \cap n|/n$.

It may be helpful to separately state two special cases.
\begin{definition}  For a computable function $h$, we let $\+ B(\neq^*_h)$ denote  the class of oracles $A$ that compute a function $y <  h$  such that   for each computable function $x$, we have     $\fa^\infty n \, x(n) \neq y(n)$.  

For $p< 1/2$, we let $\+ B(\sim_p)$, or $\+ B(p)$ for short,  denote  the class of oracles $A$ that compute a set $y$  such that   for each computable  set $x$, we have $\ul \rho (x \lra y) > p$.  \end{definition}

Recall Definition~\ref{def:delta} and  that  for each $p$ with $0 \le p < 1/2$,  
\bc  $\Delta(A) > p \RA   A \in \+ B(\sim_p) \RA  \Delta(A) \ge p$. \ec

We show that   $\+ B(\sim_p)= \+ B(\neq^*, \tp{(\tp n)})$  for each  $p\in (0,1/2)$. In particular, $\Delta(A) >0 \RA \Delta(A) = 1/2$ so there are only two possible $\Delta $ values.

 We begin with some preliminary facts of independent interest.
 On occasion we denote a function $\lambda n. f(n)$ simply by $f(n)$. 
 \begin{lemma} \label{lem:doubleR}  \bi \item[(i)] Let $h$ be nondecreasing and $g(n) = h(2n)$. We have   $\+ B(\neq^*,h) = \+ B(\neq^*,g)$.

\item [(ii)] For each $a,b >1$ we have   $\+ B(\neq^*,{\tp{(a^n)})}= \+ B(\neq^*,{\tp{(b^n)}})$. \ei \end{lemma}

\begin{proof} (i)  Trivially, $h \le g$ implies that       $\+ B(\neq^*,h) \sub \+ B(\neq^*,g)$. So it suffices to show the converse inclusion $\+ B(\neq^*,h) \supseteq \+ B(\neq^*,g)$.

Let $y\leT A$, $y < g$ be a function witnessing that $A \in B(\neq^*,g)$. Let $\hat y(2n+i) = y(n)$ for $i \le 1$, so that $\hat y < h$. Given any computable function $x$, for almost every $n$ we have $x(2n) \neq y(n) $ and $x(2n+1) \neq y(n)$.  Therefore $\fa^\infty n \, [ x(n)  \neq \hat y(n)]$. Hence $\hat y$ is a witness for $A \in B(\neq^*,h)$.

(ii) is immediate from (i) by iteration using that   $a^{2^i} >b$  and $b^{2^i} >a$ for sufficiently large $i$. 
\end{proof}
 Recall Def.\ \ref{df:KhLh} of the $L_h$ and $K_h$ operators.
\begin{lemma} Let $a \in \omega - \{0\}$. We have  $\+ B( \neq^*,{\tp{(a^n)}}) \supseteq  \+ B(1/a)$.  \end{lemma}
\begin{proof} Let $I_m$ for $m \ge 2$  be the $m-1$-th consecutive interval of length $a^m$ in $\omega-\{0\}$, i.e.

\[ I_m = \big [\frac{a^m-1}{a-1} , \frac{a^{m+1} -1 } {a-1}\big). \]
 
Let $h(m) = \tp{a^m}$.  Let $y \leT A$ witness that $A \in \+ B(1/a)$, and let $\hat y = K_h(y)$. Given a computable function $x < h$, let $x' = 1- L_h(x)$. Since $\ul \rho( x' \lra y) > 1/a$, for large enough $n$, there is $k \in I_n$ such that $x'(i) = y(i)$. Hence we cannot have $x(n) = \hat y(n)$.  Thus $\hat y$   witnesses that $A \in \+ B( \neq^*,{\tp{(a^n)}})$.  
 \end{proof}

A similar argument shows that $\+ B( \neq^*,{\tp{\hat h(m)}}) \supseteq  \+ B(0)$ for any computable function $\hat h$ such that $\fa a \fa^\infty m \,  \hat h(m) \ge a^m$.   Now the $m$-th interval has length $\hat h (m)$.

%\begin{thm} For any $p, q$ with  $0< p < q < 0.5$ we have $\frd(p)= \frd(q)$. \end{thm}
\begin{thm} \label{thm:mainDelta} Fix any  $p\in (0,1/2)$. We have $\+ B(p)= \+ B(\neq^*, \tp{(\tp n)})$. \end{thm}
\begin{proof} The two foregoing lemmas imply  $\+ B(p)\sub  \+ B(\neq^*, \tp{(\tp n)})$. It remains to show the converse inclusion  $\+ B(p)\supseteq   \+ B(\neq^*, \tp{(\tp n)})$.

\begin{definition} For strings $x,y$ of length $r$, the normalised Hamming distance is defined as the proportion of bits on which $x,y$ disagree, that is, 
\[ d(x,y) = \frac 1 r |\{ i \colon x(n) (i)    \neq y(n)(i) \}| \] \end{definition}
%%%%%%%%%%%%%
\begin{definition} Let $h$ be a function of the form $2^{\hat h}$ with $\hat h \colon \, \omega \to \omega$, and let $X=Y= X_h$ be the space of $h$-bounded functions.  Let $q \in (0, 1/2)$. 
 We   define a relation on $X \times Y$ by 
    \[x \neq^*_{\hat h,q} y \LR \fa^\infty n  \, [ d(x(n), y(n)) \ge q]   \]   
 namely for a.e. $n$ the   strings $x(n)$ and $y(n)$  disagree on a proportion of at least  $q$ of the bits. We will usually write $\la \neq^*, \hat h,q \ra $ for this  relation.  \end{definition}
 The first step makes the crucial transition from the density setting to the combinatorial setting.
 \begin{claim} \label{cl:AR} Let $q \in (0, 1/2)$. For each $c\in \omega$ such that $2/c<q$,  there is $k \in \omega$ such that   \bc $\+ B(q- 2/c) \supseteq \+ B \la \neq^*, \lfloor \tp{n/k} \rfloor, q\ra$. \ec  \end{claim}
 \n To see this, let $k $ be large enough so that $\tp{1/k}-1 < \frac 1 {2c}$. Let  $\hat h(n) = \lfloor \tp{n/k} \rfloor$ and  $h = \tp{\hat h}$. Write $H(n) = \sum_{r\le  n} \hat h(r)$. Given an infinite  bit sequence, we refer to the bits with position in an interval $[H(n), H(n+1)) $ as \emph{Block}~$n$ (the first block is Block~0).  By Monin's 2016 logic blog entry, for sufficiently large $n$,  
 \[  \hat h(n+1) \le \frac 1 c H(n). \]  
 %
  
 %%%%%%%%%%%%%%
Let $ y \leT A$  be a witness   for  $A \in \+ B \la \neq^*, \lfloor \tp{n/k} \rfloor, q\ra$. Thus, $y< h$ and  $\fa^\infty n \, [d(L_h(x)(n), L_h(y)(n) \ge q]$. Let $y' = 1  -  L_h(y)$. Then  \bc $\fa^\infty n \, [d(x'(n), y'(n) )<  q]$ \ec for each computable   $x' \in \cantor$.  Suppose that  $H(n) \le m < H(n+1)$.  Then the contribution of   Block  $n$ to the proportion of  disagreement between $x', y'$  is at most $1/c$. 
So $\ul \rho (y' \lra x') > q-2/c$. Hence $y'\leT A$  is a witness for  $A\in \+ B(q-2/c)$. 
  This shows the claim. 
 
% As in Monin's entry we use the list decoding capacity theorem from the theory of error-correcting codes.  Given $q $ as above  and $L \in \omega$, for each $r$ there is a ``fairly large" set $C$ of strings of length $r$ (the  allowed code words) such that for each string, at most $L$ strings in $C$ have normalised Hamming distance less than $q$ from $\sss$. (Hence there is only a small set of strings that could be the error-corrected  version of $\sss$.) Given string $\sss$ of length $r$, let $B_q(\sss)$ denote  an  open  ball around $\sss$ in the normalised Hamming distance, namely, $B_q(\sss) = \{ \tau \in  {}^r 2 \colon \, \sigma, \tau \text{  disagree on fewer than $qr$ bits}\}$
% \begin{lemma}[List decoding] \label{lem:ListDec} Let $q\in (0,1/2)$.  There are $\epsilon >0$ and $L \in \omega$ such that for each $r$, there is a set $C$ of $\tp{\lfloor \epsilon r \rfloor }$ strings of length $r$ as follows: \bc $\forall \sss \in {}^r 2 \, [ | B_q(\sss) \cap C | \le L]$.  \ec
% \end{lemma}
%%%%%%
For $L \in \omega$, an $L$-slalom is a function $s\colon \, \omega \to \omega^{[\le L]}$, i.e.\ a function that  maps  natural numbers to sets  of natural numbers with a  size of at most $L$. 
\begin{definition}  Fix  a function $u \colon \omega \to \omega$ and $L \in \omega$. Let $X$ be the space of $L$-slaloms (or  traces)   $s$ such that $\max s(n) < u(n)$ for each $n$. Thus $s$ maps natural numbers to sets of natural numbers of size at most $L$, represented by strong indices. Let $Y$  be the set of functions such that  $y(n)  <  u (n)$ for each $n$. Define a relation on $X \times Y$ by
 \[s \not \ni^*_{u,L} y \LR \fa^\infty n [s(n) \not \ni y(n)]. \] 
  We will   write $\la \not \ni^*, u ,L \ra $ for this  relation. 
\end{definition} 

\begin{claim} \label{cl:BR} Given $q < 1/2$, let $L, \epsilon $ be as in Lemma~\ref{lem:ListDec}. Fix  a nondecreasing  computable function $\hat h$, and   let  $u(n)= \tp{\lfloor \epsilon \hat h(n) \rfloor }$. We have \bc $ \+ B \la \neq^*,  \hat h , q\ra \supseteq \+ B \la \not \ni^*, u,L \ra$. \ec 
  \end{claim}
%For  the   inequality involving $\frd$, let $G$ be a set of functions bounded by $u$ such that $|G| < \frd \la \neq^*,  \hat h , q\ra$. We show that $G$ is not a witness set  for  the right hand side $\frd \la \not \ni^*, u,L \ra$. 
%
For each $r$ of the form $\hat h(n)$ compute a set $C= C_r$ as in Lemma~\ref{lem:ListDec}. Since $|C_r| = \tp{\lfloor \epsilon r \rfloor}$  there is  a uniformly computable sequence $\seq {\sss^r_i}_{i<  \tp{\lfloor \epsilon r \rfloor}}$ listing $C_r$ in increasing lexicographical order. 
 
 %Let $\wt G = \{ \wt y  \colon \, y \in G \}$. Then  $|\wt G| = |G| < \frd \la \neq^*,  \hat h , q\ra $. So there is a function $x$ with $x(n) \in {}^{\hat h(n)}2$ for each $n$ such that for each $y \in \hat G$ we have  $\ex^\infty n \, [d(x(n), y(n)) < q]$. Let $s$ be the slalom given by 
%\[ s(n) = \{ i \colon \, d (x(n), \sss^{\hat h(n)}_i )< q  \}.\]
%Note that by the choice of the $C_r$ according to Lemma~\ref{lem:ListDec},  $s$ is an $L$-slalom.   By the definitions, for each $y \in G$ we have $\ex^\infty n \, [s(n) \ni y(n)]$. So $G$ is not a witness set for $\frd \la \not \ni^*, u,L \ra$. 
%

%For  the   inequality involving $\+ B$,  s

Suppose that $y \leT A$ is a witness   for 
$A \in \+ B \la \not \ni^*, u,L \ra$.  Let   $\wt y < h $ be the function given by $\wt y(n)= \sss^{\hat h(n)}_{y(n)}$. 
We show that $\wt y$  is a witness for  

\n  $A\in \+ B \la \neq^*,  \hat h , q\ra$. 

For a computable function $x < h$,  let $s_x$ be the computable  $L$-trace   such that
 \bc  $s_x(n) = \{i < u(n) \colon \, d(\sss^{\hat h(n)}_i, x(n) )< q\}$. \ec 

 Since $y$ is a witness for $A \in \+ B \la \not \ni^*, u,L \ra$, for almost every $n$ we have $y(n) \not \in s_x(n)$. Hence $d(\wt y(n),x(n) \ge q$, as required.

We next need an amplification tool   in the context of traces. The proof is almost verbatim the one in Lemma~\ref{lem:doubleR}(i), so we omit it.

\begin{claim} \label{lem:doubleR2}     Let $ L \in \omega$, let the  computable function $u$ be nondecreasing and let $w(n) = u(2n)$. We have  $\+ B(\la \not \ni^*,u, L \ra = \+ B\la \not \ni^*,w,L\ra$.
 \end{claim}
Iterating the claim,  starting with the function $\hat h(n) = \lfloor \tp{n/k}\rfloor$ with $k$ as in Claim~\ref{cl:AR},  we   obtain that $\+ B\la \not \ni^*,\tp{\hat h}, L \ra=  \+ B \la \not \ni^*,\tp{(L 2^n)}, L \ra$. 
It remains to verify the following, which would work for any computable function $\hat h(n)$ in place of the  $2^n$ in the exponents.
\begin{claim} \label{cl:finalR}    $\+ B \la \not \ni^*,\tp{(L 2^n)}, L \ra \supseteq \+ B(\neq^*, \tp{(\tp n)})$. \end{claim}
 Given  $n$, we write  a number $k< \tp{(L \tp n)}$ in binary with leading zeros if necessary, and so can view $k$ as a binary string of length $L \tp n$. We view such a string as    consisting  of $L$ consecutive blocks of length  $\tp n$.

Let $y$ be a witness  function for $\+ B(\neq^*, \tp{(\tp n)})$.  That is, $y< 2^{(2^n)}$ and  $\fa^\infty n \, x(n) \neq y(n)$ for each computable function $x$.    Let $y'$ be the function bounded by $2^{(L2^n)}$ such that for each $n$, each block of $y'(n)$ equals $y(n)$.  
Given a computable $L$-trace $s$ with $\max s(n) <  \tp{(L 2^n)}$, for $i< L$ let $x_i$ be the computable function  such that $x_i(n) $ is the $i$-th block of the $i$-th element of $s(n)$ (as before we may assume that each string in $s(n)$ has length $L2^n$). For sufficiently large $n$,  we have $\fa i < L \, y(n) \neq x_i(n)$.  Hence $\fa^\infty n \, s(n) \not \ni y'(n)$ and $y'$ is a witness for $A\in \+ B \la \not \ni^*,\tp{(L 2^n)}$.

We can now summarise the argument  that  $\+ B(p) \supseteq \+ B(\neq^*, \tp{(\tp n)})$. Pick $c$ large enough such that $q= p+2/c < 1/2$.  

By Claim~\ref{cl:AR} there is $k$ such that \bc $\+ B(p) \supseteq\+ B \la \neq^*, \lfloor \tp{n/k} \rfloor, q\ra$.  \ec

 By Claim~\ref{cl:BR} there are $L$, $\epsilon$ such that where $\hat h(n) =  \lfloor \tp{n/k} \rfloor$ and $u(n)= \tp{\lfloor \epsilon \hat h(n) \rfloor }$, we have \bc  $\+ B \la \neq^*,  \hat h , q\ra \supseteq\+ B \la \not \ni^*, u,L \ra$.   \ec

Applying Claim~\ref{lem:doubleR2}  sufficiently many times we  have \bc $\+ B \la \not \ni^*, u,L \ra \supseteq\+ B \la \not \ni^*, \tp{(L 2^n)} ,L \ra$.  \ec 
\n Finally, $\+ B \la \not \ni^*,\tp{(L 2^n)}, L \ra \supseteq\+ B(\neq^*, \tp{(\tp n)})$ by Claim~\ref{cl:finalR}. \end{proof}
We note that by the proofs,  both inclusions in Theorem~\ref{thm:mainDelta} are uniform in a strong sense: from a  witness $y \leT A$ for one property one can compute a witness $y'$ for the other property.

As a  corollary to Theorem~\ref{thm:mainDelta}  we obtain a proof of Proposition~\ref{prop:2gen}. If $A$ is 2-generic then $A$ is neither high nor d.n.c., so $A$ is not in $\+ B(\neq^*) = \text{high or d.n.c.}$ the class where no bound is imposed on the witness  function $A$ computes. So $A$ is not in $B(\neq^*, \tp{(\tp n)})$, hence $\Delta(A)=0$.

 %%%%%%%%%%%%%%%%%%
  \part{Reverse mathematics}
 \section{Belanger,  Nies and Shafer: the  strength of randomness existence axioms} 
 
 David Belanger, Andr\'e Nies and Paul Shafer others discussed the strength of randomness existence axioms at NUS and University of Ghent in February/March.
 
In \cite[Section.\ 9]{LogicBlog:13},   for a randomness notion $\+ C$, we defined  $\+ C_0$ to be  the system $\mathtt{RCA} + \fa X \ex Y \, Y \in \+ C^X$. For instance, $\MLR_0$ is equivalent to weak weak K\"onigs Lemma over    $\mathtt{RCA}$.  
 
 The system $\mathtt{CR}_0$ (computable randomness)  appears to be   equivalent to the seemingly weaker $\mathtt{SR}_0$ (Schnorr randomness) (\cite[Prop.\ 9.2]{LogicBlog:13}; the strength of induction axioms  that are needed  to show this  remains to be checked carefully). 
 
 \subsection*{2-randomness versus weak 2-randomness}  On the other hand $\twoR_0$ (the system for 2-randomness) is strictly stronger than 
%:
 $\WtwoR_0$  (the system for weak 2 randomness). To see this, take a weakly 2-random $Z$ that does not compute a   2-random. For instance,   any 2-random has  hyperimmune degree. Any computably dominated ML-random $Z$ is weakly 2-random and hence does not compute a 2-random. For each $n$ let $Z_n$ be the $n$-th column of $Z$, that is,   $Z_n = \{ k \colon \, \la k,n \ra \in Z\}$. Let $\+ M = ( \NN, \+ S)$ where  $\+ S$ consists of all the sets Turing below the join of finitely many columns of $Z$.  Note that $Z_n$ is weakly 2-random in any finite sum of columns not containing $Z_n$. So $\+ M$ is a model of  $\WtwoR_0$.  
 
 We can also separate the two randomness existence axioms via an interesting mathematical  consequence:  Csima and Mileti~\cite{Csima.Mileti:09} have shown that $\twoR_0$ implies the Rainsey Rambo's theorem. 
 
 \begin{prop} $\WtwoR_0$ does not imply the Rainbow Ramsey's theorem. \end{prop}
 
 \begin{proof} Joseph Miller has shown that the Rainbow Ramsey's theorem is equivalent over $\mathtt{RCA}_0  $ plus  some induction to the existence of a d.n.c.\ function relative to $\Halt$.  (Detail needed here on the construction for the forward direction: recursive in  a homogeneous set we obtain the function.) By \cite[Exercise 4.3.18]{Nies:book} there is a weakly 2-random set $Z$ that does not compute a 2-fixed point free function, and hence it computes no d.n.c.\ function relative to $\Halt$. Construct the model $\+ M \models \WtwoR_0$ from $Z$ as above. Then Rainbow Ramsey's theorem fails in $\+ M$. 
 \end{proof}

 \subsection*{Weak Demuth randomness  versus $\mathtt{WKL}$}
\begin{definition} 
	A $\DII$ function $f$ is $\w$-c.e.\ if there is a computable function $h$ such that $h(n)$ bounds the number of changes in a computable approximation to $f(n)$. % (We also say that $f$ is $h$-c.e.\ 
	
	 A \emph{Demuth test} is an effective sequence $\seq{\+ U_n}$ of effectively open ($\Sigma^0_1$) subsets of Cantor space such that:
	\begin{enumerate}
		\item For all $n$, the measure $\lambda (\+ U_n)$ of $\+ U_n$ is bounded by $2^{-n}$; and
		\item there is an $\w$-c.e.\ function mapping $n$ to a $\Sigma^0_1$ index for $\+ U_n$.
	\end{enumerate}
	  A set (an element of Cantor space) $X$ is captured by a Demuth test $\seq{\+ U_n}$ if $X\in \+ U_n$ for infinitely many $n$. A set is \emph{Demuth random} if it is not captured by any Demuth test. 
	  
	   A set $Z$ \emph{weakly passes} a test $\seq{\+ U_n}$ if $Z \not \in \bigcap_n \+ U_n$. A set $Z$ is \emph{weakly Demuth random} if it weakly  passes every Demuth test.  

\end{definition}  
Figueira et al.\  \cite{Figueira.Hirschfeldt.ea:15} introduced balanced randomness, a notion in between weak Demuth randomness and ML-randomness where the $m$-th test component $\+ U_m$ of a test  can be ``replaced"   at most  $O(2^m)$ times. This notion is still stronger than Oberwolfach randomness.  More generally, for an order function $h$ we say that $Z$ is \emph{$h$- weak Demuth random}  if it weakly passes each  Demuth test where the component $\+ U_m$ can be replaced at most $h(m)$ times.
 
  It was observed in \cite{Bienvenu.Greenberg.ea:16} that the methods of \cite{Figueira.Hirschfeldt.ea:15} show the following.
 
 \begin{prop} Let $Z= Z_0 \oplus Z_1$ be ML-random. Then $Z_0$ or $Z_1$ is balanced random, and in fact $O(r(n)2^{n})$-weak Demuth random for some order function $r$. \end{prop} 
 
 \begin{proof}	 We use the terminology of \cite{Figueira.Hirschfeldt.ea:15}. If $Z_0$ is $\omega$-c.e.\ tracing then any weak Demuth test  (which is given by an $\omega$-c.e.\ function) can be converted into a ML-test relative to $Z_0$. So by van Lambalgen theorem, $Z_1$ is weak Demuth random. 
 
 If $Z_0$ is not  $\omega$-c.e.\ tracing then by \cite[Thm.\ 23]{Figueira.Hirschfeldt.ea:15} $Z_0$ is  $O(r(n)2^{n})$-weak Demuth random for some order function $r$, and in particular, balanced random. 
 \end{proof}
 
 So weak weak K\"onigs lemma plus sufficient induction implies the axiom for balanced randomness. 
 \begin{definition} For a computable function  $h$,  we say that a set $Z$ is $h$-c.e.\ if there is a computable approximation such that $Z\uhr n$ changes at most $h(n)$ times. For instance, each left-c.e.\ set is $2^n$-c.e.  \end{definition} 
Such a set is clearly not $h$-weak Demuth random.   So the following $\omega$-model $\+ M$ satisfies  $\mathtt{WKL}$ but not  the axiom for weak $h$-Demuth randomness, for any function $h$ that dominates each function $\lambda n . k^n$ , such as $h(n) = 2^{n \cdot p(n)}$ for some  order function~$p$. 
% \begin{prop} For an appropriate primitive recursive function $h$, there  is an $\omega$-model $\+ M$  of $\mathtt{WKL}$ such that each set of $\+ M$  is superlow and $h$-c.e. \end{prop} 
% 
% \begin{proof}	Let $\+ P^X$ be the $\PPI$ class relative to $X$ of sets that are PA-complete in $X$. Over  $\mathtt{RCA}$, the axiom  $\mathtt{WKL}$ is equivalent to $\fa X \, \+ P^X \neq \ES$. 
% 
% Let $K^Y$ denote the   universal c.e.\  operator $\{\la e, x \ra \colon x \in W^Y_e\}$. Let $W^X$ be the c.e.\ operator defined recursively as follows. \bi \item For $n=3k$ let $W^X(n) = X(k)$.\item  For $n=3k+1$ let $W^X(n)= 1$ iff $k \in K^X$. \item For $n=3k+2$ let $W^X(n) =1$ iff $Q_\tau =\ES$, where $\tau$ is the binary string such that $k+1$ has the binary expansion $1\tau$, and \bc $Q_\tau = \{ Z \in \+ P \colon \, \fa e. \tau(e) =0 [ e \not \in W^Z]\}$. \ec \ei
% Now for each $Y$ we can run the proof of the superlow basis theorem with the operator $W$ to obtain a set $S(Y)$ such that $W^{S(Y)}$ is left-c.e.\ relative to $Y$. Let $\+ M$ be the $\omega$ model with second-order part the Turing downward closure of $\{\ES, S(\ES), S(S(\ES)), \ldots \}$. 
% 
% \end{proof}
 \begin{prop} There  is an $\omega$-model $\+ M$  of $\mathtt{WKL}$ such that each set of $\+ M$  is superlow and $k^n$-c.e.\  for some $k \in \NN$.  \end{prop} 
 
 \begin{proof}	Let $\+ S(X)$ be the $\PPI$ class relative to $X$ of sets that are PA-complete in $X$. Over  $\mathtt{RCA}$, the axiom  $\mathtt{WKL}$ is equivalent to $\fa X \, \+ S(X) \neq \ES$.

 Let $\+ Q$ be the $\PPI$ class consisting of the sets $Y$ such that $Y_{i+1} \in \+ S({\bigoplus_{k \le i} Y_k})$ for each $i$; here $Y_i = \{ n \colon \, \la i,n\ra \in Y\}$ is the $i$-th column of $Y$. 
 If $Y \in \+ Q$ then the Turing ideal  generated by the columns of $Y$ determines an $\omega$-model $\+ M$  of $\mathtt{WKL}$.

 Let $X \to  W^X$ be the c.e.\ operator  such that 
 
 \bc $2^e (2n+1) \in W^X \LR  \Phi_e^X(n) = 1\}$. \ec
 
 By the superlow basis theorem as stated  in  \cite[1.8.38]{Nies:book}, but with the operator $W$ instead of the domain of the usual Turing jump $J$, there is $Y \in  \+ Q$ such that $W^Y$ is left-c.e.
 
 Suppose that $R = \Phi_e^Y$. Since $W^Y$ is left-c.e.,  $R$ is $2^{2^e (2n+1)}$-c.e., and hence $k^n$-c.e. where $k= 2^{2^{e+2}}$. Thus each set of $\+ M$ is $k^n$-c.e.\  for some $k \in \NN$. 
 
 Clearly  $Y' \le_m W^Y$. So $Y$ is superlow. 
 \end{proof}

 \section{Belanger,  Nies and Shafer: the  strength of randomness existence axioms} 
 
 David Belanger, Andr\'e Nies and Paul Shafer others discussed the strength of randomness existence axioms at NUS and University of Ghent in February/March.
 
In \cite[Section.\ 9]{LogicBlog:13},   for a randomness notion $\+ C$, we defined  $\+ C_0$ to be  the system $\mathtt{RCA} + \fa X \ex Y \, Y \in \+ C^X$. For instance, $\MLR_0$ is equivalent to weak weak K\"onigs Lemma over    $\mathtt{RCA}$.  
 
 The system $\mathtt{CR}_0$ (computable randomness)  appears to be   equivalent to the seemingly weaker $\mathtt{SR}_0$ (Schnorr randomness) (\cite[Prop.\ 9.2]{LogicBlog:13}; the strength of induction axioms  that are needed  to show this  remains to be checked carefully). 
 
 \subsection*{2-randomness versus weak 2-randomness}  On the other hand $\twoR_0$ (the system for 2-randomness) is strictly stronger than 
%:
 $\WtwoR_0$  (the system for weak 2 randomness). To see this, take a weakly 2-random $Z$ that does not compute a   2-random. For instance,   any 2-random has  hyperimmune degree. Any computably dominated ML-random $Z$ is weakly 2-random and hence does not compute a 2-random. For each $n$ let $Z_n$ be the $n$-th column of $Z$, that is,   $Z_n = \{ k \colon \, \la k,n \ra \in Z\}$. Let $\+ M = ( \NN, \+ S)$ where  $\+ S$ consists of all the sets Turing below the join of finitely many columns of $Z$.  Note that $Z_n$ is weakly 2-random in any finite sum of columns not containing $Z_n$. So $\+ M$ is a model of  $\WtwoR_0$.  
 
 We can also separate the two randomness existence axioms via an interesting mathematical  consequence:  Csima and Mileti~\cite{Csima.Mileti:09} have shown that $\twoR_0$ implies the Rainsey Rambo's theorem. 
 
 \begin{prop} $\WtwoR_0$ does not imply the Rainbow Ramsey's theorem. \end{prop}
 
 \begin{proof} Joseph Miller has shown that the Rainbow Ramsey's theorem is equivalent over $\mathtt{RCA}_0  $ plus  some induction to the existence of a d.n.c.\ function relative to $\Halt$.  (Detail needed here on the construction for the forward direction: recursive in  a homogeneous set we obtain the function.) By \cite[Exercise 4.3.18]{Nies:book} there is a weakly 2-random set $Z$ that does not compute a 2-fixed point free function, and hence it computes no d.n.c.\ function relative to $\Halt$. Construct the model $\+ M \models \WtwoR_0$ from $Z$ as above. Then Rainbow Ramsey's theorem fails in $\+ M$. 
 \end{proof}

 \subsection*{Weak Demuth randomness  versus $\mathtt{WKL}$}
\begin{definition} 
	A $\DII$ function $f$ is $\w$-c.e.\ if there is a computable function $h$ such that $h(n)$ bounds the number of changes in a computable approximation to $f(n)$. % (We also say that $f$ is $h$-c.e.\ 
	
	 A \emph{Demuth test} is an effective sequence $\seq{\+ U_n}$ of effectively open ($\Sigma^0_1$) subsets of Cantor space such that:
	\begin{enumerate}
		\item For all $n$, the measure $\lambda (\+ U_n)$ of $\+ U_n$ is bounded by $2^{-n}$; and
		\item there is an $\w$-c.e.\ function mapping $n$ to a $\Sigma^0_1$ index for $\+ U_n$.
	\end{enumerate}
	  A set (an element of Cantor space) $X$ is captured by a Demuth test $\seq{\+ U_n}$ if $X\in \+ U_n$ for infinitely many $n$. A set is \emph{Demuth random} if it is not captured by any Demuth test. 
	  
	   A set $Z$ \emph{weakly passes} a test $\seq{\+ U_n}$ if $Z \not \in \bigcap_n \+ U_n$. A set $Z$ is \emph{weakly Demuth random} if it weakly  passes every Demuth test.  

\end{definition}  
Figueira et al.\  \cite{Figueira.Hirschfeldt.ea:15} introduced balanced randomness, a notion in between weak Demuth randomness and ML-randomness where the $m$-th test component $\+ U_m$ of a test  can be ``replaced"   at most  $O(2^m)$ times. This notion is still stronger than Oberwolfach randomness.  More generally, for an order function $h$ we say that $Z$ is \emph{$h$- weak Demuth random}  if it weakly passes each  Demuth test where the component $\+ U_m$ can be replaced at most $h(m)$ times.
 
  It was observed in \cite{Bienvenu.Greenberg.ea:16} that the methods of \cite{Figueira.Hirschfeldt.ea:15} show the following.
 
 \begin{prop} Let $Z= Z_0 \oplus Z_1$ be ML-random. Then $Z_0$ or $Z_1$ is balanced random, and in fact $O(r(n)2^{n})$-weak Demuth random for some order function $r$. \end{prop} 
 
 \begin{proof}	 We use the terminology of \cite{Figueira.Hirschfeldt.ea:15}. If $Z_0$ is $\omega$-c.e.\ tracing then any weak Demuth test  (which is given by an $\omega$-c.e.\ function) can be converted into a ML-test relative to $Z_0$. So by van Lambalgen theorem, $Z_1$ is weak Demuth random. 
 
 If $Z_0$ is not  $\omega$-c.e.\ tracing then by \cite[Thm.\ 23]{Figueira.Hirschfeldt.ea:15} $Z_0$ is  $O(r(n)2^{n})$-weak Demuth random for some order function $r$, and in particular, balanced random. 
 \end{proof}
 
 So weak weak K\"onigs lemma plus sufficient induction implies the axiom for balanced randomness. 
 \begin{definition} For a computable function  $h$,  we say that a set $Z$ is $h$-c.e.\ if there is a computable approximation such that $Z\uhr n$ changes at most $h(n)$ times. For instance, each left-c.e.\ set is $2^n$-c.e.  \end{definition} 
Such a set is clearly not $h$-weak Demuth random.   So the following $\omega$-model $\+ M$ satisfies  $\mathtt{WKL}$ but not  the axiom for weak $h$-Demuth randomness, for any function $h$ that dominates each function $\lambda n . k^n$ , such as $h(n) = 2^{n \cdot p(n)}$ for some  order function~$p$. 
% \begin{prop} For an appropriate primitive recursive function $h$, there  is an $\omega$-model $\+ M$  of $\mathtt{WKL}$ such that each set of $\+ M$  is superlow and $h$-c.e. \end{prop} 
% 
% \begin{proof}	Let $\+ P^X$ be the $\PPI$ class relative to $X$ of sets that are PA-complete in $X$. Over  $\mathtt{RCA}$, the axiom  $\mathtt{WKL}$ is equivalent to $\fa X \, \+ P^X \neq \ES$. 
% 
% Let $K^Y$ denote the   universal c.e.\  operator $\{\la e, x \ra \colon x \in W^Y_e\}$. Let $W^X$ be the c.e.\ operator defined recursively as follows. \bi \item For $n=3k$ let $W^X(n) = X(k)$.\item  For $n=3k+1$ let $W^X(n)= 1$ iff $k \in K^X$. \item For $n=3k+2$ let $W^X(n) =1$ iff $Q_\tau =\ES$, where $\tau$ is the binary string such that $k+1$ has the binary expansion $1\tau$, and \bc $Q_\tau = \{ Z \in \+ P \colon \, \fa e. \tau(e) =0 [ e \not \in W^Z]\}$. \ec \ei
% Now for each $Y$ we can run the proof of the superlow basis theorem with the operator $W$ to obtain a set $S(Y)$ such that $W^{S(Y)}$ is left-c.e.\ relative to $Y$. Let $\+ M$ be the $\omega$ model with second-order part the Turing downward closure of $\{\ES, S(\ES), S(S(\ES)), \ldots \}$. 
% 
% \end{proof}
 \begin{prop} There  is an $\omega$-model $\+ M$  of $\mathtt{WKL}$ such that each set of $\+ M$  is superlow and $k^n$-c.e.\  for some $k \in \NN$.  \end{prop} 
 
 \begin{proof}	Let $\+ S(X)$ be the $\PPI$ class relative to $X$ of sets that are PA-complete in $X$. Over  $\mathtt{RCA}$, the axiom  $\mathtt{WKL}$ is equivalent to $\fa X \, \+ S(X) \neq \ES$.

 Let $\+ Q$ be the $\PPI$ class consisting of the sets $Y$ such that $Y_{i+1} \in \+ S({\bigoplus_{k \le i} Y_k})$ for each $i$; here $Y_i = \{ n \colon \, \la i,n\ra \in Y\}$ is the $i$-th column of $Y$. 
 If $Y \in \+ Q$ then the Turing ideal  generated by the columns of $Y$ determines an $\omega$-model $\+ M$  of $\mathtt{WKL}$.

 Let $X \to  W^X$ be the c.e.\ operator  such that 
 
 \bc $2^e (2n+1) \in W^X \LR  \Phi_e^X(n) = 1\}$. \ec
 
 By the superlow basis theorem as stated  in  \cite[1.8.38]{Nies:book}, but with the operator $W$ instead of the domain of the usual Turing jump $J$, there is $Y \in  \+ Q$ such that $W^Y$ is left-c.e.
 
 Suppose that $R = \Phi_e^Y$. Since $W^Y$ is left-c.e.,  $R$ is $2^{2^e (2n+1)}$-c.e., and hence $k^n$-c.e. where $k= 2^{2^{e+2}}$. Thus each set of $\+ M$ is $k^n$-c.e.\  for some $k \in \NN$. 
 
 Clearly  $Y' \le_m W^Y$. So $Y$ is superlow. 
 \end{proof}

\section{Carlucci: Bounded Hindman's Theorem and Increasing Polarized Ramsey's Theorem}
The following results were proved after reading the paper {\em Effectiveness of Hindman's Theorem
for bounded sums} by Dzhafarov, Jockusch, Solomon and Westrick, \cite{DJSW:16}. % (references are located at the end of this entry). 

The following natural restriction of Hindman's Theorem to sums with a bounded number of terms
was discussed by Blass in \cite{Bla:05} and first studied from a Reverse Mathematics perspective 
by Dzhafarov et alii in \cite{DJSW:16}.

\begin{definition}[Hindman's Theorem with bounded sums]
$\mathsf{HT}^{\leq n}_k$ states that for every coloring $f:\mathbb{N}\rightarrow k$ there exists an infinite set $H$ 
such that $FS^{\leq n}(H)$ is monochromatic for $f$, where $FS^{\leq n}(H)$ denotes the 
set of all finite sums of at most $n$ distinct members of $H$. 
We denote $\forall k \ \mathsf{HT}^{\leq n}_k$ by $\mathsf{HT}^{\leq n}$.
\end{definition}

The main results in \cite{DJSW:16} are (1) that $\mathsf{HT}^{\leq 2}_2$ is unprovable 
in $\mathsf{RCA}_0$ and (2) that $\mathsf{HT}^{\leq 3}_3$ proves $\mathsf{ACA}_0$ over $\mathsf{RCA}_0$. 

\bigskip
The following version of Ramsey's Theorem is introduced in \cite{Dza-Hir:11}.

\begin{definition}[Increasing Polarized Ramsey Theorem]
$\mathsf{IPT}^n_k$ is the following principle: for every $f:[\mathbb{N}]^n\rightarrow k$ there exists a sequence
$\langle H_1,\dots,H_n\rangle$ of infinite sets and  $c < k$ such that
for all increasing tuple $(x_1,\dots,x_n)\in H_1\times\dots\times H_n$ we have $f(x_1,\dots,x_n)=c$.
The sequence $\langle H_1,\dots,H_n\rangle$ is called increasing p-homogeneous for $f$.
We denote $\forall k \ \mathsf{IPT}^n_k$ by $\mathsf{IPT}^n$.
\end{definition}

We first prove that $\mathsf{HT}^{\leq 2}_5$ implies $\mathsf{IPT}^2_2$ over $\mathsf{RCA}_0$
by a direct combinatorial argument.

\begin{proposition}\label{prop:ht25_to_ipt22} 
$\mathsf{RCA}_0 \vdash \mathsf{HT}^{\leq 2}_5 \rightarrow \mathsf{IPT}^2_2$
\end{proposition}

\begin{proof}
We use the following notation: for $n\in \mathbb{N}$, 
if $n = i_0\cdot 3^{k_0} + \dots + i_t\cdot 3^{k_t}$ with $k_0 < \dots < k_t$
and $i_0,\dots,i_t \in \{1,2\}$,
we denote $k_0$ by $\lambda(n)$, 
$k_t$ by $\mu(n)$, and $i_0$ by $i(n)$. 
The following elementary properties hold (see \cite{DJSW:16}, Theorem 3.1):

(P1) If $\lambda(n) < \lambda(m)$ then $\lambda(n+m)=\lambda(n)$ and $i(n+m)=i(n)$.

(P2) If $\lambda(n)=\lambda(m)$ and $i(n)=i(m)=1$ then $\lambda(n+m)=\lambda(n)$ and $i(n+m)=2$.

(P3) If $\lambda(n)=\lambda(m)$ and $i(n)=i(m)=2$ then $\lambda(n+m)=\lambda(n)$ and $i(n+m)=1$.

(P4) If $\mu(n) < \lambda(m)$ then $\lambda(n+m)=\lambda(n)$ and $\mu(n+m)=\mu(m)$.

\bigskip
Let $f:[\mathbb{N}]^2\rightarrow 2$ be given. Define $g:\mathbb{N}\rightarrow 5$ as follows.

$$
g(n):=
\begin{cases}
0 & \mbox{ if  } n = i\cdot 3^t, i \in \{1,2\},\\
1+f(\lambda(n),\mu(n)) & \mbox{ if } n \neq i\cdot 3^t \wedge i(n)=1,\\
3+f(\lambda(n),\mu(n)) & \mbox{ if } n \neq i\cdot 3^t \wedge i(n)=2.\\
\end{cases}
$$
Note that $g$ is well-defined since $\lambda(n)<\mu(n)$ if $n$ is not of the form $i\cdot 3^t$, $i\in\{1,2\}$.

Let $H$ witnessing $\mathsf{HT}^{\leq 2}_5$ for $g$ be an infinite set such that $FS^{\leq 2}(H)$ is monochromatic under $g$. 

First, the color of $FS^{\leq 2}(H)$ cannot be $0$: the set $H$ cannot contain two terms $3^p,3^q$ with $p < q$
since their sum is not of the form $i\cdot 3^r$, and it cannot contain two terms $2\cdot 3^p, 2\cdot 3^q$ 
with $p < q$ since their sum is not of the form $i\cdot 3^r$.

Second, for all $h,h'\in FS^{\leq 2}(H)$, $i(h)=i(h')$:\footnote{This is Property (P5) of \cite{DJSW:16}
strengthened to $FS^{\leq 2}(H)$ instead of $H$. Note that this is necessary for its application 
just below establishing (P6), and the same is true in the setting of \cite{DJSW:16}, which needs a minor correction.}
$i(n)=1$ implies $g(n)\in\{1,2\}$ and $i(n)=2$ implies $g(n)\in\{3,4\}$. 
Then also the following property (P6) holds under the assumption that $FS^{\leq 2}(H)$ 
(rather than $FS^{\leq 3}(H)$, as in \cite{DJSW:16}) 
is monochromatic for $g$: for all $k\geq 0$ there is at most
one $n\in H$ such that $\lambda(n)=k$. Suppose otherwise as witnessed by $h\neq h'$ in $H$ with $\lambda(h)=\lambda(h')$. 
Since we also have $i(h)=i(h')$, it follows that $i(h+h')\neq i(h)=i(h')$, contra (P5).

Hence we can sparsify $H$ (computably) so as to ensure the following 

\begin{center}
(Apartness Condition): 
if $h < h'$ are in $H$ then $\mu(h)<\lambda(h')$. 
\end{center}

Assume without loss of generality that $H$ satisfies the apartness condition. 
Assume without loss of generality that the value $i(h)$ for $h\in H$ is $1$.
Let $c\in \{1,2\}$ be such that $FS^{\leq 2}(H)$ has color $c$ under $g$. 

Let 
$$H_1 := \{\lambda(n)\,:\, n \in H\}$$ 
and 
$$H_2 := \{\mu(n)\,:\, n \in H\}.$$ 

We claim that $\langle H_1,H_2\rangle$ is increasing p-homogeneous for $f$. 

First observe that, letting $H=\{h_1,h_2,\dots\}_<$, we have
$H_1 = \{ \lambda(h_1),\lambda(h_2),\dots\}_<$ and 
$H_2 = \{ \mu(h_1),\mu(h_2),\dots\}_<$. 
This is so because $\lambda(h_1)<\mu(h_1)<\lambda(h_2)<\mu(h_2)<\dots$ by the apartness condition 
and the fact that the color is not $0$.

Then we claim that $f(x_1,x_2) = c-1$ for every increasing pair $(x_1,x_2)\in H_1\times H_2$.
Note that $(x_1,x_2) = (\lambda(h_i),\mu(h_j))$ for some $i \leq j$. 
If $i=j$ we have 
$$ c = g(h_i) = 1+f(\lambda(h_i),\mu(h_i)),$$
and if $i<j$ we have
$$c= g(h_i+h_j) = 1 + f(\lambda(h_i+h_j),\mu(h_i+h_j)) = 1 + f(\lambda(h_i),\mu(h_j)),$$
since $FS^{\leq 2}(H)$ is monochromatic for $g$ with color $c$.
Thus, in any case
$$c = 1+f(\lambda(h_i),\mu(h_j)) = 1 + f(x_1,x_2).$$
This shows that $\langle H_1,H_2\rangle$ is increasing p-homogeneous of color $c-1$ for~$f$.
\end{proof}

\subsection*{Discussion}
Proposition~\ref{prop:ht25_to_ipt22} should be compared with Corollary 2.4 of \cite{DJSW:16}: 
$\mathsf{RCA}_0 + B\Sigma^0_2 +\mathsf{IPT}^2_2  \vdash \mathsf{SRT}^2_2$. 
Proposition~\ref{prop:ht25_to_ipt22} has the following corollaries.

\begin{corollary}
Over $\RCA$, $\mathsf{HT}^{\leq 2}_5$ implies $\mathsf{SRT}^2_2$. 
\end{corollary}

\begin{proof}
Dzhafarov and Hirst show that  $\mathsf{IPT}^2_2$ 
implies $\mathsf{D}^2_2$ over $\mathsf{RCA}_0$ (Proposition 3.5 in ~\cite{Dza-Hir:11}). 
Chong, Lempp and Yang have later proved that
$\mathsf{D}^2_2$ implies $\mathsf{SRT}^2_2$ over $\mathsf{RCA}_0$ (Theorem 1.4 in \cite{Cho-Lem-Yan:10}).
\end{proof}

Since $\mathsf{SRT}^2_2$ implies $B\Sigma^0_2$, we also have the following corollary.

\begin{corollary}
$\mathsf{HT}^{\leq 2}_5$ is not provable in $\mathsf{WKL}_0$. 
\end{corollary}

Also, as Ludovic Patey kindly pointed out to us, $\mathsf{IPT}^2_2$ is known to be {\em strictly} stronger than 
$\mathsf{SRT}^2_2$: Theorem 2.2 in \cite{Cho-Sla-Yan:14} showed that there is a non-standard model 
of $\mathsf{SRT}^2_2+B\Sigma^0_2$ having only low sets in the sense of the model. Lemma 2.5 in \cite{Dza-Hir:11} 
can be formalized in $\mathsf{RCA}_0$ and shows that no model of $\mathsf{IPT}^2_2$ can contain only $\Delta^0_2$ sets.
Thus, Proposition~\ref{prop:ht25_to_ipt22} implies that $\mathsf{HT}^{\leq 2}_5$ is 
strictly stronger than $\mathsf{SRT}^2_2$.

\bigskip
The proof of Proposition~\ref{prop:ht25_to_ipt22} is easily adapted to show 
\bc $\mathsf{RCA}_0 \vdash \forall n (\mathsf{HT}^{\leq 2}_{2n+1}\rightarrow \mathsf{IPT}^2_n)$. \ec
Then $\mathsf{RCA}_0 \vdash \mathsf{HT}^{\leq 2} \rightarrow \mathsf{IPT}^2$. 
On the other hand we know that $\mathsf{IPT}^2 \rightarrow \mathsf{SRT}^2$, where
$\mathsf{SRT}^2$ denotes the Stable Ramsey's Theorem for pairs and all colors
(Dzhafarov and Hirst \cite{Dza-Hir:11}, Theorem 3.3). 
Finally, it is known that $\mathsf{RCA}_0 + \mathsf{SRT}^2 \vdash \mathsf{B}\Sigma^0_3$
(Cholak, Jockusch and Slaman \cite{Cho-Joc-Sla:01}, Section 11.2). 
Therefore we also have the following corollary. 

\begin{corollary}
Over $\mathsf{RCA}_0$, $\mathsf{HT}^{\leq 2}$ implies $\mathsf{B}\Sigma^0_3$.
\end{corollary}

The author has been up to now unable to lift the combinatorial argument in the proof 
of Proposition~\ref{prop:ht25_to_ipt22} to show that 
$\mathsf{HT}^{\leq 3}_k$ implies $\mathsf{IPT}^3_2$, for some $k\geq 2$.
Note that by results of \cite{Dza-Hir:11} and \cite{DJSW:16} the following holds:
$\mathsf{RCA}_0 \vdash \mathsf{HT}^{\leq 3}_4 \rightarrow \mathrm{IPT}^3_2$.

  \part{Computational complexity theory}

\section{Thompson: Symmetric functions can be computed by Boolean  circuits of linear size and logarithmic depth}

 Declan Thompson and Matt Bray, Honour's students from Auckland University,  visited the 
  Research Centre Coromandel in February. Declan worked  on Boolean circuits, connecting to the CompSci 750 class on computational complexity which    Andr\'e  had co-taught in Semester 2,  2015.
 
 \vsp
%	
%	Let $\mathbb{B} = \{0,1\}$.
%	\begin{definition}
%		A function $f:\B^n\to\B^m$ is \emph{symmetric} if for every permutation $\pi$ of n arguments, $f(\vec{x}) = f(\pi\vec{x})$.
%	\end{definition}
%	
%	Examples of symmetric functions include:
%	\begin{itemize}
%		\item The functions $E^n_k:\B^n\to \B$ defined by $E^n_k(x)=1$ iff $x_0 + \dots + x_{n-1} = k$.
%		\item The threshold functions $T^n_k:\B^n\to\B$ defined by $T^n_k(x)=1$ iff $x_0 + \dots + x_{n-1} \geq k$.
%		\item The parity functions $P^n:\B^n\to\B$ defined by $P^n(x) = x_0 + \dots + x_{n-1} \mod 2$.
%	\end{itemize}	
%	A symmetric Boolean function depends only on the number of $1$s in the input. Hence, given a binary representation of this number we can construct\dots.
%	
%	Thus once we have a binary representation of the number of $1$s in the input within a circuit, we can compute the function using linear size and logarithmic depth. This document is not concerned with this step of the procedure. Instead, we focus on the task of obtaining a binary representation of the number of $1$s in the input in linear size.

	A  function $f\colon \, \{0,1\}^n\to \{0,1\}$ is symmetric if its value does not depend on the order of its arguments. For Boolean symmetric functions, this means that the output depends only on the number of $1$s in the input.   It is not hard to see that given a binary representation of the number of $1$s in the input, a Boolean  circuit of   linear size and logarithmic depth  can  compute the value of a symmetric function.  So it suffices to find  an efficient method for obtaining this binary representation from the original input.
	
	Circuits exist which add binary numbers in linear size and logarithmic depth (see, for example, the Krapchenko adder of \cite{Wegener:87}). A n\"aive approach to counting the number of $1$s is to treat each input as an individual number and utilise an adding tree. Unfortunately this requires a circuit of more than  logarithmic depth. Instead, we will utilise so called ``3-for-2'' adders and finally only one Krapchenko adder.
	
	A \emph{full adder} is a circuit for calculating the sum of three bits. Figure \ref{fig:full adder} gives a construction for a full adder. There are three input bits and two output bits, representing the two bit output. A \emph{carry save adder} (CSA) utilises a chain of full adders to take three $n$-bit inputs $a,b,c$ and return two $n+1$-bit outputs $u,v$ such that $a+b+c = u+v$. The CSA works by sending sum bits ($y_1$ in Figure \ref{fig:full adder}) to $u$ and carry bits ($y_2$) to $v$. Specifically, if the full adder returns $y_2y_1$ from $a_i, b_i, c_i$, then $u_i = y_1$ and $v_{i+1} = y_2$. We set $u_n = v_0 = 0$. As an example, consider the following sum.
	
	\begin{align*}
	100110&\\
	111101&\\
	\underline{+\phantom{99}101101}&\\
	0110110&\quad = u\\
	1011010& \quad = v
	\end{align*}
	
	\begin{figure}
		\centering
		\includegraphics[width=0.35\textwidth]{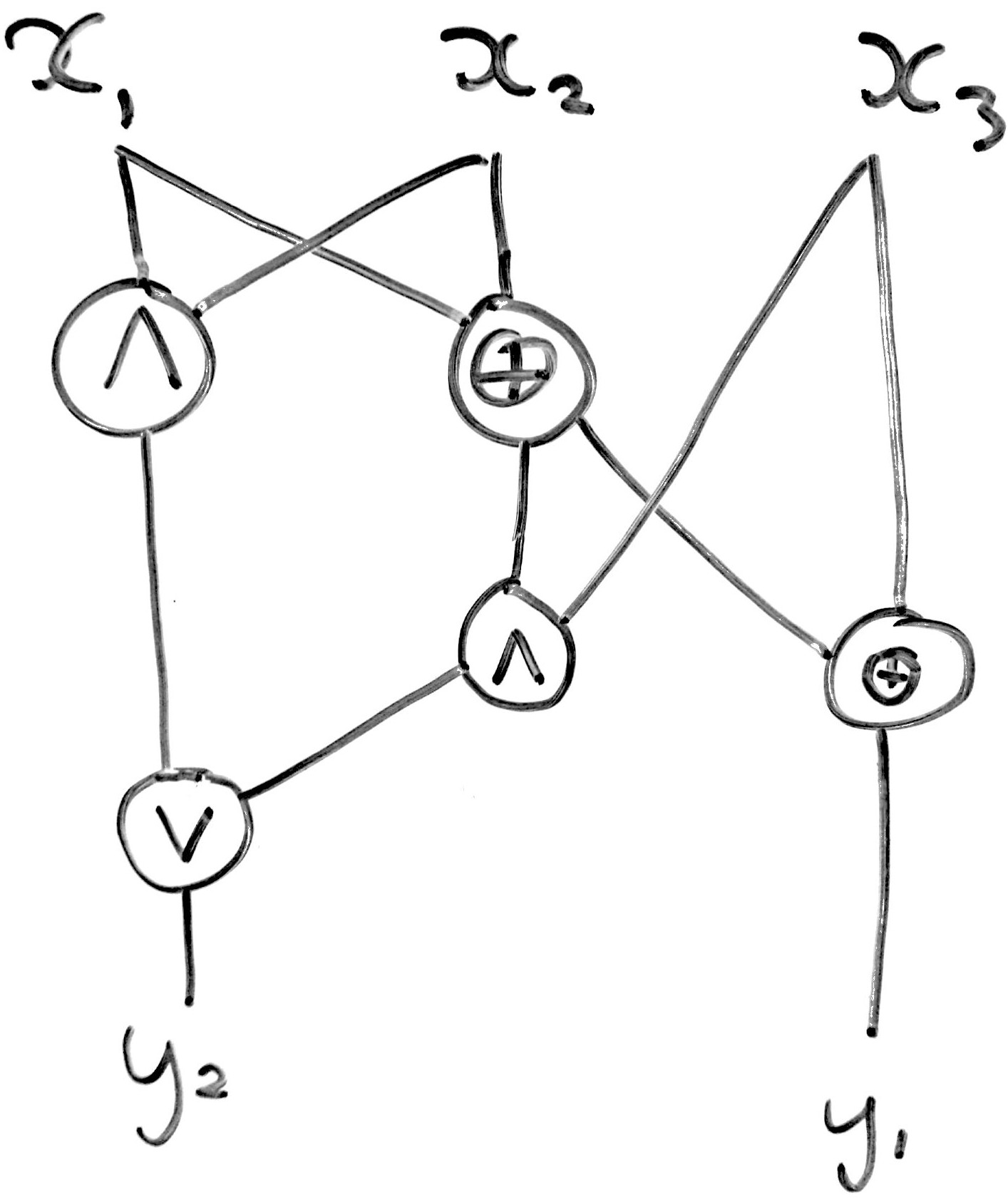}
		\caption{A full adder. $x_1 + x_2 + x_3 = y_2y_1$. $\oplus$ denotes addition modulo $2$.}\label{fig:full adder}
	\end{figure}
	
	A full adder has a constant number of gates and a CSA uses a full adder for each bit. Hence a CSA uses $O(n)$ size circuits with depth constant, where $n$ is the number of bits of each input number.
	
	We construct an adder tree for $x_0, \dots x_{n-1}$ by running CSAs in parallel and then combining their outputs. At the first ``level'', there are $\frac{1}{3}n$ CSAs each taking in three $1$ bit inputs and returning two $2$ bit outputs (4 wires in total). At the second level there are $\frac{1}{3}n\cdot2$ total inputs and each CSA takes three inputs, so there are $\frac{1}{3}\frac{2}{3}n$ CSAs. In general, at the $jth$ level there are $\frac{1}{3}\frac{2}{3}^{j-1}n$ CSAs taking three $j$ bit inputs and returning two $j+1$ bit outputs. Each CSA decreases the number of numbers to add by $1$ until we are left with two binary numbers. Hence we require $n-2$ CSAs. 
	
	Logarithmic depth of the adder tree is easy to establish. Each CSA has constant depth and since we combine them in a tree fashion the depth of CSAs is logarithmic. Each CSA has size proportional to the number of bits in the inputs. From this, we can conclude that the size of the adder tree is approximately
	\[
	O(\frac{1}{3}n\sum_{j=1}^{\lfloor\log n\rfloor}\frac{2}{3}^{j-1}j).
	\]
	However $\sum_{j=1}^{\infty}\frac{2}{3}^{j-1}j = 9$ and so in fact the adder tree is of size $O(n)$.
	
	Thus we have obtained two binary numbers which add to the total number of $1$s in the input, using linear size and logarithmic depth circuits. It is now a simple process to use a linear size, logarithmic depth adder (such as Krapchenko's) to obtain a single binary number, which can then be used to determine the output of an arbitrary symmetric function.

\section{Describing ordinals less than $  \varepsilon_0$ by finite trees}

Andreas Weiermann, Paul Shafer and Nies discussed in Ghent in March. This discussion  connected some well-known facts. 

For this note a \emph{tree} $B$  is an acyclic connected directed finite graph with a distinguished  root $r$.  $B$ can be naturally viewed as a finite subset of $\omega^{<\omega}$ (sequences of natural numbers) closed under initial segments; $r$ is the empty string. If $B \neq r$ then $B$ is given by a tuple $(B_1, \ldots, B_n)$ of trees, where the $B_i$ correspond to the successors of $r$ from left to right.

 The ordinal $o(B)$ is defined recursively as follows: let $o(r) = 0$. If $B= (B_1, \ldots, B_n)$ then \bc  $o(B) = \sum_{i=1}^n \omega^{o(B_{\pi(i)})}$ \ec where $\pi$ is a permutation of $\{1, \ldots, n\}$ such that $o(B_{\pi(1)}) \ge \ldots \ge o(B_{\pi(n)})$.

 The \emph{norm} $N(\alpha)$ of an ordinal $\alpha < \epsilon_0$ is defined by $N(0) = 0$ and $N(\gamma) =  1 + N(\alpha) + N(\beta)$ if $\gamma$ has	
 Cantor normal form  $\omega^\alpha +\beta$.    Clearly each $\gamma < \varepsilon_0$ has the form $o(B)$ for some tree $B$, and  $N(\gamma)$ is the number of edges of such a tree.
 
 The ordinal  of a tree is a complete  invariant for tree isomorphism:
 
 \begin{prop} Let $B, C$ be trees. We have \bc $B \cong C \LR o(B)= o(C)$. \ec
 \end{prop}
 \begin{proof}First suppose that $\rho \colon B \cong C$. If $B=C = r$ then $o(B) =o(C) =0$. Otherwise $B = (B_1, \ldots, B_n)$, $C = (C_1, \ldots, C_n)$ and $\rho \colon B_i \cong C_{\sss(i)}$ for $\sss \in S_n$. Inductively we have $o(B_i) \cong o(C_{\sss(i)})$. Choose $\pi \in S_n$ such that $o(B_{\pi(1)}) \ge \ldots \ge o(B_{\pi(n)})$. Then $o(C_{\sss(\pi(1))}) \ge \ldots \ge o(C_{\sss(\pi(n))})$. Therefore $o(B)= o(C)$. 
 
 Now suppose $\alpha = o(B)= o(C)$. If $\alpha=0$ we have $B \cong C$ trivially. Otherwise write $\alpha$ in Cantor normal form $\sum_{i=1}^n \omega^{\beta_i}$ where $\beta_1 \ge \ldots \ge \beta_n$. Since CNF is unique we have $B = (B_1, \ldots, B_n)$, $C = (C_1, \ldots, C_n)$ with $\beta_i= o(B_{\pi(i)}) = o(C_{\theta(i)})$ for $\pi, \theta \in S_n$. This shows that $B \cong C$. 
 \end{proof}

The number of ordinals $\alpha$  with $N(\alpha)= n$ is asymptotically $(2.95)^n n^{-1.5}$ by a result of Weiermann. The number of trees with $n$ edges is $\binom {2n} n /(n+1)$ (Catalan number), which is asymptotically $4^n n^{-1.5}/\sqrt \pi$ (much larger). 

A tree $B$ is canonical if $B$ is  a root, or  $B = (B_1, \ldots, B_n)$ for canonical trees $B_i$ such that $o(B_1) \ge \ldots \ge o(B_n)$. Clearly each ordinal less than $\epsilon_0$ is $o(B)$ for a unique canonical tree. Isomorphism and canonization is in LOGSPACE by a result of Aho,   Hopcroft and Ullman. Also see Buss 1995 where algorithms in the smaller class ALOGSPACE are given.

\begin{question} How can one generate a random ordinal with given norm  $n$ using poly($n$) steps and the appropriate number of $n \log_2 (2.95)$ of  random bits? \end{question}
 
 This amounts to generating a random canonical tree with $n$ edges.

\section{Nies and Scholz: Grothendieck's constants}
(Written by Nies after some discussion with Volkher Scholz, who works at U Gent in the research group of Prof. Verstraete, and visited Auckland in April.)
Let $\mathbb F$ be one of the fields $\RR, \mathbb C$.  Choose scalars from $\mathbb F$. A.\ Grothendieck proved  in   1953    that there are positive constants $K_G^\mathbb F$ as follows.  

 Let $C=\seq{\gamma_{i,k}}$ be an $n \times n$ matrix such that \begin{equation} \label{eqn:matrix condition}  \sum_{i,k} \gamma_{i,k} a_i b_k \le \sup_i |a_i| \sup_k |b_k| \end{equation} for each scalars $a_i, b_k$. Then  for each Hilbert space $H$ over $\mathbb F$,   and vectors $x_i, y_k \in H$  ($1\le i,k \le n$)
 \[ \sum_{i,k} \gamma_{i,k} \la x_i, y_k \ra \le K_G^\mathbb F \sup_i  ||x_i|| \sup_k ||x_k||. \] 
 See the reprinted paper~\cite{Grothendieck:96}. In (\ref{eqn:matrix condition}) we may assume that $|a_i| = |b_k| =1$ because for fixed $C$ the inequality describes a convex set, so it suffices to look at its extreme points. Given the matrix $C$ as an input, it is known,  that the condition in NP-hard. Say  for $\RR$ the problem Max-Cut can be reduced to it (which asks given  a graph whether one can partition the set of vertices into two sets so that the number of crossing edges is above a threshold).  Easy matrices $C$  satisfying the condition are: $n^{-1} I$, or the one where each $\gamma_{i,k}$ equals $n^{-2}$. 
 
 We can get away with finite-dimensional Hilbert spaces. However the dimension of $H$ must be unbounded. E.g the value of $K^\RR_G$ for   dimension $2$ is $\sqrt 2$.

 Interestingly, G's work is closely related to, but predated the Bell inequality (1964). The transition from a problem in real analysis to the setting of Hilbert space corresponds to the transition from classical to quantum physics.  The real Grothendieck constant can be viewed as an upper bound on the possible violation of Bell's inequality in a quantum system.
 
  The precise values of the constants are not known. We know that (see \cite{Pisier:12})
 \[ 1 < K_G^{\mathbb C} < K_G^\RR < 1.782.\]
 
 Raghavendra and Steurer \cite[Thm 1.3]{Raghavendra.Steurer:09} prove that $K^\RR_G$ is  a computable real, and in fact computable up to precision $\eta>0$ in time proportional to  $\exp(\exp(O(\eta^{-3}))$. No surprise we still don't know what its value is.

%  \part{Algebra}
\part{Group  theory and its connections to logic}

\section{Doucha and Nies: Primitive group actions in the setting of Polish Spaces}

Michal Doucha and Nies worked at the RCC in 2014/2015.
\begin{definition}
Let $G$ be a Polish group and $X$ a Polish $G$-space with all orbits dense. 
  $(G,X)$ is called \emph{imprimitive} if there exists a closed proper subset $D\subseteq X$ such that for every $g\in G$ either $g\cdot D=D$ or $g\cdot D\cap D=\emptyset$ and moreover, $D$ intersects some orbit in at least two elements. I.e., $\ex x \neq y \in D \ex g \in G [ g \cdot x =y]$. The set $D$ is   called a \emph{closed domain of imprimitivity}. Otherwise, $(G,X)$ is called \emph{primitive}.
\end{definition}
Given an action $G \curvearrowright X$, we say that an equivalence relation $E$ on $X$ is \emph{$G$-invariant}  if for every $g\in G$ and every $x,y\in X$ we have $x E y\lra g\cdot x E g\cdot y$. Equivalently, $g \cdot [x]_E= [g\cdot x]_E$ for each $g\in G$ and $x \in X$.
\begin{prop}
Given a  Polish group $G$ and a Polish $G$-space $X$ with all orbits dense, TFAE: 

\bi \item[(i)]  $(G,X)$ is   imprimitive  \item[(ii)] there exists a $G$-invariant smooth  equivalence relation $E$ on $X$ other than $\mathrm{id}_X$ or $X^2$,  with all equivalence classes closed,  such that       some   equivalence class intersects some  orbit in at least two elements.
\item[(iii)] as in (ii) but  without the restriction to smoothness.
\ei \end{prop}
\begin{proof}
(iii)$\to$(i) Let $D=[x]_E$ be an equivalence class of $E$ that intersects some  orbit   in at least two points. We claim that $D$ is   a closed domain of imprimitivity. It is clearly closed and properly contained in $X$.   For each $g\in G$, if $gx E x$ we have $gD = [gx]_E= D$. Otherwise $gD \cap D = \ES$.

\n (i)$\to$(ii). Let  $D$ be  a closed domain of imprimitivity. Let  \bc $H=\{h\in G: h\cdot D=D\}$. \ec Clearly, $H$ is a closed subgroup of $G$. Since every orbit in $X$ is dense and $D$ is a proper subset, $H$ is a proper subgroup of $G$. 

Let us define a relation $E$ on $X$. For $x,y\in X$ we define \bc $x E y \lra \exists g\in G  [ g^{-1}\cdot x \in D \lland g^{-1}\cdot y \in D$. \ec  We shall prove that $E$ is a smooth  $G$-invariant equivalence relation with closed classes such that some class intersects some orbit in at least two elements. 

Since for every $g\in G$ the map $x\to g\cdot x$ is a homeomorphism of $X$ we have that for each $g\in G$ the set $g\cdot D$ is also closed. Denote by $X/D$ the set $\{g\cdot D:g\in G\}$. We claim for every $F_1\neq F_2\in X/D$ we have $F_1\cap F_2=\emptyset$. Indeed, write $F_i$ as $g_i\cdot D$ for $g_i\in G$, where $i\in \{1,2\}$. If $g_1\cdot D\cap g_2\cdot D\neq \emptyset$, then $D\cap (g_1^{-1}\cdot g_2)\cdot D\neq \emptyset$. Thus by assumption $(g_1^{-1}\cdot g_2)\cdot D=D$, so $g_1\cdot D=g_2\cdot D$, a contradiction. It follows that $X/D$ is the set of $E$-classes. In particular, $E$ is an equivalence relation with each equivalence class being closed and the $E$-class $D$ intersects by assumption some orbit in at least two elements.

Consider the quotient $G/H$ consisting of left cosets of $H$ with the quotient topology. It is a folklore fact that this is a Polish space metrizable by the metric $\delta$ defined for any $g_1\cdot H, g_2\cdot H$ as the Hausdorff distance  \bc $\delta(g_1\cdot H,g_2\cdot H)=\inf\{d_G(h_1,h_2): h_1\in g_1\cdot H,h_2\in g_2\cdot H\}$, \ec  where $d_G$ is some compatible right-invariant metric on $G$. See for example \cite{Gao:09} for details. Let us define the map $\phi:X\rightarrow G/H$ as follows: \bc  $\phi(x)=g\cdot H$ if $g^{-1}\cdot x\in D$.  \ec This is well-defined since $H$ fixes $D$. We now claim that $\phi$ is a Borel reduction of $E$ into $\mathrm{id}(G/H)$. This will show that $E$ is smooth.

To show that it is a reduction, pick some $x,y\in X$. Suppose that $x E y$. Then by definition there is $g\in G$ such that $g^{-1}\cdot x$ and $g^{-1}\cdot y$ lie in $D$, thus $\phi(x)=\phi(y)=g\cdot H$. On the other hand, if  $\phi(x)=\phi(y)=g\cdot H$, then $g^{-1}\cdot x$ and $g^{-1}\cdot y$ lie in $D$. 

It remains to check that $\phi$ is Borel. We show that $\phi$ factorizes through $X/D$ as $\psi\circ \pi$, where $\pi:X\rightarrow X/D$ is the canonical projection and $\psi:X/D\rightarrow G/H$ sends $g\cdot D$ to $g\cdot H$. We note that a set $U\subseteq X/D$ is open if and only if $\bigcup U$ is open in $X$. Then we show that $\psi$ is continuous (even open) and $\psi$ is a bijection whose inverse $\psi^{-1}$ is continuous, that suffices.

That $\pi$ is continuous (and open) follows directly from the definition of the topology on $X/D$. Also, it is clear that $\psi$ is a bijection. We check that $\psi^{-1}$ is continuous. Let $U\subseteq X/D$ be an open neighbourhood of some $g\cdot D$. Pick arbitrarily some $x\in D$. Since $\bigcup U$ is an open neighbourhood of $x$ and the group action is continuous there exists an open neighbourhood $V$ of $g$ such that for every $h\in V$ we have $h\cdot x\in \bigcup U$. We have that $V_H=V\cdot H$ is an open neighbourhood of $g\cdot H$ in $G/H$ and we claim that $\psi^{-1}(V_H)\subseteq U$. This is immediate. Any element of $V_H$ is of the form $h\cdot H$ for some $h\in V$ and thus sent by $\psi^{-1}$ to $h\cdot D$. We have that $h\cdot x\in \bigcup U$ and since $\bigcup U$ is $E$-invariant we have that $h\cdot D\subseteq \bigcup U$, thus $h\cdot D\in U$.
\end{proof}
\begin{prop} \label{prop:stabilizer}
Let $G$ be a Polish group and $X$ a Polish $G$-space with all  orbits dense. Suppose that for any $x\in X$ the stabilizer $G_x$ is a maximal closed subgroup of $G$. Then $(G,X)$ is primitive.
\end{prop}
\begin{proof}
Suppose that $(G,X)$ is imprimitive, so  there exists a closed domain of imprimitivity $D\subseteq X$. Let $H=\{h\in G:h\cdot D=D\}$. $H$ is clearly a closed subset of $G$. Moreover, since $D$ is a closed domain of imprimitivity we have that $H$ is a group. Let $x\in X$ be such that we have $|D\cap G\cdot x|\geq 2$. We may suppose that $x\in D$. We have that $G_x\leq H$. We shall show that $G_x<H<G$ and that will be a contradiction with the maximality of $G_x$. By assumption there exists $h\in G\setminus G_x$ such that $h\cdot x\in D$. Since $D$ is a domain of imprimitivity, so for every $g\in G$ we have either $g\cdot D=D$ or $g\cdot D\cap D=\emptyset$, we must have that $h\cdot D=D$ and thus $h\in H$. We have shown that $G_x<H$. On the other hand, $D$ is a proper closed subset of $X$ and since the orbit of $x$ is dense, there exists $g\in G$ such that $g\cdot x\notin D$ and thus $g\cdot D\cap D=\emptyset$ and $g\notin H$. We have shown that $H<G$ and the proof is complete.
\end{proof}

 \n {\bf Questions.}

\n 1. How about the converse implication in  Prop. ~\ref{prop:stabilizer}?

\n 2. Robinson \cite[7.2.5]{Robinson:82}  states that if $G$ is primitive on $X$,  then every nontrivial normal subgroup is transitive. (The latter property is called quasiprimitive by Cheryl Praeger.)   Check this in the Polish setting,  where the normal subgroup is closed, and we have topological transitivity in that every orbit is dense.  

\section{Nies and Tent: a sentence of size $O(\log n)$ expressing that a group has $n$ elements}
This post  is related to Nies'  and Tent's article  ``Describing finite groups by short first-order sentences"  \cite{Nies.Tent:17}. We use the definition $\log n = \min \{ r \colon \, 2^r \ge n\}$. 
  In that article we gave a description of any finite group $G$ via a first order sentence of length $O (\log^3|G|)$. Here we want to express that the group has size $n$ by a first-order sentence of length $O(\log n)$.   This will also yield a new way to describe the finite simple 
  groups in length $O(\log n)$, still relying on CFSG but  not relying on the short presentations in \cite{Guralnick:08}
  
This work happened in March 2016  at UCLA, after some preliminary work of Nies with the honour's student Matthew Bray. At the BIRS permutation groups meeting Nov 13-18, Csaba Schneider and David Craven provided the crucial references needed to distinguish  by short first order sentences the few examples of non isomorphic simple groups of the same size.

\begin{thm} \label{thm: express size} For each $n$ there is a sentence $\phi_n$ in the first order language of groups such that $|\phi_n | = O (\log n)$ and for each group $G$, 
\bc $G \models \phi_n \LR |G| = n$. \ec \end{thm}

We first provide some necessary facts. We use \cite[Lemma 2.1]{Nies.Tent:17}:

\begin{lemma} \label{repeated squaring} 
For each positive integer $r$, there is an     existential  formula $\theta_r(g,x)$   in the first-order  language of monoids  $L(e, \circ)$, of length $O(\log  r)$, such that for  each   monoid~$M$, 
$M \models \theta_r(g,x)$ if and only if $x^r=g$.  

%\n {\rm (ii)}  $M \models \rho_n(g, x)$ if and only if $x^r=g$ for some $r$ with $0< r \le n$. 
\end{lemma}
In particular, we can express that a  group $G$ has exponent dividing $r$ using $O(\log r)$. 
The following variant is also needed.

\begin{lemma} \label{lem:exp monoid} 
For each positive integer $k$, there is a        formula $\psi_r(y,x)$   in the first-order  language of monoids  $L(e, \circ)$, of length $O(\log  r)$, such that for  each   monoid~$M$, 
$M \models \psi_r(g,x)$ if and only if $x^i=g$ for some $i$ with $0 \le i \le r$.  
\end{lemma}
\begin{proof} Let $\psi_1(y,x) \equiv y=1 \lor y=x$.   Recursively define 
\begin{eqnarray*} \psi_{2k}(y,x) & \equiv & \ex u,v [ y = uv \land \fa z. (z=u \lor z=v) \psi_k(z,x)] \\
 \psi_{2k+1}(y,x) &\equiv & \ex u,v [ (y = uv \lor y = uvx)  \land \fa z. (z=u \lor z=v) \psi_k(z,x)]
 \end{eqnarray*}
 Clearly $\psi_r$ works as required. Further, $|\psi_r| = O(\log r)$.
\end{proof}

We next provide   an easy fact on $p$-groups (which   is not first order at this stage).
\begin{fact} \label{fact:easy on p} Suppose   $L$ is a $p$-group. Then   $|L|\le p^r  \LR$ 

$ \ex x_1, \ldots,  \ex x_r  \in L \fa y \in L $

\begin{equation} \tag{$\diamond$}  \ex a_1,  \ldots,  a_r  \, [ 0\le a_i   < p \land     \,  y = \prod_i x_i^{a_i}]. \end{equation}
\end{fact}
\begin{proof} The implication $\LA$ is immediate. 

For the implication $\RA$, we use induction on~$r$. The base case $r=1$ is obvious since $L$ is trivial or cyclic of order $p$. Now suppose $r >1 $ and the implication holds for $r-1$. If $L$ is non-trivial, pick $x_r$ in the centre   of order $p$. By inductive hypothesis for $L/N$ where $N = \la x_r \ra$, we can choose $x_1, \ldots, x_{r-1} \in L$ such that statement holds in $L/N$ via $x_i N$, $1 \le i < r$.  So for each $y \in L$, we have 
$$yN = \prod_{i=1}^{r-1} (x_iN)^{a_i}=  \prod_{i=1}^{r-1} (x_i)^{a_i}N$$
 for some $a_i$ with $0 \le a_i < p$. Therefore  there is $a_r$ with $0 \le a_r < p$ such that $y = \prod_i x_i^{a_i}$, as required.  \end{proof} 

\begin{lemma} \label{lem:p-subgroup} For  $k=p^r$, $p$ prime, $r \in \NN$ there is a sentence $\beta_{k}$ of length $O(\log k)$  such that $G \models \beta_{k}$ iff $G$ has a subgroup of size $k$. \end{lemma}
\begin{proof}
We  express $(\diamond)$  by a first-order sentence of length $O(\log k)$ via the formulas in Lemma~\ref{lem:exp monoid}: 
\begin{eqnarray*}  \chi_k(y; x_1, \ldots, x_r) & \equiv &  \ex s_0,  \ldots,  s_r  \\   &&  [ s_0 = 1 \land s_r=y \land \bigwedge_{i=1}^r  \ex v \, ( \psi_{p-1} (v, x_i)  \land s_i = s_{i-1}v)].\end{eqnarray*}
The length of $|\chi_k|$ is $O(r \log p) = O(\log k)$. 

Given a group $G$ and $\ol x = x_1, \ldots, x_r \in G$, write  $U^k_{\ol x} = \{y \in G \colon \, G\models \chi_k(y; \ol x)\}$. 
The sentence $\beta_k $ expresses that there is $\ol x = x_1, \ldots, x_r$ such that $U^k_{\ol x} $ is a subgroup of exponent dividing~$k$, and $r$ is optimal, namely, there is no $\ol y = y_1, \ldots y_{r-1}$ such that $U^{k/p}_{\ol y} = U^k_{\ol x}$. 

If $G$ has a subgroup $L$ of size $k$ then $G \models \beta_k$ by Fact~\ref{fact:easy on p}.  Now suppose $G \models \beta_k$ via $x_1, \ldots x_r \in G$.  Then $L= U^k_{\ol x }$ is a subgroup of $G$ of size $k$. 
\end{proof}

\begin{proof}[Proof of Thm.\ \ref{thm: express size}] Using Lemma~\ref{repeated squaring} we can express that the group $G$ has exponent dividing $n$. In particular, only prime factors of $n$ can occur in the order of $G$.   
Suppose $n = \prod_{i=1}^m p_i^{r_i}$ for prime numbers $p_1, \ldots,  p_m$.    We    express using Lemma~\ref{lem:p-subgroup}   for each $i \le m$ that there is a Sylow subgroup of size $p_i^{r_i}$.  For each $i$ this takes length $O(\log(p_i^{r_i}))$ with the   $O$-constant independent of $i$.  So the resulting sentence has length $O(\log n)$. 
\end{proof}

It would   be interesting to find sentences as in Theorem~\ref{thm: express size} with a bounded number of quantifier alternations.

Next we express in logarithmic length that a group is simple. We     use \cite[Lemma 2.3]{Nies.Tent:17} on generation. The notation is adapted slightly.
 \begin{lemma} \label{generation}
For each positive integers $k,v$, there exists a first-order formula $\alpha^v_k(g;z_1,\ldots,z_k)$   of length $O(k+\log v)$ such that   for each group $G$ of size at most $v$,  ${G \models \alpha^v_k(g;z_1,\ldots,z_k)}$ if and only if ${g \in \langle z_1,\ldots,z_k \rangle}$. 
\end{lemma}

Given a group $G$ of size at most $v$, we write $L^v_k({\ol z}) = \{y \in G \colon \, G\models \alpha^v_k(y; \ol x)\}$ in case this is a subgroup. 

Every finite group  has a  generating set of logarithmic size. So  a group $G$ of size at most $v$  is simple iff 
\bc $G \models \fa z_1,\ldots,z_k \, [ L^v_k({\ol z}) \lhd G \to ( L^v_k({\ol z}) = \{e\} \lor L^v_k({\ol z}) = G)]$, \ec 
where $k = \log v$.

One  can   use  these facts to give a new type of first-order description of  finite simple groups in logarithmic length. First one says what the size of the groups is,  and that it is simple.  A finite simple group is   determined by its size, with the exception of \bi \item $ \mathtt{PSL}_3(4)$ which has the same size as $\mathtt{Alt}_8$ without being isomorphic to it, and  \item  the groups   $B_n= P\Omega_{2m+1}(q)$  and 
$C_m=PSp_{2m}(q) $, $q$ an odd prime power,  $m >2$, which have the same size  $\frac 1 2 q^{m^2} \prod_{i=1}^m (q^{2i}-1)$  without being isomorphic. \ei  (See 
http://mathoverflow.net/questions/107620/non-isomorphic-finite-simple-groups for background.)
 The exceptional cases above can   be distinguished    by the fact  that the nonisomorphic  groups of the same size  have different numbers of conjugacy classes of involutions, and that the number of these conjugacy classes is logarithmic in the size. Firstly, in $ \mathtt{PSL}_3(4)$ all involutions are conjugate, while in $A_8$ there are two conjugacy classes, namely $(12)(34)$ and $(12)(34)(56)(78)$. Next,  in $B_m$ there are $m$ classes, in $C_m$ for $m$ odd there are $(m+1)/2$ classes and for $m$ even there are $m/2+1$ classes. For the latter, see 
\cite[Table 4.5.1, p.\ 172]{Gorenstein.ea:98}.
To read this table, note that  classes in the simple group are the coset `1', diagonal involutions (outer automorphisms of the simple group) are labelled `d'. The notation $1/d$ [condition] means 1 if condition holds, $d$ if condition does not hold. (Thanks to David Craven for pointing out the reference and explaining this.)

We note that \cite[Lemma 2.5]{Kimmerle.etal:90} shows that the number of involutions of $B_m$, $C_m$ also differs for $m>2$. This could also be used. (Thanks to Csaba Schneider for this reference.)

\section{Melnikov and Nies: A computable compact abelian group such that the Haar measure is not computable} \label{Melnikov Nies Haar measure computable}
Melnikov and Nies worked at the Research Centre  Coromandel in June.
Recall a computable  topological space $X$ is  given  by a sequence of basis sets $\seq {B_n} \sN n$ such that for every two such sets $B_i,B_k$  we can uniformly  represent $B_i  \cap B_k$ as an effective union of basic sets. Each computable metric space is also a computable topological space with the basis given by the $B_\delta(p)$ for $\delta \in \QQ^+$ and $p$ a special point.  
We say that a Borel measure $\mu$ on $X$ is computable if   $\mu(B_i)$  is   uniformly  left-c.e. in~$i$ (if boundaries of   open sets are null   we could as well require that is it uniformly computable).

A \emph{computable topological  group} is a group $G$ that is also a computable topological space, in such a way that the group operations are effectively continuous.  

Recall that every separable compact group has a unique translation invariant probability measure, called the Haar measure. 
\begin{thm} There is a computable compact abelian group such that the Haar measure is not computable. \end{thm}
\begin{proof} 
Let $K$ denote the halting problem with effective enumeration $\seq{K_s}$; we may assume that $K_2 = \ES$. We first give uniformly in $e$ a presentation of a discrete cyclic group $G_e$ such that the Haar measure on $G_e$ is not uniformly computable. 

For a real $\theta$ let $[\theta]= e^{2 \pi i \theta}$. The distance between to points on the circle is the usual shortest arc length.

  We define a computable real $v= v_e$ uniformly in $e$;  $[v_e]$ will be a generator of $G_e$ seen as a subgroup of the circle group. At stage $s$, if $e \not \in K_s$, define $v_s = \frac 1 2 + \tp {-s}  $.  If $e\in K_s \setminus K_{s-1}$ define $v_t = v_{s-1}$ for all $t \ge s$. 

The special points of $G_e$ are   the  uniformly computable reals given by the Cauchy name  $\seq{i v_{s+i}}_{\sN s}$. If $e \not \in K$ then $G_e = \{ [0], [1/2]\}$;  if $e \in K$ then $|G_e | \ge 8$.%  (and $G_e$ consists of 

The   discrete topology on $G_e$ is uniformly computable: as an effective  basis take the sets $G_e \cap B_\delta([i v^e])$ for $\delta \in \QQ \cap (0,1/2]$. 
Let $\mu_e$ be Haar measure on $G_e$ with the discrete topology.  Clearly by the translation invariance of $\mu_e$ we have $e \not \in K \lra \mu_e(B_{1/8}(1)) > 1/4$. So $\mu_e$ is not uniformly computable.

Now let $G= \prod_e G_e$ topologized with the product topology which is compact and effective with the usual product basis.  Let $p_e \colon G \to G_e$ be the projection onto $G_e$ which is computable uniformly in $e$. If the Haar  measure on $G$ is computable then the image measure $p_e(\mu)$  on $G_e$ is uniformly computable. However, $\mu_e= p_e(\mu)$ by uniqueness of Haar measure, contradiction.
\end{proof} 
\section{Fouche and Nies:     computable profinite groups} \label{Fouche-Nies-groups}
Willem Fouch\'e and Nies went on a 1-week  retreat near Port Elizabeth, South Africa, and before and after  worked at Unisa Pretoria. As one  topic they studied  randomness in  computable profinite groups.

\subsection{Background on profinite groups} A separable compact group $G$ is called \emph{profinite} if $G$ is the inverse limit $$G =\varprojlim_n \seq {G_n, p_n}\sN n$$ of a system of discrete finite groups $$ \to_{p_n} G_n \to_{p_{n-1}} G_{n-1} \to \ldots_{p_2} \to G_1.$$ The inverse limit is   determined up to isomorphism by the universal property formulated in terms of category theory. For a concrete instantiation, it   can be seen as a closed subgroup $U$  of the direct product $\prod_{n} G_n$  consisting of the functions~$\alpha $ such that $p_{n}(\alpha (n+1)) = \alpha(n)$ for each $n>0$.     

We may assume that  the maps $p_n$ are onto  after replacing $G_n$ by its  subgroup of elements such that all the iterated pre-images under the maps $p_i$  are defined. This corresponds to pruning  a tree by removing dead ends.

 \begin{remark}[Haar measure] \label{rem:Haar}  Given $G$ as an inverse limit of an onto system, the Haar  probability measure $\mu$  can be concretely  defined as follows. Let $q_n \colon G \to G_n$ be the natural projection. A clopen set $C$ of $G$ has the form $ C = q_n^{-1} (F)$ for a finite set $F \sub G_n$. By definition $\mu$ is  translation invariant, so   $\mu (C) =|F|/|G_n|$.  As the clopen sets form a basis this determines the measure on all the Borel sets of $G$.
 \end{remark}

 \subsubsection*{Completion} \label{ss:completion}
  The definition  below is taken from  \cite[Section 3.2]{Ribes.Zalesski:00}. 
 Let $G$ be a group, $\+ V$ a set of normal subgroups of finite index in $G$ such that $U, V \in  \+ V $ implies that there is $W \in \+ V$  with  $W \sub U  \cap V$. We can turn $G$ into a topological group by declaring $\+ V$ a basis of neighbourhoods (nbhds) of the identity. In other words, $M \sub G$ is open  if for each $x \in M$ there is $U \in \+ V$ such that $xU \sub M$. 
 
\begin{definition} \label{def:completion} The completion of $G$ with respect to $\+ V$ is the inverse limit  $$G_{\+ V} = \varprojlim_{U \in \+ V} G/U,$$ where $\+ V$ is ordered under inclusion and the inverse system is equipped with  the natural    maps: for $U \sub  V$, the  map  $p_{U,V} \colon G/U \to G/V $  is given by $gU \mapsto gV$. 
\end{definition} The inverse limit can be seen as a closed subgroup of the direct product $\prod_{U \in \+ V} G/U$ (where each group $G/U$ carries  the discrete topology), consisting of the functions~$\alpha $ such that $p_{U,V}(\alpha (gU)) = gV$ for each $g$.  Note that the map $g\mapsto (gU)_{U \in \+ V}$ is a continuous homomorphism $G \to G_{\+ V}$ with dense image; it is injective iff $\bigcap \+ V = \{1\}$. 
 
 If the set $\+ V$ is understood from the context, we will usually  write $\widehat G$ instead of $G_\+ V$.
 
  \subsubsection*{Free profinite groups}   

\begin{definition} \label{def: F_finite }  Let $   \widehat F_k$ be the free profinite group in $k$ generators $x_0,   \ldots, x_{k-1}$  ($ k <  \omega$).  \end{definition} Clearly, $   \widehat F_k$ is the profinite completion of   the abstract free group on $k$ generators with respect to the system of all subgroups of finite index.  Any topologically finitely generated  profinite group can be written in  the form \[    \widehat F_k / R\]
for some $k$ and   a closed normal subgroup~$R$ of~$   \widehat F_k$.

\begin{definition} \label{def: F_omega} Let $   \widehat F_\omega$ be the free profinite group on a sequence  of  generators $x_0, x_1,  x_2 \ldots $   converging to $1$ \cite[Thm.\ 3.3.16]{Ribes.Zalesski:00}. \end{definition} 
 Thus, $   \widehat F_\omega$ is the   completion   in the sense of the previous subsection of  the   free group $F_\omega$ on   generators $x_0, x_1, \ldots$ with respect to  the system of normal   subgroups of finite index  that contain almost all the $x_i$.  Any   profinite group $G$   has a   generating sequence $\seq{g_i}\sN i$ converging to $1$.  This is easy  to see using coset representatives for a descending  sequence of open normal subgroups that form a fundamental system of nbhds of $1_G$. (Also  see \cite[Prop. 2.4.4 and 2.6.1]{Ribes.Zalesski:00}.) By the universal property of the completion, the map from the abstract free group induced by $x_i \to g_i$ extends to a continuous epimorphism $ \widehat F_\omega \to G$. So $G$   can be written in  the form \[    \widehat F_k / R\]
where $R$ is a closed normal subgroup of $   \widehat F_k$.

  \subsubsection*{``Almost everywhere"  theorems in profinite groups}  

\begin{thm}[Jarden; see \cite{Fried.Jarden:06}, 18.5.6] \label{thm:Jarden1} Let  $G = \mathrm{Gal}(\bar \QQ / \QQ)$ be the absolute Galois group of  $\QQ$. 
 For almost all tuples   $\sss= (\sss_1, \ldots, \sss_e) \in G^e$, the closed subgroup of $G$    topologically generated by $\sss$,   is a free profinite group of 
rank~$e$.  \end{thm} 
A group $G$ is called \emph{small} if it   has only finitely many subgroups of each index. Each small residually finite  (r.f.)group  is \emph{hopfian}, namely, every epimorphism $\alpha \colon G \to G$ is an isomorphism.   (Proof: let $V_n$ be the intersection of subgroups of index $\le n$, and check that $\alpha(V_n) = V_n$ for each $n$.) If $g \in \ker \alpha$  and $g \neq 1$ then $g \not \in  V_n$ for some $n$, so $\alpha(g) \neq 1$ as well.)

Every f.g.\ profinite group is small and r.f., and hence hopfian. It follows that $\sss$ above actually freely topologically generates $G$. 

A field $L$ is  PAC   (pseudo-algebraically closed) if  every (irreducible) variety over $L$ has a point in $L$. Besides algebraically closed fields, examples of PAC fields are the algebraic extensions $L$ of $\mathbb F_q$ with $|L:\mathbb F_q| = \omega$. See \cite{Fried.Jarden:06} for background. A field $L$ is  $\omega$-free if  $\mathrm{Gal}(L)$ is topologically isomorphic to $\hat F_\omega$. 
\begin{thm}[Jarden,  see Thms 18.6.1 and 27.4.8  in \cite{Fried.Jarden:06}] \label{thm:Jarden2} Let  $G = \mathrm{Gal}(\QQ)$ be the absolute Galois group of  $\QQ$. 
 For $\sss= (\sss_1, \ldots, \sss_e) \in G^e$ let $L_\sss=\QQ[\sss]$ denote the maximal Galois extension of $\QQ$ contained in the fixed field of $\sss$; equivalently, $L_\sss$   is the fixed field of the \emph{normal} closure of $\sss$.
 
 For almost all tuples   $\sss$, $L_\sss$ is PAC and $\omega$-free. 
   \end{thm}

Jarden and Lubotzky  \cite{Jarden.Lubotzky:99} study a related setting, namely $G= \hat F_n$. 
\begin{thm}[Jarden and Lubotzky,  Thm.\ 1.4 in \cite{Jarden.Lubotzky:99}] \label{thm: Jarden and Lubotzky} Let $G = \hat F_n$ for finite  $n\ge 2$. For almost all tuples   $\sss= (\sss_1, \ldots, \sss_e) \in G^e$, the closed \emph{normal}  subgroup they  topologically generate  either has finite index or is a free profinite group of 
rank $\omega$. The second case holds for all $\sss$ if $e< n$, for almost all $\sss$ if $e= n$, and for a set of $\sss$ with positive measure   if $e>n$.    \end{thm}

\subsection{The algorithmic theory}
This section has benefitted from discussions with A.\ Melnikov. 
 \subsubsection*{Effectiveness conditions on profinite groups}

\begin{definition}[Smith~\cite{Smith:81}] \label{def: computable profinite} \  \bi \item[(i)]   A profinite group $G$ is called \emph{co-r.e.} if it is the inverse limit  of a computable  inverse system $\la G_n, p_n\ra$ of finite groups  (i.e.\ the groups $G_n$ and the maps $p_n$ between them are uniformly computable). Equivalently, the subgroup $U$  above is a $\PI 1$ subclass of $\prod_{n} G_n$. \item[(ii)] $G$  is called \emph{computable} if, in addition,  the  maps $p_n$    can be chosen  onto. In other words, the set of extendible nodes in the tree corresponding to $U$ is computable.   \ei\end{definition}

%
%\subsubsection*{Aside:  computable   metric groups}
% Consider a   Polish metric space which is also a group:  $(G,\delta, \circ, {.}^{-1} )$ with the group operations continuous.   This structure is called computable if $(G , \delta)$  is a computable complete metric space  with respect to which  the group operations are computable.  Such a group is also a computable topological group in the sense of Section~\ref{Melnikov Nies Haar measure computable}. 
%  
% 
% If a profinite group is computable in the sense of Definition~\ref{def: computable profinite} then it is a computable Polish metric group via the usual ultrametric on $\prod_n G_n$, and hence a computable topological group. For a converse one should show that a computable topological group that is effectively compact and effectively 0-dimensionial (something like: the basis $\seq{B_n}$ in Section~\ref{Melnikov Nies Haar measure computable} above can be chosen clopen and is effectively closed under complement) is computable in the sense of profinite groups. Note that a computable profinite group enjoys these effectiveness properties. 
%  
% It would also be interesting to study computable categoricity. 
% E.g., is every computable   profinite group that is topologically finitely generated
% computably categorical?
 
 \subsubsection*{Absolute Galois groups}
 Let $K$ be a computable field.  Then the algebraic closure $\ol K$ has a computable presentation. ($\bar \QQ$ has a unique one, i.e.\ is autostable, by a result of Ershov.)  
 
  Suppose in addition that $K$ has   a splitting algorithm  (for polynomials in one variable), i.e.\ one can decide whether a polynomial is irreducible. An example of such a field is   $\QQ$. Then $\ol K$ has a computable presentation so that the $K$ viewed as a subset of $\ol K$ is decidable \cite[Lemma 6]{Rabin:60}.

   Also,  the absolute Galois group of $K$ is computable. To show this, intuitively, one builds a computable chain $K= L_0 \le L_1 \le \ldots $ of finite Galois extensions of $K$ with union $\bar K$. The computable inverse system is given by the groups $G_n = \mathrm{Gal}( L_n/K)$ where the projection $p_n \colon G_{n+1} \to G_n$ is given by restricting $\phi \in  \mathrm{Gal}( L_{n+1}/K)$ to $L_n$.

 \subsubsection*{Computable profinite groups that are completions}
 Suppose $G$ is a computable group, and the class $\+ V$ in Definition~\ref{def:completion} is uniformly computable in that there is a uniformly computable sequence  $\seq{R_n}$ such that $\+ V = \{ R_n \colon n \in \NN \}$. Suppose further that  $W$ above can be obtained effectively from $U,V$. Then there is a  uniformly computable  descending  subsystem  $\seq {T_k}$ of $\seq {R_n}$  such that  $\fa n \ex k \, T_k \le R_n$. Since we can effectively find coset representatives of $T_n$ in $G$, the inverse system  $\seq  {G/T_n}$ with the natural projections  $T_{n+1} a \to T_n a$ is computable. So $G_\+ V$ is computable.  
 
 The criterion above is satisfied by $F_k$ and $F_\omega$ with the systems of normal subgroups introduced above. Thus their completions $\hat F_k$ and $\hat F_\omega$ are computable profinite groups. 
 
 \begin{lemma} Let $G$ be $k$-generated ($ k \le \omega$). Then $G$ is computable [co-r.e.] iff $G= \hat F_k/N$ for some computable normal subgroup $N$ ($\PI 1$ N).  \end{lemma}

\subsubsection*{Computability of Haar measure} 
We use the notion of a  computable probability  space by Hoyrup and Rojas \cite{Hoyrup.Rojas:09a}.
 \begin{lemma} Let $G$ be computable profinite group. Then $\mu_G$, the Haar probability measure on $G$, is computable. \end{lemma}
 \begin{proof} The inverse system $\la G_n, p_n\ra$ is computable. So for a clopen set $C=q_n^{-1}(F)$ as in Remark \ref{rem:Haar}, given by the parameters $n$ and $F$, we can compute the measure. This suffices. \end{proof} 
 
 \begin{remark} {\rm Here is a more concrete description of the Haar measure. As before $q_n \colon G \to G_n$ is the natural projection. We also assume that $G_1 $ is trivial. We write $V_n =\ker {q_n}$. For each 
$n$ we can effectively determine  $k_n = |V_n : V_{n+1}|$ and  a sequence  $\seq{ g^{(n)}_{i}}_{i< k_n}$ of coset representatives for $V_{n+1}$ in $V_n$ such that  $g^{(n)}_0 = 1$.  

Let $T$ be the tree of   strings  $\sss \in \omega^{< \omega}$  such that  $\sss(i) < k_i $ for each $i< \sssl$. For $\sssl = n$  we have  a  coset of $V_n$ in $G$ \begin{equation*} \label{eqn:coset} C_\sigma = g^{(0)}_{\sss(0)} g^{(1)}_{\sss(1)} \ldots g^{(n-1)}_{\sss(n-1)} V_n.\end{equation*} 
The  clopen sets $C_\sss$ form  a basis for    $G$. In this way  $G$ is naturally homeomorphic to  $[T]$  where   the identity element corresponds to  $0^\omega$.  The Haar measure is the usual uniform measure on $[T]$. }\end{remark}

 \subsubsection*{Randomness notions defined via algorithmic tests}

 Let $G$ be a computable profinite group.  Given that the Haar measure $\mu$ on $G$  is computable, the usual randomness notions defined  via algorithmic tests, or effectively descriptive set theory tests, can be applied. 
 The usual question is: how strong a randomness notion on a group element $g$ suffices for  an  ``almost everywhere" properties  to hold for $g$? 
 
A point in a computable measure space is Kurtz random (or weakly random)  if it is  in no $\PI 1$ null class. For Jarden's Thm.\ \ref{thm:Jarden1},and at least the first part of his Thm.\ \ref{thm:Jarden2},  this  rather  weak randomness notion      is sufficient. If the underlying topological space is effectively homeomorphic to Cantor space, then any weakly 1-generic point  (i.e., in every dense $\SI 1 $ set) is  Kurtz random. As this applies to the setting of profinite groups at hand, the points for which Jarden's results hold are also comeager. 
 
\begin{thm}[effective form   of  Jarden's   Thm.\ \ref{thm:Jarden1} for $\QQ$] \label{thm:Jarden1rand} Let  $G = \mathrm{Gal}(\QQ)$ be the absolute Galois group of  $\QQ$. 
 Let  $\sss = (\sss_1, \ldots, \sss_e) \in G^e$ be  Kurtz random.  The closed subgroup  generated  by $\sss$   is a free profinite group of 
rank $e$ (freely generated by $\sss$).  \end{thm} 

Here is a sketch why this is true. (For a full proof,  thorough understanding of Jarden's result would be  needed, which is hard work  because of the long chain of dependencies and heavy notation leading up to \cite[18.5.6]{Fried.Jarden:06}.)  All item numbers  refer to   \cite{Fried.Jarden:06}.

Two extensions $L_1, L_2$ of a common field $K$  (tacitly assumed to be contained  in a common field) are called \emph{linearly disjoint} if whenever a tuple from $L_1$ is linearly independent over $K$, it remains linearly independent over $L_2$. By Lemma 2.5.1.\  this is symmetric. 

A sequence of extensions $L_1, L_2, \ldots$ is linearly disjoint if $L_{j+1} $ is l.d.\ from the compositum  $L_1 \ldots L_j$ for each~$j$. Given $k$, Cor.\  16.2.7.\ builds such a sequence of Galois extensions for $K= \QQ$ such that for each $n$ we have $\mathrm{Gal}(L_n/ \QQ ) \cong  S_k$ via some isomorphism $\rho_n$. The sequence of fields $\seq {L_n}$  is uniformly computable, as can be derived  from the proof (hopefully). 

%and they uniformly have a splitting algorithm?
By 18.5.1.\ the linear disjointness implies that the absolute Galois groups $\mathrm{Gal}(L_n)\le \mathrm{Gal}(\QQ)$ are $\mu$-independent;  recall that $\mu$ is the Haar measure on $G = \mathrm{Gal}(\QQ)$. More generally, by 18.5.2, for $e \ge 1$,  if for each $n$ one picks a  coset  $C_n$ of $\mathrm{Gal}(L_n)^e$ in $G^e$,  the $  C_n$ are $\mu^e$-independent. 

Now to conclude the argument, we want to show that each finite group $R$ generated by $e$ elements is a quotient of 
the closed subgroup $\la \sss \ra \le G$  generated  by $\sss$,  as this suffices to show freeness. Embed $R$ into $S_k$ for $k$ sufficiently large, and let $\pi_1, \ldots, \pi_e$ be the images of the generators of $R$ under this embedding.  

For each $n$ we have a natural onto homomorphism $G \to \mathrm{Gal}(L_n/\QQ)$ given by restriction to $L_n$ (note that $L_n$, being   a normal extension of $\QQ$,  is preserved under automorphisms of $\bar \QQ$). So we can effectively pick a coset $C_n$ of $\mathrm{Gal}(L_n)^e$ in $G^e$ that maps to $\la \pi_1, \ldots, \pi_e \ra$. Since the $L_n$ are independent, these cosets are independent,  and each  has measure   $1/(k!)^e$. Hence, by  Borel-Cantelli their union has measure $1$. The union is $\SI 1$ and hence contains the  Kurtz random   $\sss$. This shows that $R$ is a quotient of $\la \sss \ra$ as required.

\begin{thm}[effective form  of the first part of   Jarden's   Thm.\ \ref{thm:Jarden2} for $\QQ$] In the setting of Thm. \ref{thm:Jarden2}, recall that for an automorphism $\sss \in \mathrm{Gal} (\QQ)$, $L_\sss$   is the fixed field of the \emph{normal} closure of $\sss$. If $\sss$ is Kurtz random, then $L_\sss$ is PAC. \end{thm}
\begin{proof} We first show that $L_\sss$ is uniformly computable from (a code for) $\sss$. For $x \in \ol \QQ$ to be in $L_\sss$, it suffices that each conjugate $\tau^{-1}\sss \tau$ fixes $x$, or equivalently $\sss(y) = y$ for each $y$ in the orbit of $x$ under the action of $\mathrm{Gal}(\QQ)$. This orbit consists of all the conjugates of $x$, and can be computed from $x$. 

By Jarden's theorem  and the fact that a Kurtz random is in every $\PI 2$ conull set,  it now suffices to show that the PAC  subfields $K$ of $\ol \QQ$   form a  $\PI 2$ class. 

 Recall that for a unital ring $R$,  a nonconstant polynomial in $R(X_1, \to X_n)$ is called irreducible if it cannot be written as a product for two nonconstant polynomials.   For a field $K$, a polynomial in $K[X_1, \ldots, X_n]$ is called \emph{absolutely irreducible} if it is irreducible in $\ol K[X_1, \ldots, X_n]$.   We use some facts on PAC fields from \cite[Section 11.3]{Fried.Jarden:06}.  \cite[Theorem 11.2.3]{Fried.Jarden:06}  implies: 

\begin{prop} A field $K$ is PAC if and only if  for every absolutely irreducible polynomial $ f \in K[X,Y]$, there is a point $(a,b) \in K^2$ with $f(a,b) = 0$. \end{prop}

This yields the required $\PI 2$ condition as soon as we can express  using  a $\PI 1$ condition on the coefficients of $f$  that a polynomial $f \in K[X,Y]$  is absolutely irreducible. But irreducibility of $f$  in the polynomial ring $R[X_1, \ldots, X_n]$ is a $\PI 1$ condition on the coefficients of $f$ for any computable ring $R$, as one sees  directly by inspecting the definition.  \end{proof}

\begin{remark} The conditions (1) on page 200 in \cite[Section 11.3]{Fried.Jarden:06} yield a set of sentences in first-order logic expressing that a field is PAC. $S_K(2,d)$  denotes the set of polynomials $f \in K(X,Y)$ of degree $<d$ in both $X$ and $Y$.  They  say that for each irreducible $h \in \QQ(T)$, some $\PI 1$ condition holds (for each triple of polynomials $g_1, g_2, g_3) $  of bounded degree, something not involving quantifiers fails). It is decidable whether $h$ is irreducible as $\QQ$ has a splitting algorithm, so we can also get a co-r.e.\ condition on coefficients in this way. \end{remark}

%\cite[18.5.6]{Fried.Jarden:06}.) 

  \subsubsection*{Potential generalisation to Hilbertian fields} A field $K$ is Hilbertian if   every finite set of irreducible polynomials in a finite number of variables and having coefficients in $K$ admit a common specialization of a proper subset of the variables to field elements such that all the polynomials remain irreducible (Wikipedia).  Hilbert showed that $\QQ$ has this property (Hilbert's irreducibility theorem). All the classic a.e.\ theorems from \cite{Fried.Jarden:06}  mentioned above are stated  for any countable Hilbertian field, rather than just for  $\QQ$.  If  $K$ is computable and has the effective splitting property, then  one can expect that the effective versions   hold as well.

\section{Rute: On the computability of compact groups} \label{Nies Rute compact groups}
This note covers basic theorems about the computability of Haar measures and profinite groups.  The results are due to  Jason Rute and were proved after  discussions with Nies and Melnikov.  

 While we could work in the more general context of  computable topological spaces, we will only focus here on computable metric spaces (also known as computably presented Polish spaces), since they are well understood. 

Recall that a \emph{computable metric space} $\mathbb{X}$ is a complete separable metric space $(X,d)$ along with a dense sequence of points $(a_i)$ such that the map $i,j \mapsto d(a_i,a_j)$ is computable.  We say $\mathbb{X}$ is \emph{effectively compact} if one can enumerate all $\Sigma^0_1$ sets which cover $\mathbb{X}$.  It is easy to see that Cantor space is effectively compact.  We will need the following well-known properties about effectively compact spaces.

\begin{prop}\label{prop:covering}
Let $\mathbb{X}$ be a computable metric space along with a computable sequence $A_n$ where $A_n$ is a list of points $(a^n_1, \ldots, a^n_k)$ such that every point in $\mathbb{X}$ is within distance $2^{-n}$ of some $a^n_i$.  Then $\mathbb{X}$ is effectively compact.
\end{prop}

\begin{proof}
Using the double sequence $(a^n_i)$ we can construct a computable onto map $f \colon 2^\mathbb{N} \to \mathbb{X}$. (The details are routine but a bit technical.  A similar construction can be found in Simpson \cite[IV.1]{Simpson:09}.)  Now, we want to enumerate all effectively open covers of $\mathbb{X}$.  That is the same as enumerating all empty $\Pi^0_1$ subsets of $\mathbb{X}$.  Consider a $\Pi^0_1$ set $P \subseteq \mathbb{X}$.  The set $f^{-1}(P)$ is $\Pi^0_1$ subset of $2^\mathbb{N}$, and it is empty iff $P$ is empty.  Since $2^\mathbb{N}$ is effectively compact we will enumerate $f^{-1}(P)$ eventually if it is empty, and then we can use that to enumerate $P$.
\end{proof}

\begin{prop}\label{prop:singleton}
If $\mathbb{X}$ is an effectively compact computable metric space and $\{a\} \subseteq \mathbb{X}$ is a $\Pi^0_1$ singleton set, then $a$ is computable.  If $f  \colon \mathbb{X} \to \mathbb{X}$ is a function whose graph $f \subseteq \mathbb{X} \times \mathbb{X}$ is $\Pi^0_1$, then $f$ is computable.
\end{prop}

\begin{proof}
Let $U$ be the complement of $\{a\}$.  Compute $a$ by enumerating covers of the space of the form $U \cup B$ where $B$ is a basic open ball of small radius.  Similarly, use this method to compute $f(x)$ from $x$.
\end{proof}

\begin{definition} A \emph{computable topological Polish group} to be a computable metric space $G$ with a computable group operation  and a computable inverse operation.%
\footnote{One may object to this definition of ``computable Polish group'' since in the literature a Polish space is only defined as a completely metrizable space.  The choice of metric does not matter.  In a similar way, we could develop an equivalence relation on presentations of computable metric spaces.  We say that two presentations of a computable metric space are equivalent if the identity and its inverse are  computable.  For example, $\mathbb{R}^2$ with the Euclidean ($\ell_2$) metric and the $\ell_\infty$ metric are equivalent with any natural choice of dense set. Then we can define a \emph{computable Polish space} as the set of equivalence classes of computable metric spaces.  Similarly, we can use this idea to give a slightly more natural definition of \emph{computable Polish group}, again as equivalence classes.  The downside of this approach is that we need to constantly show that the properties we are interested in are preserved by this equivalence relation.  For example, if two presentations of a space $\mathbb{X}$ are equivalent, and one presentation is effectively compact, then so it the other presentation. Also, if two presentations of $\mathbb{X}$ are equivalent, then so are the various corresponding presentations of the space of Borel probability measures on $\mathbb{X}$.  } \end{definition} %  We omit the details.
The computable analogue of a compact group is a computable Polish group which is effectively compact.

\begin{fact}
If $G$ is an effectively compact computable metric space with a computable group operation, then the inverse operation is also computable, and therefore $G$ is a computable Polish group.
\end{fact}

\begin{proof}
Notice that the graph of the inverse function $\{(g, g^{-1}) \mid g \in G\}$ is a $\Pi^0_1$ set.  So the inverse map is computable by Proposition~\ref{prop:singleton}.
\end{proof}

If $\mathbb{X}$ is a computable metric space, then the space of Borel probability measures on $\mathbb{X}$ is also a computable metric space.  In particular, a Borel probability measure $\mu$ is computable if and only if the map $f \mapsto \int f d\mu$ is a computable map for all bounded computable functions $f \colon \mathbb{X} \to \mathbb{R}$.  If $\mathbb{X}$ is effectively compact, then so is the space of Borel probability measures on $\mathbb{X}$.  For more on the computability of Borel probability measures, see Hoyrup and Rojas \cite{Hoyrup.Rojas:09} or Bienvenu, Gacs, Hoyrup, Rojas, and Shen \cite{Bienvenu.Gacs.ea:11}.

\begin{prop}
Let $G$ be an effectively compact computable Polish group.  Then the left and right  Haar probability measures  are  computable.
\end{prop}

\begin{proof}
The set of, say, left  Haar probability measures is a $\Pi^0_1$ singleton set in the effectively compact space of Borel probability measures on $\mathbb{X}$.  By Proposition~\ref{prop:singleton} the left Haar measure is computable.
\end{proof}

\begin{thm}
Let $G$ be a compact  computable Polish group for which the  (left) Haar probability measure is computable.  Then $G$ is effectively compact.
\end{thm}

\begin{proof}
 We re-metrize the space $G$ by replacing $d(x,y)$ with the average of $d(gx,gy)$ where the average (the integral) is taken in the Haar measure as g varies across the group.  This new metric is computable since $d(x,y)$ is a bounded computable function.  (The metric is bounded by the compactness of $G$.)  Now one has a computable $G$-invariant distance which is equivalent to the original distance.

Now, to show that $G$ is effectively compact in this new metric, it is enough for each rational $k$, to effectively find a finite set of points $a^k_0, ..., a^k_{n-1}$ for which every point in $G$ is within distance $2^{-k}$ of one of these points.  Fix $k$.  Using our Haar measure find the measure of a ball of radius $2^{-(k+1)}$.  Call this measure $\delta$.  (Since the new distance is $G$-invariant, all balls of the same radius have the same measure.)  Using blind search find a collection of balls $B_0, ..., B_{n-1}$ of balls with radius $2^{-(k+1)}$ whose union $C = B_0 \cup ... \cup B_{n-1}$ has measure $> 1 - \delta$.  Now, consider any point $x$ not in this union $C$.  It has to be distance $<2^{-(k+1)}$ from the union.  Otherwise, there would be a ball centered at $x$ with radius $2^{-(k+1)}$, and hence measure $\delta$, which is disjoint from the union $C$.  But the union $C$ has measure $> 1 - \delta$, so this cannot happen.  Therefore all points of $G$ are within distance $2^{-k}$ of the centers of $B_0, ..., B_{n-1}$.  This algorithm shows that the space is effectively compact in the new metric.  To show it is effectively compact in the original metric, for any finite list of rational balls in original metric, convert it to a list of balls in the new metric.  Now, if this list of balls covers the space $G$, by effective compactness, we will eventually find this out.
\end{proof}

\begin{example}
There is a computable group $G$ isomorphic to a product of finite cyclic groups $\Pi_n G_n$ which is compact but for which the Haar measure is noncomputable.
\end{example}

\begin{proof}
Let $h(n)$ be the characteristic function of the Halting set.  Then we will let $G = \prod_n \mathbb{Z}_{2^{h(n)}}$ with the ultrametric $\rho(f,g) = \inf \{2^n \mid f(n) = g(n)\}$.  This is a computable metric space (but it is not effectively compact).  It is also a computable Polish group.  The Haar measure of any cylinder set of length $n$ is equal to $\prod_{k<n} {2^{-h(n)}}$.  If we could compute the Haar measure, then we could compute $h(n)$.
\end{proof}

As Section~\ref{Fouche-Nies-groups} in this  year's Logic Blog   mentioned, a profinite group is an inverse limit of finite groups.  So there exists a descending chain of normal clopen subgroups $N_s \vartriangleleft G$ which converges to the identity.  Then $G$ is the inverse limit of $G/N_s$.  A $G$ is computably profinite.

\begin{thm}
Let $G$ be a profinite  effectively compact computable Polish group.  Then we can compute a sequence of finite groups $G_s$ such that $G$ is the inverse limit $G_s$ and the corresponding homomorphism $h_s \colon G\to G_s$ is computable.
\end{thm}
%(but not necessarily computably profinite)
\begin{proof}
Since $G$ is profinite, we need to find a descending chain of normal clopen subgroups $N_s \vartriangleleft G$.  Since $G$ is effectively compact, we can enumerate all clopen sets $A$ by finding disjoint covers made up of finitely many open balls $B_1, \ldots B_n, C_1, \ldots, C_m$ where $B_i$ and $C_j$ are disjoint.  We can then use this to enumerate all normal clopen subgroups as follows.  Since each clopen set $A$ is both effectively open and closed, the property
\[\forall g\in G\quad gAg^{-1} \subseteq A\]
is $\Pi^0_1$ (if the property fails, then search for some $g\in G$ and some $a\in A$ such that $gag^{1}$ is outside of $A$) and $\Sigma^0_1$ (if the property holds, then wait for an enumeration of balls which cover $gAg^{-1}$ and which are disjoint from a cover of the complement of $A$).  Therefore, we can enumerate all clopen normal subgroups $N \vartriangleleft G$.  

Let $N_s$ be the intersection of all clopen normal subgroups enumerated at stage $s$ of the construction (where $N_0 = G$).  This is also a clopen normal subgroup.  Now enumerate all cosets $gN_s$.  (These can be enumerated since each coset $gN_s$ is also clopen.  Also we know when we have enumerated them all since they form an open cover of an effectively compact space.)  From this we can compute $G/N_s$ along with the corresponding homomorphism.   (If $G$ is finite, then at some stage, $G/N_s$ is isomorphic to $G$.)
\end{proof}
  
  \part{Metric spaces and descriptive set theory}
\section{Nies and Weiss: \\ complexity of topological isomorphism for subshifts}

Let $\Sigma$ be a finite alphabet.
A \emph{subshift} is a closed  subset $X \sub \Sigma^ \ZZ$ which is invariant under the shift $\sss$. We consider the complexity of the isomorphism relation $(X, \sss_X) \cong (Y, \sss_Y)$ where $\Sigma,\Delta$ are  finite alphabets and $X \sub \Sigma ^ \ZZ$,  $Y \sub \Delta ^ \ZZ$ are subshifts. To be isomorphic means that  there is a continuous bijection $\theta: X \to Y$ such that $\theta (\sss_X(z) ) = \sss_Y( \theta(z))$ for each $z \in X$. Note that $\theta $ is given by the clopen sets $\theta^{-1}(\{ Z \colon \, Z(0)=b\}) $ for $b \in \Delta$, which are of the form $ \bigcup_{i<k} [\alpha_i]$ where $\alpha_i: [-N, N] \to \Sigma$ (such a collection of clopen sets is called a block code). So isomorphism  is a countable Borel equivalence relation.

 J.\ Clemens (Israel J. of Maths, 2009) proved that isomorphism is  in fact a $\le_B$-complete countable Borel equivalence relation. 
 
 A subshift is \emph{minimal} if every orbit of an element is dense; equivalently, every possible pattern (i.e.\ subword of a fixed length of an element of   the subshift) occurs in a large enough section of every element of the subshift. By compactness, the length of this section  only  depends on the pattern.  As a consequence, minimality  is a Borel property of subshifts. Also, it is sufficient to require the property that patterns re-occur within a distance only dependent on the pattern for one word with dense orbit. 
 
Clemens  asked in the paper and in a 2014 talk (available on youtube) the following question: 

\begin{question} How complex is the  isomorphism relation between   minimal subshifts? \end{question}  
 Gao, Jackson and Seward  \cite[Section 9.3]{Gao.Jackson.etal:16} proved that $E_0$ is Borel reducible to isomorphism of minimal subshifts. 
 
 \begin{proof} Here is a sketch of  a short proof of this  fact. % (details to be checked).
 
  The idea is to build a sequence of blocks
${A_n, B_n}$ of length $L_n = (66)^{n}$.  The construction
is controlled by a fixed element $x \in \{0,1\}^\NN$
with the n-stage blocks a function of $x_i$ for
$1 \leq i  \leq n$. This sequence of blocks determines a subshift $S_x$: the allowed patterns of length $L_n$ are   the  $A_n$
and $B_n$. 

The blocks ${A_{n+1}, B_{n+1}}$
will be built out of the n-stage blocks in such a way that any bi-infinite sequence formed by concatenating these two
blocks has a unique parsing into blocks of these two types. This
parsing defines for each bi-infinite word $\seq {z(i)}_{i \in \ZZ}$ a unique integer modulo
$L_n $ which indicates the position of z(0) in the block.
   
   Recall that an odometer is a dynamical system that is  an inverse limit of periodic rotations.  The simplest example is the $2$-adic integers with addition by $1$. Since $L_{n+1}$ is a multiple of $L_n$,   the position of $z(0)$ modulo $L_{n+1}$
when reduced modulo $L_n$ gives the position of $z(0)$ in its ``n-block".
   This will yield a common odometer  as a factor of all   the minimal shifts.

Let $A_0 = 0$ and $B_0 = 1$. We  describe 4 recipes for concatenating ${A,B}$:

   \bi \item[1.] $ AB(A^4B^4)^8$

   \item[2.] $ AB(A^8B^8)^4$

   \item [3.] $ AB(A^{16}B^{16})^2$

   \item [4.] $ AB{A^{32}B^{32}}$. \ei

\n Depending upon the value of $x_{n}$
the ${A_{n+1}, B_{n+1}}$ will be formed in two different
ways.  If $x_{n+1}=1$ one uses   1 and 2,   otherwise one uses 
3 and 4. Assuming that one can recognize A and B,  the initial ABA in all
recipes guarantees that the concatenations have a unique parsing.

   The minimality is immediate since all four possibilities
AA, AB, BA, BB   occur. 

  Each specfic minimal system has  its own
collection of the two types of blocks,  where the nature of the blocks
up to level n depend only on the first n bits of the control element
from ${0,1}^\NN$.  Clearly if x and y agree from some point on
there is a finite block code that will map one shift to the other.
 If x and y differ infinitely often then no matter what the
length  of the code eventually it will pick up the pattern of
repetitions which differ significantly in all 4 recipes.
 \end{proof}

Simon Thomas \cite{Thomas:13} has given a Borel reduction of $E_0$ to a special  class of minimal subshifts, the Toeplitz subshifts.  A Toeplitz word is a bi-infinite word $W$ such that for $n\in \ZZ$ there is a ``local period" $k\in \ZZ$ such that $\fa i\in \ZZ  \, [W(n) = W(n+ik)]$.  A subshift is Toeplitz if it contains a Toeplitz word with dense orbit. To see that this is minimal, suppose that $w$ is a subword of $W$, and let $r$ be the l.c.m.\ of the local periods of any symbol in $w$. Then $\sigma^r(w)$ is also a subword of $W$, for any $r \in \ZZ$.

 Not much  beyond that   is known so far on the complexity of conjugacy for minimal subshifts. 

 \part{Higher computability theory/effective descriptive set theory}
   
\section{Yu:   $\Pi^1_1$-hyperarithmetic determinacy}
Input by Yu.

The following theorem was claimed in \cite{Harrington:78}: the hyperdegrees of a $\Pi^1_1$ set with a perfect subset contain an upper cone. 
\begin{theorem}[Harrington \cite{Harrington:78}]\label{theorem: pi11 det}
Let  $A\subseteq 2^{\omega}$ is a $\Pi^1_1$ set that  contains a perfect subset. There is a real $z\in 2^{\omega}$ so that for any real $y\geq_h z$, there is a real $x\in A$ for which $x\equiv_h y$.
\end{theorem}

The following proof is based on some communications with Leo Harrington. His original  idea seem  model theoretical. Here is a tree proof.

The following theorem is proved by Martin.
\begin{theorem}[Martin \cite{Martin76}]\label{theorem: friedman}
If $A$ is an uncountable $\Delta^1_1$-set, then  for any real $y\geq_h \KO$, there is a real $x\in A$ for which $x\equiv_h y$.
\end{theorem}

Fix an uncountable $\Pi^1_1$ set $A$ throughout in this section.
Since $A$ is $\Pi^1_1$, there is a recursive oracle functional
$\Phi$ so that $$x\in A \Leftrightarrow \Phi^x \mbox{ codes a well
ordering of $\omega$.}$$
In other words, the binary relation $n\leq_x m$ if and only if $\Phi^x(\langle n,m\rangle)=1$ is a well  ordering over $\omega$. We use $n\in Dom(\Phi^x)$ to denote that there is some $m$ so that $\Phi^x(\langle n,m\rangle)=1$ or $\Phi^x(\langle m,n\rangle)=1$. For a finite binary string $\sigma$, we also use $<_{\Phi^{\sigma}}$ to denote the finite linear order coded by
$\Phi^{\sigma}$.

Let
$$\beta=\min\{\beta\mid |\{ x \mid \Phi^x \mbox{ codes a well
ordering of $\omega$ with order type } \beta\}|>\aleph_0\}.$$

Since $A$ has a perfect subset,  such $\beta$ must exist.

We fix the $\beta$ throughout.

We associate a tree $T$ with $A$ by defining $(\sigma,\tau)\in T$ if
and only if
\begin{enumerate}
\item $\sigma \in 2^{<\omega}$ and;
\item $\tau$ is finite order preserving function from $Dom(\Phi^{\sigma})$
to ordinals.
\end{enumerate}

We always assume that $|Dom(\Phi^{\sigma})|=|\sigma|$. 

$(\sigma_0,\tau_0)\preceq (\sigma_1,\tau_1)$ if both $\sigma_1$ and $\tau_1$ extends $\sigma_0$
and $\tau_0$ respectively. $(\sigma_0,\tau_0)$ is at the left of $(\sigma_0,\tau_0)$
if \begin{enumerate}\item for some are $m\leq \min\{|\sigma_0|,|\sigma_1|\}$,
$\sigma_0\uh m=\sigma_1\uh m$ but $\sigma_0(m+1)<\sigma_1(m+1)$ or \item for every $\sigma_0\uh
\min\{|\sigma_0|,|\sigma_1|\}=\sigma_1 \uh \min\{|\sigma_0|,|\sigma_1|\}$ and for some $k$,
$\tau_0(k)<\tau_1(k)$ but $\tau_0(j)=\tau_1(j)$ for all $j<_{\Phi^{\sigma}}k$.

\end{enumerate}

Then $x \in A$ if and only if there is an $f$ so that $(x,f)$ is an
infinite branch of $T$.

Let $T^{\beta}$  be the tree $T$ restricted to the ordinal
$\beta$, i.e. the range of every $\sigma$ is a subset $\beta+1$.
Obviously $T^{\beta}\in L_{\omega_1^{\beta}}$. Moreover, there are
uncountable many infinite branches in $T^{\beta}$ by the definition
of $\beta$. Let $$A_{\beta}=\{x\mid \Phi^x \mbox{ codes a well
ordering of $\omega$ of order type } \leq\beta\}.$$

Then $A_{\beta}$ is exactly the collection of reals $x$ for which
there is an $f$ so that $(x,f)$ is an infinite branch through
$T^{\beta}$.

Obviously $A_{\beta}$ is an uncountable Borel set containing a perfect subset.

Let $\omega_1^{\beta}$ be the least admissible ordinal greater than $\beta$. Obviously $T^{\beta}\in L_{\omega_1^{\beta}}$.

\subsection{Case(1):  There is a real $z $ so that $z\in L_{\omega_1^{\beta}}$ and $\omega_1^z=\omega_1^{\beta}$.}

Then let $B_{\beta}=\{x\mid \Phi^x \mbox{ codes a well
ordering of order type }  \beta\}\subseteq A_{\beta}$ be a $\Delta^1_1(z)$-set. Then for any real $x\in B_{\beta}$, $x\geq_h z$. Relativizing the proof of Theorem \ref{theorem: friedman} to $z$, we may have Theorem \ref{theorem: pi11 det}.

\subsection{Case(2): Otherwise.}  

Then for any real $x\in A_{\beta}$, if $\Phi^x$ codes a well ordering of $\beta$, then $x\not\in L_{\omega_1^\beta}$.

Fix a recursive enumeration of set theoretical $\Sigma_0$-formulas $\{\varphi_i(u,v,\beta)\}_{i\in \omega}$ with $\beta$ as a parameter. Then $\omega_1^{\beta}$ is the least ordinal $\gamma>\beta$ so that for any $i$, $$L_{ \gamma}\models \forall u<\beta\exists v \varphi_i(u,v,\beta)\rightarrow \exists w\forall u<\beta\exists v\in L_{w} \varphi_i(u,v,\beta).$$

We also do a Cantor-Bendixon derivation to $T^{\beta}$. I.e.  $T_{0}^{\beta}=T^{\beta}$; and for any stage $\gamma<\omega_1^{\beta}$ and $(\sigma_1,\tau_1)\succ (\sigma_0,\tau_0)\in T_\gamma^{\beta}$, if
\begin{enumerate}
\item either there exists an order preserving (in the $<_{KB}$ sense ) function $f \in L_{\beta}$ so that $f:   T_1^{\beta}[(\sigma_1,\tau_1)]\to \beta$; or
\item there exists a real $x\in L_{\beta}$ so that $\{x\}=\{z\succ \sigma_1\mid \exists f \succ \tau_1\forall n(x\uh n, f\uh n)\in T^{\beta}_{\gamma}\}$,
\end{enumerate}
 then we let $T_1^{\beta+1}=T_1^{\beta}\setminus [(\sigma_1,\tau_1)]$ and claim that $(\sigma_1,\tau_1)$ is cut off from $T_\gamma^{\beta}$ at stage $\gamma$.
 
 If $\gamma$ is a limit stage, then $T_{\gamma}^{\beta}=\bigcap_{\gamma'<\gamma}T_{\gamma'}^{\beta}$.
 
 Let $$T_{\omega_1^{\beta}}^{\beta}=\bigcap_{\gamma<\omega_1^{\beta}}T_{\gamma}^{\beta}.$$
 
 Obviously $T_{\omega_1^{\beta}}^{\beta}$ is not empty.

 Let $T^1\subseteq 2^{\omega}\times \omega^{<\omega}$ be a recursive tree so that $p[T^1]=\{x\mid \exists f (x,f)\in [T^1]\}=\{x\mid x\not\in L_{\omega_1^x}\}$.

Since $A_{\beta}$ contains a perfect subset, $p[T^1]\cap A_{\beta}\neq \emptyset$. 

\begin{lemma}\label{lemma: not in beta}If $x\in p[T^1]\cap A_{\beta}$, then $ x\not\in L_{\omega_1^{\beta}}$ and $\omega_1^x\geq \omega_1^{\beta}$. Moreover, if $\Phi^x$ codes a well ordering of order type less than $\beta$, then $x\in L_{\omega_1^x}$ and $\omega_1^x\leq  \omega_1^{\beta}$. \end{lemma}
\begin{proof}
Fix a real $x\in A_{\beta}$.

If $\Phi^x$ codes a well ordering of order type $\beta$, then $\omega_1^x\geq \omega_1^{\beta}$. So $x\not\in L_{\omega_1^{\beta}}$.

If $\Phi^x$ codes a well ordering of order type $\gamma$ less than $\beta$, then $A_{\gamma}$ is a countable set which is $\Delta^1_1(z)$ for any real $z$ with $\omega_1^z\geq \gamma$. Then $x\leq_h z$ for any real $z$ with $\omega_1^z\geq \omega_1^{\gamma}$. But $\omega_1^x\geq \omega_1^{\gamma}$. So $x\in L_{\omega_1^{\gamma}}$ and $\omega_1^{\gamma}=\omega_1^x$. Hence $x\not\in p[T^1]$.
\end{proof}

So if $x\in p[T^1]\cap A_{\beta}$, then $x>_h \KO$.

Let $$T^2=T^1\otimes T^{\beta}=\{(\sigma_0,\sigma_1,\sigma_2)\mid (\sigma_0,\sigma_1)\in T^1 \wedge (\sigma_0,\sigma_2)\in T^{\beta}\}.$$

Obviously $T^2\in L_{\omega_1^\beta{}}$ and $[T^2]$ is not empty. Let $(x,f,h)$ be the leftmost infinite path through $T^2$. Then $x\in p[T^1]\cap A_{\beta}$ and so by Lemma \ref{lemma: not in beta},  $x\in L_{\omega_1^{\beta}+2}\setminus L_{\omega_1^{\beta}}$. In other words, there must be a master code in $L_{\omega_1^{\beta}+2}\setminus L_{\omega_1^{\beta}}$. 

{\em Fix a standard master code $z_0 \in L_{\omega_1^{\beta}+2}\setminus L_{\omega_1^{\beta}}$}.

Now let $y_0>_h z_0$ be a real.

\begin{definition}

Given a tree $S\subset 2^{<\omega}\times \alpha^{<\omega}$ where $\alpha$ is an ordinal, a finite pair $(\sigma,\tau)\in S $ is called a {\em splitting node} in $S$ if for any $i\leq 1$, there is some $\gamma_i$ so that $(\sigma^{\smallfrown}i,\tau^{\smallfrown}\gamma_i)\in S$.

\end{definition}

\begin{definition}
Given an infinite path $(x,f)\in [T_{\omega_1^{\beta}}^{\beta}]$, a number $n$ and an ordinal $\gamma\leq \omega_1^{\beta}$, we say that {\em $\gamma$ is correct up to $n$ respect to $(x,f)$} if for any $i\leq n$, $(x\uh i,f\uh i)$ is a splitting node  in $T_{\omega_1^{\beta}}^{\beta}$ if and only if $(x\uh i,f\uh i)$ is a splitting node  in $T_{\gamma}^{\beta}$.
\end{definition}
 
 The following lemma is clear.
 \begin{lemma}\label{lemma: correct lemma}
 Suppose that  $(x,f)\in [T_{\omega_1^{\beta}}^{\beta}]$, $n$ is a number and $\gamma\leq \gamma'\leq \omega_1^{\beta}$. If $\gamma$ is correct up to $n$ respect to $(x,f)$, the so is $\gamma'$.
 \end{lemma}

 The following definition is crucial to the proof. Intuitively we use even parts to code $y_0$ so that we may find a very large stage at which we may witness whether $\varphi_{i}(u,v,\beta)$ can be satisfied. However we use the odd parts to indicate when the coding stage is finished.
 
 \begin{definition}\label{definition: correct definition}Given a finite pair $(\sigma,\tau)\in T_{\omega_1^{\beta}}^{\beta}$, let {\em $(x_{\sigma},f_{\tau })\in [T_{\omega_1^{\beta}}^{\beta}[\sigma,\tau]]$} be an infinite path  satisfying the following properties:
 \begin{itemize}
 \item If $f_{\tau }\uh  (n)$ is the least ordinal $\gamma$ so that $[T_{\omega_1^{\beta}}^{\beta}[x_{\sigma}\uh n+1 ,f_{\tau }\uh n^{\smallfrown }\gamma]]\neq \emptyset$; and
 \item  If $n$ is the $2k$-th number (in the natural ordering sense) for some $k>0$ in $SP_{ \sigma, \tau }=\{j\mid (x_{\sigma}\uh j,f_{\tau }\uh j)\in [T_{\omega_1^{\beta}}^{\beta}[\sigma,\tau]] \mbox{ is a splitting node}\}$, then $x_{\sigma }(n+1)=x_{\sigma }(n)^{\smallfrown}y_0(k)$; and
 \item If $n$ is the $2k+1$-th number for some $k\geq 0$ in $SP_{ \sigma, \tau } $, then $x_{\sigma }(n+1)=x_{\sigma }(n)^{\smallfrown}0$.
 \end{itemize}
 \end{definition}
 Obviously given any number $n$ and pair $(\sigma,\tau)\in T_{\omega_1^{\beta}}^{\beta}$, $(x_{\sigma },f_{\tau })$ always exists and $f_{\tau} \in L_{\max\{\omega_1^x , \omega_1^{\beta}\}}[x]$.
 
 \begin{lemma}\label{lemma: omega1 xsigma greater beta}
 $\Phi^{x_{\sigma }}$ codes a well ordering of $\omega$ with order type $\beta$.
 \end{lemma}
 \begin{proof}
 Otherwise, by Lemma \ref{lemma: not in beta}, $x\in L_{\beta}$. So by a zig-zag decoding argument over $T_{\omega_1^{\beta}}^{\beta}$, $y_0\leq x\oplus z_0\equiv_h z_0$, a contradiction.
 \end{proof}
 
 \begin{lemma}\label{lemma: nice way correct}
 If there is a pair $(\sigma,\tau)\in T_{\omega_1^{\beta}}^{\beta}$  and some stage $\gamma< \omega_1^{\beta}$ so that for any $n$, $\gamma$ is correct up to $n$ respect to $(x_{\sigma },f_{\tau })$, then $ x_{\sigma }\equiv_h y_0$.
 \end{lemma}
 \begin{proof}
 By the property of $(x_{\sigma },f_{\tau })$, $f_{\tau }\in L_{\omega_1^{x_{\sigma }}}[x_{\sigma }]$. We claim that $\omega_1^{\beta}\leq \omega_1^{x_{\sigma }}$. Otherwise, by Lemma \ref{lemma: not in beta}, $x_{\sigma}\in L_{\beta}$. By by a zig-zag coding over $T_{\gamma}^{\beta}[\sigma,\tau]$, we have that $y_0\in L_{\omega_1^{\beta}}[x_{\sigma}]=L_{\omega_1^{\beta}}$ , a contradiction to the choice of $y_0$. 
 
 So $\gamma<\omega_1^{\beta}\leq \omega_1^{x_{\sigma }}$. Then it is clear that, by a zig-zag decoding over  $T_{\gamma}^{\beta}[\sigma,\tau]$, we may decode $y_0$   by $(x_{\sigma }, f_{\tau })$. So $y_0\leq_h x_{\sigma }$. Obviously $y_0\geq_h  x_{\sigma }$. So $ x_{\sigma }\equiv_h y_0$.
 \end{proof}
 
 So if the assumption of Lemma \ref{lemma: nice way correct} holds, then the proof of Theorem \ref{theorem: pi11 det} is finished.
 
 \bigskip
 
 {\em From now on, we assume for any pair $(\sigma,\tau)\in T_{\omega_1^{\beta}}^{\beta}$ and any ordinal $\gamma< \omega_1^{\beta}$, there is some number $n$ so that $\gamma$ is not correct up to $n$ respect to $(x_{\sigma },f_{\tau })$.}
 
 \bigskip
 
 Now we turn to the real construction.  We will construct an infinite path $(x,f)\in T_{\omega_1^{\beta}}^{\beta}$ so that $y_0\equiv_h x$. To code $y_0$, we use a zig-zag coding which is performed in $L_{\omega_1^{\beta}+1}$. So the point is  show $\omega_1^x>\omega_1^{\beta}$.  
 
 We start to construct $(x,f)$ by induction on $\omega$.
 
 At stage $0$. Let $(\sigma_0^0,\sigma^1_0)=(\emptyset,\emptyset)\in T_{\omega_1^{\beta}}^{\beta}$.

 At stage $s+1$. Suppose that $(\sigma_s^0,\sigma^1_s)\in T_{\omega_1^{\beta}}^{\beta}$ has been constructed so that  $(\sigma_{s}^0,\sigma^1_{s})$  is a splitting node  in $T_{\omega_1^{\beta}}^{\beta}$ (Without loss of generality, we may assume that $(\emptyset,\emptyset)$ is a splitting node in $ T_{\omega_1^{\beta}}^{\beta}$.).

 Substage (1): We code $y_0(s)$ at this substage. 
 Let $\sigma_{s,1}^0$ be the  shortest finite string so that there is a  string $\sigma^1_{s,1}$ such that
 \begin{itemize}
 \item[(1)] $(\sigma_{s,1}^0,\sigma^1_{s,1})\in T_{\omega_1^{\beta}}^{\beta}$ is a splitting node; and
 \item[(2)] $\sigma_{s,1}^0\succeq (\sigma_s^0)^{\smallfrown}y_0(s)$; and
 \item[(3)] $(\sigma_{s,1}^0,\sigma^1_{s,1})$ is the leftmost string  in $\{(\sigma_{s,1}^0,\tau)\mid (\sigma_{s,1}^0,\tau)\in T_{\omega_1^{\beta}}^{\beta}\}$. \end{itemize}
 
Obviously such a pair $(\sigma_{s,1}^0,\sigma^1_{s,1})$ exists.

Substage(2): We try to make sure $\omega_1^x>\omega_1^{\beta}$ at this stage. Let $(x_{\sigma_{s,1}^0}, f_{\sigma^1_{s,1}}) \in T_{\omega_1^{\beta}}^{\beta}[\sigma_{s,1}^0,\sigma^1_{s,1}]$ be as defined in Definition \ref{definition: correct definition}.

  Case(2.1). $L_{\omega_1^{\beta}}\models \forall u<\beta\exists v \varphi_i(u,v,\beta)$. Then let $\gamma_s$ be the least ordinal so that $\forall u<\beta\exists v\in L_{\gamma_s} \varphi_i(u,v,\beta)$. Then by the assumption, we may let $n_{s}$ be the least number $n$ so that
  \begin{itemize} 
  \item $(x_{\sigma_{s,1}^0}\uh n, f_{\sigma^1_{s,1}}\uh n)$ is a splitting node in  $T_{\omega_1^{\beta}}^{\beta}$; and 
  \item $n$ is the $2k+1$-th number for some $k\geq 0$ in $SP_{\sigma_{s,1}^0,\sigma^1_{s,1}}=\{j\mid (x_{\sigma_{s,1}^0}\uh j, f_{\sigma^1_{s,1}}\uh j) \in T_{\omega_1^{\beta}}^{\beta}[\sigma_{s,1}^0,\sigma^1_{s,1}] \mbox{ is a splitting node}\}$;
  \item $\gamma_s$ is not correct up to $n$ respect to $(x_{\sigma_{s,1}^0}, f_{\sigma^1_{s,1}})$.
  \end{itemize} 
  
  Then let $(\sigma_{s+1}^0, \sigma^1_{s+1})$ be a finite string such that
 \begin{itemize}
 \item[(1)] $(\sigma_{s+1}^0,\sigma^1_{s+1})\in T_{\omega_1^{\beta}}^{\beta}$ is a splitting node extending $( x_{\sigma_{s,1}^0}\uh n_s,  f_{\sigma^1_{s,1}}\uh n_s)$; and
 \item[(2)] $\sigma_{s+1}^0\succeq x_{\sigma_{s,1}^0}\uh n_s^{\smallfrown}1$ (we use this to indicate the coding construction at this stage is finished); and
 \item[(3)] $(\sigma_{s+1}^0,\sigma^1_{s+1})$ is the leftmost string  satisfying above property. \end{itemize}
 
 Case(2.2). Otherwise. Then there is some $u<\beta$ so that $L_{\omega_1^{\beta}}\models \forall v \neg\varphi_i(u,v,\beta)$.
 
 Then by the assumption, let $n_{s}$ be the least number $n$ so that
  \begin{itemize} 
  \item $(x_{\sigma_{s,1}^0}\uh n_s, f_{\sigma^1_{s,1}}\uh n)$ is a splitting node in  $T_{\omega_1^{\beta}}^{\beta}$; and 
  \item There is some $d\in Dom(\Phi^{x_{\sigma_{s,1}^0}\uh n})$ so that $f_{\sigma^1_{s,1}}\uh n(d)=u$ (remember that $f_{\sigma^1_{s,1}}\uh n$ is a finite order preserving function from $Dom(\Phi^{x_{\sigma_{s,1}^0}\uh n})$ to $\beta$); and
  \item $n$ is the $2k+1$-th number for some $k\geq 0$ in $SP_{\sigma_{s,1}^0,\sigma^1_{s,1}}$;
  \end{itemize} 
  
  Since $\Phi^{x_{\sigma_{s,1}^0}}$ codes a well ordering of order type $\beta$, such a number $n_s$ must exist.
   
 Then let $(\sigma_{s+1}^0,\sigma^1_{s+1})$ be a finite string such that
 \begin{itemize}
 \item[(1)] $(\sigma_{s+1}^0,\sigma^1_{s+1})\in T_{\omega_1^{\beta}}^{\beta}$ is a splitting node extending $( x_{\sigma_{s,1}^0}\uh n_s,  f_{\sigma^1_{s,1}}\uh n_s)$; and
 \item[(2)] $\sigma_{s+1}^0\succeq x_{\sigma_{s,1}^0}\uh n_s^{\smallfrown}1$ (we use this to indicate the coding construction at this stage is finished); and
 \item[(3)] $(\sigma_{s+1}^0,\sigma^1_{s+1})$ is the leftmost string  satisfying above property. \end{itemize}
This finishes the coding construction at stage $s+1$.

\bigskip

 Let $$(x,f)=\bigcup_{s\in \omega}(\sigma_{s}^0,\sigma^1_{s}).$$
 
 By the construction, $f$ is an automorphism between $\Phi^x$ and an initial segment of $\omega_1^x$. By the same proof of Lemma \ref{lemma: omega1 xsigma greater beta}, $\Phi^x$ codes a well ordering of $\omega$ with order type $\beta$ and so $\omega_1^x\geq \omega_1^{\beta}$.  Hence $f$ is an  automorphism between $\Phi^x$ and $\beta$.
 
 We use a method in \cite{Yu11} to  decode the coding construction. We shall $x$-hyperarithmetically construct an increasing sequence ordinals $\{\alpha_i\}_{\i\in \omega}$ so that $\lim_i \alpha_i=\omega_1^{\beta}$. Then $\omega_1^x>\omega_1^{\beta}$. Once this is archived, then by a zig-zag decoding, we have that $x\geq_h y_0$ and so $x\equiv_h y_0$.
 
 \begin{definition}Given a finite increasing sequence $\{n_i\}_{i\leq s}$ for some $s$ and an ordinal $\gamma<\omega_1^{\beta}$, we say that $\gamma$  {\em matches  $\{n_i\}_{i\leq s}$} if all the following facts hold:
 \begin{itemize}
 \item $n_0=0$; and
 \item For any $l\leq n_s$, $(x\uh l,f\uh l)$ is the leftmost in $\{(x\uh l,\tau)\mid (x\uh l,\tau)\in T_{\gamma}^{\beta}\}$; and 
 \item For any $j\in (0,s]$, 
   \begin{itemize}
   \item There is a number $l_0$ which is the least number greater than $n_{j-1}$ so that $(x\uh l_0,f\uh l_0)$  is  splitting node in $T^{\beta}_{\gamma}$; and
   \item $(x\uh l_0,f\uh l_0)$ is the leftmost finite string in $\{(x\uh l_0,\tau)\mid (x\uh l_0,\tau)\in T_{\gamma}^{\beta}\}$; and
   \item There is a number $l_1>l_0$ so that $(x\uh l_1, f\uh l_1)$ is the $2k+1$-th number,  for  some $k$, in $SP_{ x\uh l_0,f\uh l_0}=\{j\mid (x \uh j, f \uh j) \in T_{\gamma}^{\beta}[x\uh l_0,f\uh l_0] \mbox{ is a splitting node}\}$ so that $x\succ x\uh l_1^{\smallfrown}1$; and
   \item Either $L_{\omega_1^{\beta}}\models \forall u<\beta\exists v\in L_{\gamma} \varphi_{i-1}(u,v,\beta)$ or there is some $d\in Dom(\Phi^{x\uh l_1})$ so that $L_{\omega_1^{\beta}}\models \forall v\in L_{\gamma} \neg \varphi_{i-1}(f(d),v,\beta)$; and
   \item $n_j$ is the least number greater than $l_1$ so that $(x\uh n_j, f\uh n_j)$ is a splitting node  in $T_{\gamma}^{\beta}$. 
   \end{itemize}
 \end{itemize}
 \end{definition}
 Intuitively if $\gamma$    matches  $\{n_i\}_{i\leq s}$, then, up to $n_s$, $T_{\gamma}^{\beta}$ is ``quite like $T_{\omega_1^{\beta}}^{\beta}$ along $(x,f)$".
 \bigskip
 
 Now we start to do the decoding construction.
 
 At stage $0$, let $\alpha_0=0$ and $n_0^0=0$. Claim that $0$ is inactive, $i$ is active and $n_i^0$ is undefined for any $i>0$.

At stage $s+1$, then $i_{s}$ is the least $i$ so that $i$ is active. Also, by induction, $n^{s}_{j}$ is defined for any $j<i_{s}$.

  Case(1). There is a number $i'<i_s$ so that $\alpha_{s}+1$ does not match $\{n_j\}_{j\leq i'}$. Let $i_{s+1}$ be the least such $i'$.  Let $\alpha_{s+1}=\alpha_s+1$  and claim $i_{j}$ is active and $n_j^{s+1}$ is undefined for all $j\geq i_{s+1}$. Moreover, set $n_j^{s+1}=n_j^s$ for any $j<i_{s+1}$.   Go to next stage.
  
  Case(2). Otherwise. Search an ordinal $\gamma>\alpha_s$ less than $\omega_1^{\beta}$ and a corresponded unique natural number $n>\max\{n^s_j\mid j<i_s\}$ so that $\gamma$  matches  the finite sequence $\{n^s_j\}_{j< i_s}\cup \{n\}$. If during  the search, we found an ordinal $\gamma'$ so that there is a number $i'<i_s$ so that $\gamma$ does not match $\{n_j\}_{j\leq i'}$. Then do the action as in Case(1). In other words, let $i_{s+1}$ be the least such $i'$.  Let $\alpha_{s+1}=\gamma'$  and claim $j$ is active and $n_j^{s+1}$ is undefined for all $j\geq i_{s+1}$. Moreover, set $n_j^{s+1}=n_j^s$ for any $j<i_{s+1}$.   Go to next stage. Otherwise, by the construction of $(x,f)$, there must be such $\gamma$ and $n$. Find the least such $\gamma$ and the corresponded $n$. Let $\alpha_{s+1}=\gamma$ and $i_{s+1}=i_s+1$.  Claim $j$ is active and $n_j^{s+1}$ is undefined for all $j\geq i_{s+1}$. Moreover, set $n_j^{s+1}=n_j^s$ for any $j<i_{s}$,   $n^{s+1}_{i_{s}}=n$ and claim that $i_s$ is inactive.   Go to next stage.
  
  This finishes the construction at stage $s+1$.
  
  \bigskip
  
 Let $$\theta=\bigcup_{s\in \omega}\alpha_s.$$
 
\begin{lemma}\label{lemma: stable lemma}
For any $i$, there is some $s$ so that for any $t\geq s$, $n^t_i$ is defined, $n^t_i=n^s_i$  and $i$ is inactive at stage $t$ for any $t\geq s$.
\end{lemma}
\begin{proof}
Suppose not. Let $i$ be the largest number ($i$ could be $0$) so that there is some $s$ so that for any $t\geq s$, $n^t_i$ is defined and $n^t_i=n^s_i$ for any $t\geq s$. Then there is an increasing sequence $\{s_j\}_{j\in \omega}$ so that $i_{s_j}=i+1$ and $s_j>s$ for any $j$. Note that, by the construction, at stage $s_j+1$, $n^{s_j}_{i+1}$ is defined and in the tree $T_{\alpha_{s_{j}+1}}^{\beta}[x\uh n^{s_j+1}_i,f\uh n^{s_j+1}_i]$, $(x\uh n^{s_j+1}_{i+1},f\uh n^{s_j+1}_{i+1})$ turns to right at most twice. In other words, there are at most two numbers $l_0<l_1\in (n^{s_j+1}_{i}, n^{s_j+1}_{i+1})$ so that both $(x\uh l_0,f\uh l_0)$ and  $(x\uh l_1,f\uh l_1)$  are splitting nodes in $T_{\alpha_{s_{j}+1}}^{\beta}[x\uh n^{s_j+1}_i,f\uh n^{s_j+1}_i]$ such that  $x\succ x\uh l_0^{\smallfrown}1$ and  $x\succ x\uh l_1^{\smallfrown}1$. Moreover, $l_1$ is the largest number less than $n^{s_j+1}_{i+1}$ so that $(x\uh l_1,f\uh l_1)$ is a splitting node in $T_{\alpha_{s_{j}+1}}^{\beta}[x\uh n^{s_j}_i,f\uh n^{s_j}_i]$. Since  $i+1$ is activated at $s_{j+1}$,  $\alpha_{s_j}$ does not match $\{n^{s_{j}+1}_k\}_{k\leq i+1}$. Then either  $(x\uh l_1,f\uh l_1)$ is not a splitting node  in $T_{\alpha_{s_{j+1}}}^{\beta}[x\uh n^{s_j+1}_{i},f\uh n^{s_j+1}_{i}]$ or there exists some $d\in \Phi^{x\uh n^{s_j+1}_{i+1}}$ such that $L_{\omega_1^{\beta}}\models \forall v\in L_{\alpha_{s_j}}\neg \varphi_{i}(f(d),v,\beta)$ but $L_{\omega_1^{\beta}}\models \exists v\in L_{\alpha_{s_j+1}} \varphi_{i}(f(d),v,\beta)$. In either case, the finite string $(x\uh l_1^{\smallfrown}0,\tau)$ will be cut off from $T_{\alpha_{s_{j+1}}}^{\beta}[x\uh n^{s_j+1}_{i},f\uh n^{s_j+1}_{i}]$. But this happens for every $j$, so it is clear that $(x,f)$ must be the leftmost infinite path in $T_{\theta}^{\beta}[x\uh k,f\uh k]$  for some fixed $k \geq n^{s }_{i}$. Then  $(x,f)$ is the leftmost infinite path in $T_{\omega_1^{\beta}}^{\beta}[x\uh k,f\uh k]$, which contradicts our construction of $(x,f)$ (since $y(i)=1$ for infinitely many $i$'s). 
\end{proof}

\begin{lemma}\label{lemma: every i is satified}
For any $i$, there must be some $s$ and $d\in Dom(\Phi^{x\uh n^s_{i+1}})$ so that for any $t\geq s$, either $L_{\omega_1^{\beta}}\models \forall u<\beta\exists v\in L_{\alpha_s}\varphi_i(u,v,\beta)$ or  $L_{\omega_1^{\beta}}\models \forall v\in L_{\alpha_t}\neg \varphi_i(f(d),v,\beta)$ 
\end{lemma}
\begin{proof}
By Lemma \ref{lemma: stable lemma},  there is some $s$ so that for any $t\geq s$, $n^t_i$ is defined, $n^t_i=n^s_i$  and $i$ is inactive at stage $t$ for any $t\geq s$. Then at stage $s$, either $L_{\omega_1^{\beta}}\models \forall u<\beta\exists v\in L_{\alpha_s}\varphi_i(u,v,\beta)$ or there is some $d\in Dom(\Phi^{x\uh n^s_{i+1}})$ such that  $L_{\omega_1^{\beta}}\models \forall v\in L_{\alpha_s}\neg \varphi_i(f(d),v,\beta)$. If the first case happens, then we finishes the proof. Otherwise, since $i$ is never activated from stage $s$ and $Dom(\Phi^{x\uh n^s_{i+1}})$ is finite, there must be some fixed $d\in Dom(\Phi^{x\uh n^s_{i+1}})$  so that $L_{\omega_1^{\beta}}\models \forall v\in L_{\alpha_t}\neg \varphi_i(f(d),v,\beta)$. 
\end{proof}

So by Lemma \ref{lemma: every i is satified}, $\theta\geq \omega_1^{\beta}$ and so $\theta= \omega_1^{\beta}$. 

This completes the proof of Theorem \ref{theorem: pi11 det}.

{\bf Remark:} This argument can pushed up to prove Friedman's conjecture for $\Delta^1_3$-sets and answer several questions in \cite{Q_theory} for level 3. Then by recent work of Yizheng Zhu, we believe those questions can be answered fully.

 %%%%%%%%%%%%%%%%%%%%%%%%%%%%%%%%%
 \part{Model theory and definability}
 
\section{Descriptions in second order logic}
The following is a result of Hyttinen, Kangas and V\"a\"an\"anen \cite[Thm. 3.3]{Hyttinen.etal:13}. It shows that under the right hypothesis on a cardinal $\kappa$, the models of a  countable complete theory that have  size $\kappa$ can be described in second order logic iff the theory is easy in the sense of Shelah's main gap. 
\begin{thm}[\cite{Hyttinen.etal:13}] Let $T$ be a countable complete theory. For every infinite cardinal $\kappa $ with  $\kappa = \aleph_\alpha$, where $\beth(|\alpha| +\omega_1) < \kappa$ and $\tp \lambda < \tp \kappa$ for each $\lambda < \kappa$, the following are equivalent:

\bi \item[(i)] Every model of $T$  of size $\kappa$ has a description in $L^2_{\kappa, \omega}(T)$.
\item[(ii)] $T$ is superstable, shallow, fails  the dimensional order property DOP and fails the omitting types order property OTOP.
\ei
\end{thm}

Explanations: 
$L^2_{\kappa, \omega}(T)$ is the second-order language over the symbol set of $T$ with disjunctions of size $< \kappa$ and finite strings of quantifiers.

Superstability of $T$ is stronger than stability: no infinite linear order can be defined in a model of $T$ using a formula in $L_{\omega_1, \omega}$ with parameters. Shelah proved that a model of a   theory $T$ without DOP  can be thought of as built from a tree of small models. Shallowness of $T$ means that for each model of $T$, this tree has no infinite path.

\section{Kolezhitskiy:   Robinson's theorem that $\ZZ$ is definable in $\QQ$} Yan Kolezhitskiy and Andr\'e Nies discussed a celebrated result of Julia Robinson as part of a semester reading project. The result was originally obtained as part of her 1948 PhD thesis under the supervision of Alfred Tarski.  It then appeared  in the 1949 J.Symb.\ Logic~\cite{Robinson:49}. Much simpler formulas for defining $\ZZ$ in $\QQ$ have been obtained in subsequent work: a $\Pi_2$ by Poonen, and recently a $\Pi_1$ by Jochen Koenigsmann. If $\ZZ$ was also $\Sigma_1$ definable in $\QQ$ then the existential theory of $\QQ$ would be undecidable, which is an open problem at present. 

\begin{theorem}[\cite{Robinson:49}] The set of integers is definable without parameters in the field  of rationals  $(\QQ, + , \times, 0,1)$. \end{theorem}

\subsection*{Idea and structure of the proof}
A subset $S$ of $\QQ$ is called \emph{inductive} if $0 \in S$ and $y \in S \to y+1 \in S$ for each $y$.  The following is a  monadic \emph{second-order} definition of $\NN$ in $\QQ$:  
\begin{equation} \label{eqn:second order} k \in \NN \LR \forall S [ \,  S \text{ is inductive }  \to \, k \in S ]. \end{equation}
Julia's idea was that it suffices to quantify over a small  collection of sets $S$, which is uniformly parameterised by  pairs of rationals $a,b$. Let
\begin{equation} \label{eqn:phi}	\phi (a,b,k) \equiv  \exists x \exists y \exists z ( 2+abk^2+bz^2 = x^2 +ay^2 ) \end{equation}
She used some number theory to show that the sets $S_{a,b}$ of the form $\{k \colon \, \QQ \models \phi(a,b,k)\}$ suffice.  

We can now  turn the universal second-order quantification over $S$ into a universal quantification over  rationals $a,b$ in order to obtain a first-order definition replacing (\ref{eqn:second order}):

\begin{equation} \label{eqn:first order} k \in \NN \LR \forall a \forall b [ \,  S_{a,b} \text{ is inductive }  \to \, k \in S_{a,b} ]. \end{equation}

Clearly, with a smaller collection of sets  $S$ the implication from left to right in (\ref{eqn:first order}) still holds. The worry is that we don't have enough sets any longer to separate a rational not in $\NN$ from $\NN$.  (We note that Julia actually only manages to separate non-integer rationals from $\NN$; a small complication of  to the idea outlined above will be needed for that. She in effect first defines a set $V$ in between $\NN$ and $\ZZ$, then notes that $\ZZ = V \cup -V$. )
  
    We have to pick the condition $\phi(a,b,k)$ wisely, ensuring that the relevant sets $S_{a,b}$ are inductive. This is done via the following fact. We say that  $k = n/d$ \emph{in its  lowest terms},  if $n \in \ZZ$, $d \in \NN -\{0\}$ and  $(n,d)=1$.   We also say that $d$ is the denominator of $k$ in its lowest terms. 
\begin{fact} Let $S$ be a set of rationals given by  a condition that holds for $0$ and only depends on the denominator of the rational in lowest terms. Then $S$ is inductive. \end{fact}
 The fact 
 is evident   because a rational  $q$ has  the same denominator in lowest terms as $q+1$.

 Next she shows that for two particular kinds of choices for $a,b$, the condition  $\phi (a,b,k) $ in (\ref{eqn:phi}) only depends on the denominator of $k$ in its lowest terms. Firstly, $b$ is a prime $p$ such that  $p \equiv 3 \mod 4$, and $a=1$.
 	\begin{lemma}\label{lem:3}
		Suppose that $p$ is a prime such that  $p \equiv 3 \mod 4$. The equation $2+ pk^2 + pz^2 = x^2+y^2$ has a solution for $x,$ $y$, and $z$ iff the denominator of $k$ in its lowest terms is odd, and is co-prime to $p$.
	\end{lemma}	
	
Secondly, $a$ is a prime $q$ and $b$ is a prime $p$, with some additional  restrictions.	Recall that the Legendre ``symbol'' $(k/p)$   is  a binary function that returns $1$ if $k$ is a quadratic residue $mod p$, $0$ if $k=0$, and $-1$ otherwise.

	\begin{lemma}\label{lem:4}
		Suppose that   $p$ and $q$ are odd primes such that  $p \equiv 1 \mod 4$ and $(q/p) = -1$. The equation $2+ pqk^2 + pz^2 = x^2+ qy^2$ has a solution for $x$, $y$, and $z$ iff the denominator of $k$ in its lowest terms is co-prime to both $q$ and $p$.
	\end{lemma}	
	
	Given these two lemmas, the proof concludes as follows. First one needs a number theoretic claim which provides the $q$ we need in Lemma~\ref{lem:4}. 
\begin{claim}\label{cl:1}
		Suppose $p$ is a prime such that $p \equiv 1 \mod 4$. There exists an odd prime $q$ such that $(q/p)=-1$.
		\end{claim}
		\begin{proof} Let $ s$ be any (quadratic) non-residue  mod $ p$. Then $s+p$ is also a non-residue. One of $s$, $s+p$ is odd, say the former. There is an (odd)   prime factor of $s$ which is also a non-residue, because the Legendre symbol is multiplicative. Let this prime factor be $q$.  \end{proof}
		
		Let  $\psi(k)$ be a first-order formula expressing the  right hand side of (\ref{eqn:first order}). 	Clearly $k \in \NN \RA \psi(k)$. We   verify that  $  \psi(k) \RA k \in \ZZ$.  Suppose that  $k \in \QQ$ and $\psi(k)$ holds. Write     $k= n/d$ in lowest terms. By   Lemma~\ref{lem:3}, $d$ is odd and not divisible by any prime $p$ such that $p \equiv 3 \mod 4$. By   Lemma~\ref{lem:4} using Claim~\ref{cl:1}, $d$ not divisible by any prime $p$ such that $p \equiv 1 \mod 4$. Therefore $d=1$. 
		
	    Thus, $k \in \NN \RA      \psi(k) \RA k \in \ZZ$, so  the formula $\psi (k) \vee \psi(-k)$ provides  a first-order definition of $\ZZ$ in $\QQ$.

		%%%%%%%%%%%%%
		\subsection*{Proofs of Lemmas~\ref{lem:3} and \ref{lem:4}}
Julia relies on two claims that follow from the Hasse-Minkowski theorem.

\begin{claim} \label{lem:1}
		
		Suppose  $p$ is a prime such that  $p \equiv 3 \mod 4$. We have  $x^2 + y^2 - pz^2=m$ for some $m \in \mathbb{Q},$ $m\neq0$, iff it is not the case that 
		
		\indent\indent a) $m = pks^2$, where $(k/p) = 1$, or \\
		\indent\indent b) $m = ks^2$ with $k \equiv p$ $(mod$ $8)$\\\\
		
	\end{claim}	
	
	\begin{claim}\label{lem:2}
		Suppose  $p$ and $q$ are odd primes such that  $p \equiv 1 \mod  4$ and $(q/p) = -1$. There is some non-zero $m\in \QQ$ such that $x^2 + qy^2 - pz^2=m$, iff it is not the case that:\\
		\indent\indent a) $m = pks^2$ and $(k/p) = -1$\\
		\indent\indent b)  $m = qks^2$ and $(k/q) = -1$\\\\
	\end{claim}

\def\cprime{$'$} \def\cprime{$'$}

%\def\cprime{$'$} \def\cprime{$'$}
%
%\bibliographystyle{plain}
%
%\bibliography{../bibs/Nies,../bibs/randomness,../bibs/settheory,../bibs/various,../bibs/recursiontheory,../bibs/analysis,../bibs/Kucera,../bibs/modeltheory,../bibs/reverse_maths,../bibs/groups,../bibs/ergodic_theory,../bibs/computer_science}

\end{document}